\documentclass[10pt,reqno,twosdie]{amsart}

\usepackage{amsmath,amssymb,mathrsfs}

\usepackage{tikz}
\usepackage[asymmetric,top=3.2cm,bottom=3.9cm,left=2cm,right=2cm]{geometry}
\usepackage{amsmath,amsfonts,amsthm,mathrsfs,amssymb,cite}
\usepackage{color}

\usepackage{mathtools} 
\usepackage{graphics}
\usepackage{framed}
\usepackage{enumerate}

\usepackage{hyperref}
\usepackage{xcolor}
\hypersetup{
  colorlinks,
  linkcolor={blue!80!black},
  urlcolor={blue!80!black},
  citecolor={blue!80!black}
}
\usepackage[perpage]{footmisc}
\usepackage{float}

\usepackage{graphicx}
\usepackage{latexsym} 
\usepackage{enumerate}
\usepackage{caption}
\usepackage{subcaption}
\theoremstyle{plain}
\newtheorem{theorem}{Theorem}[section]

\newtheorem{proposition}[theorem]{Proposition}
\newtheorem{definition}[theorem]{Definition}

\newtheorem{remark}[theorem]{Remark}

\newtheorem*{assumptions}{Assumptions}
\renewcommand{\eqref}[1]{\textnormal{(\ref{#1})}}

\numberwithin{equation}{section}

\usepackage{color}

\newcommand{\U}{\mathbb{U}}
\newcommand{\D}{\mathbb{D}}
\newcommand{\C}{\mathbb{C}}
\newcommand{\hh}{\mathbb{H}}

\def\Oh{{\mathcal  O}}

\usepackage{booktabs}
  
\newcommand{\diam}{\mathrm{diam}}      
\renewcommand{\div}{\mathrm{div}\,} 

\title[Non periodic distributions]{Effective Medium Theory For\\  Van-der-Waals heterostructures
 }

\author[Cao, Ghandriche and Sini]{Xinlin Cao $^*$ Ahcene Ghandriche  $^{**}$ and Mourad Sini$^{\ddag}$}
\thanks{$^*$ Department of Applied Mathematics, The Hong Kong Polytechnic University. Email: xinlin.cao@polyu.edu.hk}
\thanks{$^{**}$ Nanjing Center for Applied Mathematics, Nanjing, 211135, People’s Republic of China. Email: gh.hsen@njcam.org.cn}
\thanks{$^{\ddag}$ RICAM, Austrian Academy of Sciences, Altenbergerstrasse 69, A-4040, Linz, Austria. Email: mourad.sini@oeaw.ac.at. This author is partially supported by the Austrian Science Fund (FWF): P 32660}

\date{\today}

\begin{document}

\begin{abstract}
We derive the electromagnetic medium equivalent to a collection of all-dielectric nano-particles (enjoying high refractive indices) distributed locally non-periodically in a smooth domain $\Omega$. Such distributions are used to model well known structures in material sciences as the Van-der-Waals heterostructures. Since the nano-particles are all-dielectric, then the permittivity remains unchanged while the permeability is altered by this effective medium. This permeability is given in terms of three parameters. The first one is the polarization tensors of the used nano-particles. The second is the averaged Magnetization matrix $\vert \Omega_0\vert^{-1} \, \int_{\Omega_0} \underset{x}{\nabla} \int_{\Omega_0} \underset{y}{\nabla} \Phi_0(x,y) \cdot I_3 \, dy \, dx$, where $\Phi_0(x, y) \, := \, \dfrac{1}{4 \, \pi \, \vert x-y\vert}$, $I_3$ is the identity matrix and $\Omega_0$ is the unit cell. The third one is $\nabla \nabla \Phi_0 (z_i, z_j)$, where $z_i$'s are locations of the local nano-particles distributed in the unit cell. This last tensor models the local strong interaction of the nano-particles. To our best knowledge, such tensors are new in both the mathematical and engineering oriented literature. This equivalent medium describes, in particular, the effective medium of $2$ dimensional type Van-der-Waals heterostructures. These structures are $3$ dimensional which are build as superposition of identical ($2$D)-sheets each supporting locally non-periodic distributions of nano-particles. An explicit form of this effective medium is provided for the particular case of honeycomb heterostructures.
\newline

At the mathematical analysis level, we propose a new approach to derive the effective medium when the subwavelength nano-particles are distributed non-periodically. The first step consists in deriving the point-interaction approximation, also called the Foldy-Lax approximation.
The scattered field is given as a superposition of dipoles (or poles for other models) multiplied by the elements of a vector which is itself solution of an algebraic system. This step is done regardless of the way how the particles are distributed. As a second step, which is the new and critical step, we rewrite this algebraic system according to the way how these nano-particles are locally distributed. The new algebraic system will then fix the related continuous Lippmann Schwinger system which, in its turn, indicates naturally the equivalent medium.

\medskip 

 

\end{abstract}

\maketitle

\section{Introduction}\label{sec:Intro}
\subsection{Background and motivation}

During the last years, there has been high interest in the analysis and computational aspects of wave propagation in
the presence of subwavelength resonators consisting of small-scaled inhomogeneities enjoying
high contrast or sign changing constitutive material properties as compared to natural media, see \cite{Osipov-Tretyakov, Pan:2000}. A key aspect of such subwavelength materials, is that under certain critical scales of their size/contrast, such inhomogeneities can resonate at specific frequencies which allow them to behave as resonators
that amplify locally the used incident waves. Such amplifications have tremendous applications in
applied fields related to nanotechnology. Of particular interest are the materials and devices built up as innovative assemblages of 3D-type or 2D-type building blocks which are distributed periodically in a given bounded domain.  These building blocks are themselves designed using subwavelength nano-resonators which are distributed non-necessarily periodically (in the building blocks). This makes the whole structure globally periodic, at the building blocks scales, but locally non periodic inside each block. In this work,
we focus on resonators described with dispersive nano-particles. The dispersive character allows
these nano-particles to enjoy the needed critical scale/contrast in specific bands of frequencies. In particular, we deal with all-dielectric nano-particles that enjoy high indices of refraction.
\bigskip

 The ultimate goals are first to derive the expansions of the electric fields in terms of the cluster of the nano-particles described above. As a second goal, apply them to the generation of fully 3D or {\it{2D-like materials}} enjoying desired properties and functionalities as van der Waals heterostructures. To give a fortaste on the importance of such material, let us first briefly describe them, see also \cite{van-der-Waals-paper}. The van der Waals heterostructures are built as superposition of 2D-sheets, of subwavelength nanoparticles, which are globally periodic but locally formed of heterogeneous clusters of nanoparticles. Example of such distributions as trigonal prismatic, octahidral, chalcogenides, Boron nitrids, etc., see \cite{van-der-Waals-paper, van-der-Waals-book}. They offer many functional possibilities. The control of the electromagnetic wave propagation in such structures (which are periodic globally and heterogeneous locally) can be handled using our approach. Therefore, if we succeed in characterizing the effective material behind such structures, which captures the behavior of the local heterogeneity, then we can tune such structures to reach desired functionalities.  
\bigskip

As mentioned briefly above, at the mathematical analysis level, our approach to derive the effective medium when the subwavelenght particles are distributed non-periodically is divided into two steps. The first step consists in deriving the point-interaction approximation, also called the Foldy-Lax approximation.
The scattered field is given as a superposition of dipoles (or poles for other models) multiplied by the elements of a vector which is itself solution of an algebraic system. This step is done regardless of the way how the particles are distributed (in particular, periodicity is not needed), see \cite{CGS}. As a second step, which is the new and key step, we rewrite this algebraic system according to the way how these particles are locally distributed (in the building blocks). The form of the new algebraic system will then fix the related continuous Lippmann Schwinger system which, in its turn, indicates naturally the equivalent medium.

\subsection{The electromagnetic fields created by all-dielectric nano-particles.}\label{subsec1}
We deal with the time-harmonic electromagnetic scattering by nano-particles based on the following Maxwell-related model
\begin{equation}\label{model-m}
	\begin{cases}
	\mathrm{Curl} \, E^T-i k \mu_rH^T=0\quad\mbox{in}\ \mathbb{R}^3,\\	
	\mathrm{Curl} \, H^T+i k \epsilon_r E^T=0\quad\mbox{in}\ \mathbb{R}^3,\\
	E^T=E^{in}+E^s,\quad H^T=H^{in}+H^s,\\
	\sqrt{\mu_0\epsilon_0^{-1}}H^s\times\frac{x}{|x|}-E^s=\Oh({\frac{1}{|x|^2}}),\quad \mbox{as} \ |x|\rightarrow\infty,
	\end{cases}
\end{equation}
\medskip
\phantom{}
\newline
with $\epsilon_0$ and $\mu_0$ are respectively the electric permittivity and the magnetic permeability of the vacuum outside $D$, where we denote by $D \, := \, \overset{M}{\underset{m=1}{\cup}} D_{m}$ a collection of $M$ connected, bounded and $C^{1}$-smooth nano-particles of $\mathbb{R}^{3}$. Furthermore, $\epsilon_{r} :\, = \, \dfrac{\epsilon}{\epsilon_{0}}$ and $\mu_{r} :\, = \, \dfrac{\mu}{\mu_{0}}$ are the relative permittivity and permeability satisfying $\epsilon_{r} = 1$ outside $D$ while $\mu_{r} = 1$ in the whole space $\mathbb{R}^{3}$. Without loos of generality, we choose as incident waves the plane waves
\begin{equation}\label{EincHinc}
    E^{Inc}(x) \, := \, \mathrm{p} \, e^{i \, k \, \theta \cdot x} \quad \text{and} \quad     H^{Inc}(x) \, := \, \left( \theta \times \mathrm{p} \right) \, e^{i \, k \, \theta \cdot x},
\end{equation}
satisfying 
\begin{equation*}
    \begin{cases}
	   Curl\left(E^{Inc}\right) \, - \, i \, k \, H^{Inc} & \text{=} 0, \quad \text{in} \quad \mathbb{R}^{3}, \\
    & \\
         Curl\left(H^{Inc}\right) \, + \, i \, k \, E^{Inc} & \text{=} 0, \quad \text{in} \quad \mathbb{R}^{3},
		 \end{cases},
\end{equation*}
\medskip
\phantom{}
\newline
with $\theta$ and  $\mathrm{p}$ in $\mathbb{S}^2$, $\mathbb{S}^2$ being the unit sphere, such that $\theta \cdot \mathrm{p} = 0$, as the direction of incidence and polarization respectively. The problem $(\ref{model-m})$ is well-posed in appropriate Sobolev spaces, see \cite{colton2019inverse, Mitrea}, and the behaviors that will come are being possessed by the scattered wave
\begin{eqnarray*}
    E^{s}(\hat{x}) \, &=& \, \dfrac{e^{i \, k \, \left\vert x \right\vert}}{\left\vert x \right\vert} \, \left( E^{\infty}\left( \hat{x}, \theta , \mathrm{p} \right) \, + \, \mathcal{O}\left( \left\vert x \right\vert^{-1} \right) \right), \quad \text{as} \quad \left\vert x \right\vert \rightarrow + \infty, \\
    H^{s}(\hat{x}) \, &=& \, \dfrac{e^{i \, k \, \left\vert x \right\vert}}{\left\vert x \right\vert} \, \left( H^{\infty}\left( \hat{x}, \theta , \mathrm{p} \right) \, + \, \mathcal{O}\left( \left\vert x \right\vert^{-1} \right) \right), \quad \text{as} \quad \left\vert x \right\vert \rightarrow + \infty,
\end{eqnarray*}
where $\left(E^{\infty}\left( \hat{x}, \theta , \mathrm{p} \right), H^{\infty}\left( \hat{x}, \theta , \mathrm{p} \right) \right)$ is the corresponding electromagnetic far-field pattern of $(\ref{model-m})$ in the propagation direction $\hat{x}$. In addition, we suppose that 
\begin{equation}\label{DaBz}
    D_{m_{\ell}} \, = \, a \, B_{m_{\ell}} \, + \, z_{m_{\ell}}, \quad m = 1, \cdots, \aleph \quad \text{and} \quad \ell = 1, \cdots, K+1, 
\end{equation}
where $D_{m_{\ell}}$ are the $\aleph \times (K+1)$ connected small components of the medium $D$, which are characterized by the parameter $a$ and the locations $z_{m_{\ell}}$. Each $B_{m_{\ell}}$ containing the origin is a bounded Lipschitz domain and fulfills that $B_{m_{\ell}} \subset B(0,1)$, where $B(0,1)$ is the unit ball. The parameter $a$ is defined by 
\begin{equation*}
    a \, := \, \underset{1 \leq m \leq \aleph \atop 1 \leq \ell \leq K+1}{\max} \text{diam}\left( D_{m_{\ell}} \right),
\end{equation*}
and we denote by $d$ the minimal distance between the distributed nano-particles, i.e.
\begin{equation}\label{dmin}
  d \;  =  \; \underset{1 \leq m \leq \aleph}{\min}  \; \underset{1 \leq \ell , j \leq K+1 \atop \ell \neq j}{\min}  \; \text{dist}\left(   D_{m_\ell}\,, \, D_{m_{j}}  \right).
\end{equation}
The two parameters $a$ and $d$ are linked to each other through the following formula 
\begin{equation}\label{distribute}
d \, = \, c_r \, a^{1-\frac{h}{3}} \quad \text{and} \quad 
\left\vert \Omega_m \right\vert \, = \,  d^3 \quad \mbox{for any} \quad m=1, \cdots, \aleph,
\end{equation}
where $c_r$ is the dilution parameter, which is independent of the parameter $a$, and is of order one. The subdomain $\Omega_{m}$ is a shift from the unit cell $\Omega_{0}$, which will be explained in greater detail later, see for instance Assumption $(\ref{assII})$.    
\bigskip
\phantom{}
\newline 
On the basis of the Foldy-Lax approximation presented in \cite{CGS}, we investigate the corresponding effective permittivity and permeability generated by the cluster of nano-particles with high relative permittivity contrast under the following assumptions.
\medskip
\newline
\begin{assumptions}
\smallskip
\phantom{}
\begin{enumerate}
    \item[]
    \item \textit{Assumption I), Assumption II) and Assumption III)}, used in \cite[Page 4-5]{cao2023all}. More precisely.
    \begin{enumerate}
    \item[]
    \item Assumption on the shape of $B_{m_{\ell}}$\label{assIa}. 
Assume that the shapes of $B_{m_{\ell}}$'s introduced in \eqref{DaBz} are the same and we denote
\begin{equation*}\label{def-B}
	B \, := \, B_{m_{\ell}} \quad \mbox{for} \quad m=1,\cdots, \aleph \quad \text{and} \quad \ell = 1, \cdots, K+1.
\end{equation*}
The domain $B$ is a bounded and $C^1$-smooth domain that contains the origin. Recall the Helmholtz decomposition of $\mathbb{L}^2:=\mathbb{L}^2(B)$ space, for any bounded domain $B$,  
\begin{equation}\label{hel-decomp}
\mathbb{L}^2=\mathbb{H}_0(\div=0)\overset{\perp}{\otimes}\mathbb{H}_0(Curl=0)\overset{\perp}{\otimes}\nabla \mathcal{H}armonic.
\end{equation} 
Then, the projection of an arbitrary vector field $F$ onto three sub-spaces $\mathbb{H}_0(\div=0), \; \mathbb{H}_0(Curl=0)$ and $\nabla \mathcal{H}armonic$ can be respectively represented as $\overset{1}{\mathbb{P}}\left( F \right), \overset{2}{\mathbb{P}}\left( F \right)$ and $\overset{3}{\mathbb{P}}\left( F \right)$. We define the vector Newtonian operator ${\bf N}_B$, for any vector function $F$, as
\begin{equation*}\label{def-new-operator}
{\bf N}_{B}(F)(x) \, := \, \int_B \Phi_0(x, z) \, F(z)\,dz,\quad\mbox{where}\quad \Phi_0(x, z)=\frac{1}{4\pi}\frac{1}{|x-z|}\quad (x\neq z).
\end{equation*}
For the Newtonian operator, under the previous decomposition, we denote $\left(\lambda_{n}^{(2)}(B), e_{n}^{(2)} \right)_{n \in \mathbb{N}}$ as the related eigen-system (of its projection) over the subspace $\mathbb{H}_{0}\left( Curl = 0 \right)$, and we denote $\left(\lambda_{n}^{(1)}(B), e_{n}^{(1)} \right)_{n \in \mathbb{N}}$ as the related eigen-system (of its projection) over the subspace $\mathbb{H}_{0}(\div=0)$. Since 
\begin{equation*}
\mathbb{H}_{0}\left( \div = 0 \right) \equiv Curl \left( \mathbb{H}_{0}\left( Curl \right) \cap \mathbb{H}\left( \div = 0 \right) \right),    
\end{equation*}
see for instance \cite{{amrouche1998vector}}, then for any $n$, there exists $\phi_n\in \mathbb{H}_0(Curl)\cap \mathbb{H}(\div=0)$, such that 
\begin{equation}\label{pre-cond}
e_{n}^{(1)}= Curl (\phi_{n}) \,\, \mbox{ with } \,\, \nu\times \phi_{n}=0 \,\, \text{and} \,\, \div(\phi_{n})=0.
\end{equation}	
We assume that, for a certain $n_0$, 
\begin{equation*}\label{Intphin0}
\int_{B} \phi_{n_{0}}(y) \, dy \, \neq \, 0,
\end{equation*}
where $\phi_{n_0}$ fulfills \eqref{pre-cond} with $n=n_0$. Then we can define a constant tensor ${\bf P_0}$ by
\begin{equation}\label{defP0}
{\bf{P}}_0 := \sum_{\ell=1}^{\ell_{\lambda_{n_{0}}}} \int_{B}\phi_{n_{0}, \ell}(y)\,dy \otimes  \int_{B}\phi_{n_{0}, \ell}(y)\,dy, \quad \ell_{\lambda_{n_{0}}} \in \mathbb{N}^{\star},
\end{equation}
where $\phi_{n_{0}, \ell}$ fulfills
\begin{equation}\notag
e^{(1)}_{n_0, \ell}=Curl(\phi_{n_{0}, \ell}),\, \div(\phi_{n_{0}, \ell})=0, \, \nu\times \phi_{n_{0}, \ell}=0,\mbox{ and }
N(e^{(1)}_{n_0, \ell})=\lambda_{n_0}^{(1)}(B) e^{(1)}_{n_0, \ell}.
\end{equation}  
    \item[]
    \item Assumptions on the permittivity and permeability of each particle.\label{assIb} In order to investigate the electromagnetic scattering by all-dielectric nano-particles with high contrast electric permittivity parameter, we assume that for a positive constant $\eta_0$, independent of $a$, 
\begin{equation*}\label{contrast-epsilon}
\eta:=\epsilon_{r} - 1=\eta_0 \; a^{-2},\; \mbox{ with } \;\; \; ~~ a \ll 1,
\end{equation*}
and the relative magnetic permeability $\mu_r$ to be moderate, namely $\mu_r=1$. 
    \item[]
    \item Assumption on the used incident frequency $k$.\label{assIc} There exists a positive constant $c_0$, independent of $a$, such that
\begin{equation*}\label{condition-on-k}
1 \, - \, k^2 \, \eta \, a^2 \, \lambda_{n_{0}}^{(1)}(B) \, = \, \pm \; c_0\; a^h,\;\; a \ll 1,
\end{equation*}
where $\lambda_{n_{0}}^{(1)}(B)$ is the eigenvalue corresponding to $e_{n_0}^{(1)}$ in $(\ref{pre-cond})$ with $n=n_0$.
    \item[]
\end{enumerate}
    \item[]
    \item \textit{Assumptions on the distribution}\label{assII}. 
Let $\Omega$ be a bounded domain of unit volume, containing the particles $D_{n_{\ell}}$, with $n=1, \cdots, \aleph$ and $\ell = 1, \cdots, K+1$; $K \in \mathbb{N}$. We divide $\Omega$ into $\aleph:=[d^{-3}]$ { periodically distributed} subdomains $\Omega_m$ such that $|\Omega_m|=\mathcal{O}(d^3)$, $m=1, 2, \cdots, [d^{-3}]$, i.e. $D \, := \, \overset{\aleph}{\underset{m=1}\cup} \, \overset{K+1}{\underset{\ell=1}\cup} D_{m_{\ell}}$ (and then $M=\aleph \times (K+1)$). For each subdomain $\Omega_m$, we further divide $\Omega_{m}$ into $K+1$ sub-subdomains $\Omega_{m_\ell}$, $\ell=1,2,\cdots, K+1$, $K\in \mathbb{N}$, such that $z_{m_\ell}\in D_{m_\ell}\subset\Omega_{m_\ell}\subset \Omega_{m}$ with $z_{m_\ell}$ centering at $\Omega_{m_\ell}$. In particular, let $z_{m_1}\in \Omega_{m_1}\subset\Omega_{m}$ center at $\Omega_{m}$, $m=1, 2, \cdots, [d^{-3}]$. The local distribution of the sub-subdomains $\Omega_{m_\ell}\subset \Omega_{m}$ is unified in terms of the subdomains, i.e. independent on $m$. Therefore, we denote $\Omega_{0}$ as the reference set, and the distribution of $\Omega_{m}$'s is constructed by appropriate translations of $\Omega_{0}$, which indicates that $\Omega_{m}$ are distributed periodically in a global view, see \textbf{Figure \ref{fig:local}} and \textbf{Figure \ref{fig:global}} for the schematic illustrations of the local and global distributions of the cluster of particles, respectively. 
\medskip
\phantom{}
\begin{figure}[htbp]
	\centering
	\begin{minipage}[t]{0.48\textwidth}
		\centering
		\includegraphics[width=3cm]{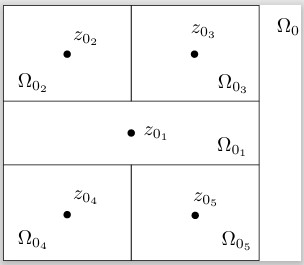}
		\caption{Schematic illustration of the local distribution of the particles in $\Omega_{0}$.}
		\label{fig:local}
	\end{minipage}
	\begin{minipage}[t]{0.48\textwidth}
		\centering
		\includegraphics[width=5cm]{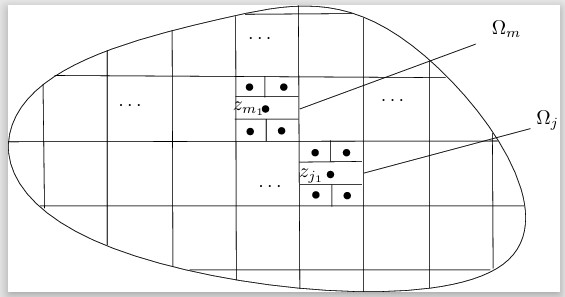}
		\caption{Schematic illustration of the global distribution of the particles.}
		\label{fig:global}
	\end{minipage}
 \caption{A visual representation of the distribution of nano-particles globally and locally.}
 \label{Fig3}
\end{figure}
\end{enumerate}  
The schematizations given by \textbf{Figure \ref{Fig3}} show the globally periodic distribution and the locally non-periodic distribution.
\end{assumptions}
\medskip
\begin{remark} 
The remarks that follow are in order.
\begin{enumerate}
    \item[]
    \item Only the Assumptions $(\ref{assIa}), (\ref{assIb})$ and $(\ref{assIc})$  are needed to derive the Foldy-Lax approximation, see \textbf{Proposition \ref{prop-discrete}} and its proof given in \cite{CGS}. The Assumption $(\ref{assII})$
    is  necessary to explain the effective medium approximation, see $(\ref{KTMB})$.  
    \item[]
    \item There are several shapes that can be considered when cutting $\Omega$ into $\overset{\aleph}{\underset{m=1}\cup} \Omega_{m}$, see Assumptions $(\ref{assII})$. Examples of shapes that can be mentioned include cubes, pyramids, rectangular parallelepipeds and hexagonal prisms. It will be revealed later how $\Omega$'s cut affects the derived formulas, see $(\ref{DefT1})$ and \textbf{Remark \ref{RemarkCut}}. 
    \item[] 
\end{enumerate}
\end{remark}
\medskip
To write short formulas in the text, we introduce the following parameter
\begin{equation}\label{def-xi}
\xi \, := \, \frac{\eta_0 \; k^2}{c_0 \; c_r^{3}}.
\end{equation}
Likewise, in case there are no ambiguities, we use the notation 
\begin{equation*}
    z_{m} \; := \; z_{m_{1}}, \quad \text{for} \quad 1 \leq m \leq [d^{-3}], 
\end{equation*}
to write technical steps without any confusion.
\medskip
\phantom{}
\newline
\begin{definition}\label{DefM}
In the text, for each subdomain $\Omega_{m}$, we refer to a local distribution tensor called $\mathcal{A}$ and describe it as
    \begin{equation}\label{SKM}
        \mathcal{A} \, := \, \left(I_{3}, \cdots , I_{3} \right) \cdot \mathbb{T}_{1}^{-1} \cdot \begin{pmatrix}
            I_{3} \\
            \vdots \\
            I_{3}
        \end{pmatrix},
    \end{equation}
    where $\mathbb{T}_1$ is the tensor given by\footnote{The dimension of each matrix $\left\{ \Lambda_{i,j} \right\}_{i,j=1}^{K+1}$ in the definition below is $3 \times 3$.}
    \begin{eqnarray*}
        \mathbb{T}_1 \, := \, \begin{pmatrix}
            \Lambda_{1,1} & \cdots & \Lambda_{1,K+1} \\
            \vdots & \ddots & \vdots \\
            \Lambda_{K+1,1} & \cdots & \Lambda_{K+1,K+1}
        \end{pmatrix} ,
    \end{eqnarray*}
    with 
    \begin{equation*}
        \Lambda_{i,i} \, := \,
			I_{3} \, - \, \pm  \, \xi \, \dfrac{1}{\left\vert \Omega_{0} \right\vert} \int_{\Omega_{0}} \nabla {\bf M}_{\Omega_{0}}\left( I_{3} \right)(x) \, dx \cdot {\bf P_0}
    \end{equation*}
    and 
    \begin{equation*}
        \Lambda_{i,j} \, := \, 
                \mp  \, \xi \, d^{3} \, \Upsilon_{0}(z_{0_i}, z_{0_j})\cdot{\bf P_0} \, - \, \pm  \, \xi \, \dfrac{1}{\left\vert \Omega_{0} \right\vert} \, \int_{\Omega_{0}} \nabla {\bf M}_{\Omega_{0}}\left( I_{3} \right)(x) \, dx \cdot {\bf P_0}, \quad  \text{for} \quad i \neq j,
    \end{equation*}
 where the constant tensor ${\bf P_0}$ is given by $(\ref{defP0})$, the kernel $\Upsilon_{0}(x,y) \, = \, \nabla \nabla  \Phi_{0}(x,y)$, with $x \neq y$, and $\nabla {\bf M}_{\Omega_{0}}$ is the Magnetization operator
 \begin{equation}\label{DefMagnetization}
      \nabla {\bf M}_{\Omega_{0}}\left( E \right)(x) := \underset{x}{\nabla} \, \int_{\Omega_{0}} \underset{y}{\nabla} \Phi_{0}(x,y) \cdot E(y) \, dy, \quad x \in \Omega_{0}.
 \end{equation}
In addition, we set 
    \begin{equation*}
        \left(\widetilde{\mathcal{A}}_{1}, \cdots, \widetilde{\mathcal{A}}_{K+1} \right) \, := \, \left(I_{3}, \cdots, I_{3} \right) \cdot \mathbb{T}_{1}^{-1}.
    \end{equation*}
\noindent From the above relation and (\ref{SKM}), it is clear that  
\begin{equation*}\label{tildeA-A}
\mathcal{A} \, = \, \overset{K+1}{\underset{\ell=1}\sum} \, \widetilde{\mathcal{A}}_{\ell}.
\end{equation*}
\end{definition}
In its definition, we observe that the tensor (\ref{SKM}) depends on both the geometry of each nano-particle and their local interactions in each subdomain $\Omega_m$. Indeed, the $D_{m_{\ell}}$'s shapes affect the formula throughout the tensor ${\bf{P}}_0$, see $(\ref{DaBz})$ and $(\ref{defP0})$. This tensor is not connected to the $D_{m_{\ell}}$'s interactions. However, knowing that each subdomain $\Omega_m$ contains $(K+1)$ nano-particles, the definition provided here describes the interactions between the nano-particles $D_{m_{\ell}}$ in each $\Omega_m$ throughout the tensors $\Upsilon_{0}(z_{0_{i}}, z_{0_j})$, with $i, j = 1, \cdots, K+1$, see the second term appearing in the definition of $\mathbb{T}_1$. In \textbf{section \ref{GIEGAS}}, we provide more details on the construction of the local distribution tensor $\mathcal{A}$, and we show that it is independent on the index $m$. In addition, in \textbf{Section \ref{Appendix}}, for the case of two closely distributed nano-particles, we explicitly compute the expression of the tensor $\mathcal{A}$.
\bigskip
\newline
We start with the proposition below.
\begin{proposition}\cite[Theorem 1.3]{CGS}\label{prop-discrete}
	Consider the electromagnetic scattering problem \eqref{model-m} generated by a cluster of nano-particles $D_1, \cdots, D_M$. Under \textbf{Assumptions} $(1)$, for $\frac{9}{11} \, < \, h \, < \, 1$, the {electric} far-field of the scattered wave possesses the expansion as
	\begin{equation}\label{approximation-E}
	E^\infty(\hat{x}) \, = \,  - \frac{i\, k^3\, \eta}{4\, \pi}\,  \sum_{n=1}^{M} \, e^{-i \, k \, \hat{x} \cdot z_{n}} \hat{x} \times Q_n+\mathcal{O}\left(a^{\frac{h}{3}}\right),
	\end{equation}
	{uniformly in all directions $\hat{x} := \frac{x}{\left\vert x \right\vert} \, \in \, \mathbb{S}^2$},
	where 
	$\left( Q_n \right)_{n=1,\cdots, M}$ is the vector solution to the following algebraic system
	\begin{equation}\label{linear-discrete}
	Q_{n}-\frac{\eta \, k^2}{\pm c_0} \, a^{5-h} \, \sum_{s=1 \atop s \neq n}^{M} {\bf{P}}_0 \cdot \Upsilon_k(z_{n}, z_s) \cdot {Q}_s
	=\frac{i \, k}{\pm c_0} \, a^{5-h} \; {\bf{P}}_0 \cdot  H^{Inc}(z_{n}),
	\end{equation}
	where $\Upsilon_k(\cdot,\cdot)$ is the dyadic Green's function given by
	\begin{equation*}\label{dyadic-Green}
	\Upsilon_k(x,z) \, := \, \underset{z}{\nabla}\underset{z}{\nabla}\Phi_k(x,z) \, + \, k^2 \, \Phi_k(x,z) \,  I_{3},\quad x\neq z,
	\end{equation*}
	with 
 \begin{equation}\label{fundsolheleq}
  \Phi_k(x, z) \; := \; \dfrac{e^{i \, k \, \left\vert x - z \right\vert}}{4 \, \pi \, \left\vert x - z \right\vert}, \qquad \text{for} \qquad x \neq z,    
 \end{equation}
  being the fundamental solution to the Helmholtz equation, and ${\bf{P}}_0$ is the polarization tensor defined by $(\ref{defP0})$. In particular, \eqref{linear-discrete} is invertible under the condition that
	\begin{equation*}
	\frac{ k^2 |\eta| a^5}{d^3 \left\vert 1 - k^2 \, \eta \, a^2 \, \lambda_{n_{0}}^{(1)}(B) \right\vert} \left\vert {\bf{P}}_0 \right\vert \, < \, 1.
	\end{equation*}
\end{proposition}
\subsection{The effective medium and the generated fields}
Based on \textbf{Assumptions} and the Foldy-Lax approximation presented in \textbf{Proposition \ref{prop-discrete}}, we will investigate the effective medium generated by a globally periodic and locally non-periodic distributed all-dielectric nano-particles. We demonstrate that the effective medium is merely a modification of permeability, via a tensor that encodes the local non-periodic distribution. We refer to $\mathring{\epsilon}_{r}$ and $\mathring{\mu_{r}}$ as the effective relative permittivity and permeability, respectively.

\begin{proposition}\label{Propomu+}
The effective permeability $\mathring \mu_r$ possesses a positive sign under the condition 
    \begin{equation}\label{Cdt+mu}
        \left\Vert {\bf{P}}_0 \right\Vert_{\ell^{2}} \; \left\Vert \mathcal{A} \right\Vert_{\ell^{2}} \; \leq \; \dfrac{1}{3 \, \xi},
    \end{equation}
    where the tensor $\mathcal{A}$ is defined by $(\ref{SKM})$ and ${\bf{P}}_0$ is the tensor defined by $(\ref{defP0})$. \medskip
Furthermore, under the condition 
\begin{equation}\label{better-condition}
       \left\Vert {\bf P_0} \right\Vert_{\ell^{2}} \;  \left\Vert \mathcal{A} \right\Vert_{\ell^{2}} \, < \, \frac{\pi}{3 \, \xi \, \sqrt{6} \, \left( \pi \, + \, k \, \left\vert \Omega \right\vert \right) },
\end{equation}
 the magnetic field $H^{\mathring{\mu}_r}$ admits the $C^{0,\alpha}(\overline{\Omega})$-regularity, for any $\alpha \in (0, 1)$. 
\end{proposition}
\begin{proof}
A straightforward computation using the expression of $\mathring \mu_r$ given by $(\ref{new-coeff})$ allows us to derive $(\ref{Cdt+mu})$. Regarding the $C^{0,\alpha}(\overline{\Omega})$-regularity of the magnetic field, under certain conditions, it was already derived in \cite[Proposition 2.1]{cao2023all}. The condition (\ref{better-condition}) is more general (i.e. more relax) than the one imposed in \cite[Proposition 2.1]{cao2023all}. We refer to  
\textbf{Section \ref{PositiveOperator}}, for more details on the derivation of this result under the condition $(\ref{better-condition})$. 
\end{proof}
\begin{theorem}\label{MainThm}
Let $\Omega$ be a bounded domain with $ C^{1,\alpha}$-regularity, for $\alpha \in (0,1)$. Then under \textbf{Assumptions}, for positive definite, see for instance \textbf{Proposition \ref{Propomu+}}, there holds the following expansion for the far-field, 
\begin{eqnarray}\label{Eq113}
\nonumber
       \left\Vert E^{\infty}_{eff, +}  -          E^{\infty} \right\Vert_{\mathbb{L}^{\infty}(\mathbb{S}^2)} 
       &=& \mathcal{O}\left(  k \, \left\vert {\bf P_0} \right\vert^{2}  \, \left\vert \mathbb{T}_{1}^{-1} \right\vert \, \left( d^{\frac{12}{7}}  \, \left\Vert H^{\mathring{\mu_{r}}} \right\Vert^{2}_{\mathbb{L}^{\infty}(\Omega)} \,  + d^{2\alpha}  \, \left[ H^{\mathring{\mu_{r}}} \right]^{2}_{C^{0,\alpha}(\overline{\Omega})} \,   \right)^{\frac{1}{2}} \right) \\ &+&  \mathcal{O}\left( a^{\frac{h}{3}}\right)   +  \mathcal{O}\left(k \, d^{\alpha} \, \left\vert {\bf P_0} \right\vert \, \left\vert \mathbb{T}_{1}^{-1} \right\vert \, \left[ H^{\mathring{\mu_{r}}} \right]_{C^{0,\alpha}(\overline{\Omega})} 
  \right),
\end{eqnarray}
uniformly in all directions $\hat{x} \in \mathbb{S}^2$, where $E^{\infty}_{eff, +}(\cdot)$ is the far-field pattern associated with the following electromagnetic scattering problem for the effective medium as $a \rightarrow 0$,  
\begin{equation}\label{model-equi}
	\begin{cases}
	\mathrm{Curl}E^{{\mathring\epsilon}_r}-i k {\mathring\mu}_r H^{{\mathring\mu}_r}= 0,\quad \mathrm{Curl} H^{{\mathring\mu}_r}+i k{\mathring\epsilon_r} E^{{\mathring\epsilon}_r} = 0 \quad\mbox{in} \ \Omega,\\
	\mathrm{Curl} E^{{\mathring\epsilon}_r}-i k H^{{\mathring\mu}_r}=0,\quad\mathrm{Curl} H^{{\mathring\mu}_r}+i k E^{{\mathring\epsilon}_r}=0\quad\mbox{in}\ \mathbb{R}^3\backslash \Omega,\\
	E^{{\mathring\epsilon}_r}=E_s^{{\mathring\epsilon}_r}+E^{{\mathring\epsilon}_r}_{in}, \quad 	H^{{\mathring\mu}_r}=H_s^{{\mathring\mu}_r}+H^{{\mathring\mu}_r}_{in},\\
	\nu\times E^{{\mathring\epsilon}_r}|_{+}=\nu\times E^{{\mathring\epsilon}_r}|_{-}, \quad \nu\times H^{{\mathring\mu}_r}|_{+}=\nu\times H^{{\mathring\mu}_r}|_{-}\quad\mbox{on} \ \partial \Omega,\\
	H^{{\mathring\mu}_r}_s\times\frac{x}{|x|}-E^{{\mathring\epsilon}_r}_s=\Oh({\frac{1}{|x|^2}}),\quad \mbox{as} \ |x|\rightarrow\infty
	\end{cases},
	\end{equation} 
The effective permittivity and permeability are respectively given by
	\begin{equation}\label{new-coeff}
		\mathring{\epsilon}_r \, = \,  I_{3} \quad \text{and} \quad
		\mathring \mu_r \, = \, I_{3} \, \pm \, \xi \,  {\bf P_0} \cdot \mathcal{A} \quad \mbox{in} \quad \Omega,
	\end{equation} 
 where ${\bf P_0}$ is the tensor defined by $(\ref{defP0})$ and the tensor $\mathcal{A}$ is defined in $\textbf{Definition \ref{DefM}}$. 
\end{theorem}
The justification of \textbf{Theorem \ref{MainThm}} highly relies on the $C^{0, \alpha}$-regularity of the solution $(E^{\mathring{\epsilon_{r}}}, H^{\mathring{\mu_{r}}})$ to \eqref{model-equi}.

\subsection{Application to the $2$D Van der Waals heterostructures}

As mentioned earlier, the van der Waals heterostructures are built as {\it{\textbf{superposition}}} of 2D-sheets, of subwavelength nanoparticles, which are globally periodic but locally formed of heterogeneous clusters of nanoparticles. In Figure $\ref{VanDerWaals}$, we depict few examples of such sheets composed of particles of distributed on honeycomb structure, of hexagonal form see Figure $\ref{subfiga}$, or following distributions as trigonal prismatic Figure $\ref{subfigb}$, octahedral Figure $\ref{subfigc}$, chalcogenides Figure $\ref{subfigd}$ and Boron nitrids Figure $\ref{subfige}$. More structures can be found in\cite{van-der-Waals-paper, van-der-Waals-book}. The 3D-materials are build by superposing such types of sheets (composed of same type of sheets).  With such superposition, the material occupies a domain $\Omega$ that we decompose as described in Assumption $(\ref{assII})$ as periodic distribution of blocks $\Omega_m$ and each block contains $K+1$ particle. As an example, we show in Figure $\ref{subfigf}$ the $\Omega_m$'s and $K+1$ particles inside each $\Omega_m$. 
\begin{figure}[H]
     \centering
     \begin{subfigure}[b]{0.25\textwidth}
         \centering
         \includegraphics[width=\textwidth]{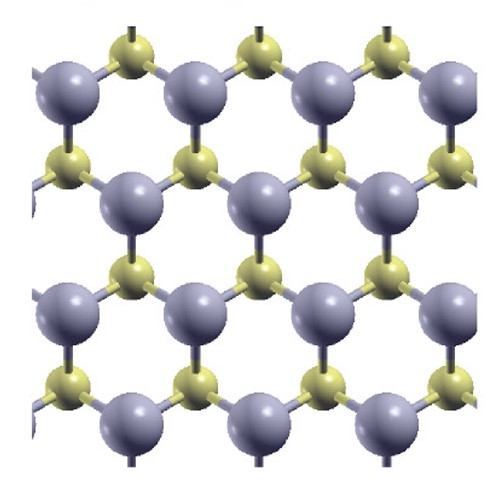}
         \caption{}
         \label{subfiga}
     \end{subfigure}
     \hfill
     \begin{subfigure}[b]{0.25\textwidth}
         \centering
         \includegraphics[width=\textwidth]{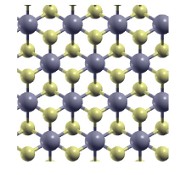}
         \caption{}
         \label{subfigb}
     \end{subfigure}
     \hfill
     \begin{subfigure}[b]{0.25\textwidth}
         \centering
         \includegraphics[width=\textwidth]{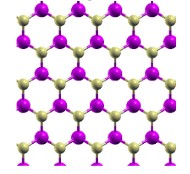}
         \caption{}
         \label{subfigc}
     \end{subfigure}
      \vfill
  \centering
     \begin{subfigure}[b]{0.25\textwidth}
         \centering
         \includegraphics[width=\textwidth]{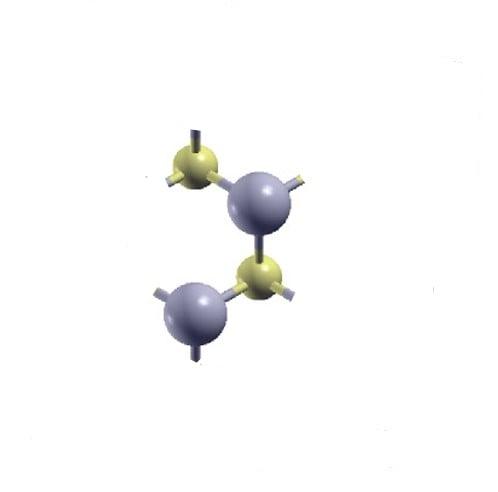}
         \caption{}
         \label{subfigd}
     \end{subfigure}
     \hfill
     \begin{subfigure}[b]{0.25\textwidth}
         \centering
         \includegraphics[width=\textwidth]{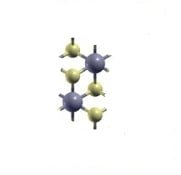}
         \caption{}
         \label{subfige}
     \end{subfigure}
     \hfill
     \begin{subfigure}[b]{0.25\textwidth}
         \centering
         \includegraphics[width=\textwidth]{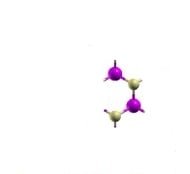}
         \caption{}
         \label{subfigf}
     \end{subfigure}    
        \caption{Van Der Waals Heterostructures. The figures $(\ref{subfiga})$, $(\ref{subfigb})$  and $(\ref{subfigc})$  show examples of 2D sheets that we can superpose to generate Van Der Waals materials. The figures $(\ref{subfigd})$, $(\ref{subfige})$ and $(\ref{subfigf})$ display the corresponding elements that can be enclosed in each $\Omega_m$ which form the decomposition of $\Omega$.}
        \label{VanDerWaals}
\end{figure}
\textit{\Large These pictures were taken from the reference \cite{van-der-Waals-book}.}
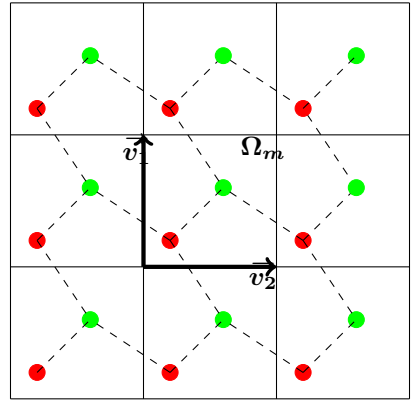
\begin{figure}[H]
    \centering
\begin{tikzpicture}[scale = 0.35]
\filldraw[red] (1,1) circle (0.3);
\filldraw[green] (3,3) circle (0.3);
\draw[dashed] (1,1) -- (3,3);
\filldraw[red] (6,1) circle (0.3);
\filldraw[green] (8,3) circle (0.3);
\draw[dashed] (6,1) -- (8,3);
\filldraw[red] (11,1) circle (0.3);
\filldraw[green] (13,3) circle (0.3);
\draw[dashed] (11,1) -- (13,3);
\filldraw[red] (1,6) circle (0.3);
\filldraw[green] (3,8) circle (0.3);
\draw[dashed] (1,6) -- (3,8);
\draw[dashed] (1,6) -- (3,3);
\draw[dashed] (6,1) -- (3,3);
\filldraw[red] (6,6) circle (0.3);
\draw[dashed] (6,6) -- (3,8);
\draw[dashed] (6,6) -- (8,3);
\filldraw[red] (1,11) circle (0.3);
\draw[dashed] (1,11) -- (3,8);
\draw[dashed] (11,1) -- (8,3);
\filldraw[green] (8,8) circle (0.3);
\draw[dashed] (6,6) -- (8,8);
\filldraw[green] (3,13) circle (0.3);
\draw[dashed] (1,11) -- (3,13);
\filldraw[red] (6,11) circle (0.3);
\filldraw[red] (11,6) circle (0.3);
\draw[dashed] (11,6) -- (13,3);
\draw[dashed] (11,6) -- (13,8);
\draw[dashed] (11,6) -- (8,8);
\draw[dashed] (6,11) -- (8,8);
\draw[dashed] (6,11) -- (3,13);
\draw[dashed] (6,11) -- (8,13);
\draw[dashed] (11,11) -- (8,13);
\draw[dashed] (13,8) -- (11,11);
\filldraw[green] (13,8) circle (0.3);
\filldraw[green] (8,13) circle (0.3);
\filldraw[red] (11,11) circle (0.3);
\filldraw[green] (13,13) circle (0.3);
\draw[dashed] (11,11) -- (13,13);
\draw (0,0) -- (0,15);
\draw (0,0) -- (15,0);
\draw (0,5) -- (15,5);
\draw (0,10) -- (15,10);
\draw (0,15) -- (15,15);
\draw (5,0) -- (5,15);
\draw (10,0) -- (10,15);
\draw (15,0) -- (15,15);
\draw[ultra thick, ->] (5,5) -- (10,5);
\draw[ultra thick, ->] (5,5) -- (5,10); 
\draw node at (9.5,9.5) {$\boldsymbol{\Omega_m}$}; 
\draw node at (9.5,4.75) {$\boldsymbol{\overset{\rightarrow}{v_{2}}}$};
\draw node at (4.75,9.5) {$\boldsymbol{\overset{\rightarrow}{v_{1}}}$};
\end{tikzpicture}
    \caption{Elongated Honeycomb. Here, the domain $\Omega$ is a cube and the decomposing subdomains $\Omega_m$'s are also cubes. In this figure, we display the sections (of square form) of $\Omega$ and $\Omega_m$'s. Each $\Omega_m$ contains $2$ particles, located at the points $z_{m_{1}}$ and $ z_{m_{2}}$, $m=1, .2... M$, (one is green and the other one red colored). The vector formed by each two points is characterized by the angle of the Honeycomb element.}
    \label{fig:enter-label}
\end{figure}
Now, we provide with more details about the effective medium for the case of honeycomb structure, see Figure \ref{fig:enter-label}. To match this structure with our analysis, we first rotate the superposed sheets in such a way that the $\Omega_m$'s have the needed form of cube. We see that, in this new view, we have two particles, in each $\Omega_m$, that are oriented according to the original honeycomb structure. Therefore, to compute the equivalent permeability, it is enough to compute it for a dimer \footnote{Dimer is referring to two nano-particles that are close to each other.}. However, the permeability $\mathring \mu_r$ defined in $(\ref{new-coeff})$ for the case of a dimer can be expressed explicitly when the tensor ${\bf P_0}$, given by $(\ref{defP0})$, is defined over the unit ball as 
\begin{equation*}
\mathring \mu_r \, = \, \alpha \, I_{3} \, - \, \beta \, \left( z_{0_{1}} - z_{0_{2}} \right) \otimes \left( z_{0_{1}} - z_{0_{2}} \right), 
\end{equation*}
where 
\begin{equation*}
    \alpha := \dfrac{\pi^{9} \, \pm \, 32 \, \xi \,  \pi^{6} \, \pm \, 3 \xi \pi^{5} \, - 192 \xi^{2} \, \pi^{3} \, \pm 384 \, \xi^{3}}{\pi^{5} \, \left[ \pi \, \left(\pi^{3} \mp 8 \, \xi \right) \pm 3 \xi \right]} \;\; \text{and} \;\;
    \beta  := \dfrac{216 \, \xi^{2} \, \left(\pi^{3} \mp 4 \xi \right)^{2}}{\pi^{5} \, d^{2} \, \left[ \pi \, \left(\pi^{3} \mp 8 \, \xi \right) \pm 3 \xi \right] \, \left[ \pi \, \left(\pi^{3} \mp 8 \, \xi \right) - \pm 6 \xi \right]},
\end{equation*}
due to $(\ref{ExpressionA})$.
\bigskip

We finish this section with few comments.
\begin{enumerate}
\bigskip

\item If we take a perfect honeycomb (with and angle of $120$ degrees), then the subdomains $\Omega_m$'s will have a shape of a parallelepiped and not cube. Therefore, the tensor $\mathcal{A}$, given in (\ref{SKM}), will have a different form and then the permeability $\mathring \mu_r$ as well. For every given, or desired to be designed, structure (as those in Fig \ref{VanDerWaals}), we need to correspond the natural subdomain $\Omega_m$'s and the elements of the  structure that are enclosed in it, After that, one can compute the tensors $\mathcal{A}$ (and ${\bf P_0}$ as well). It would be nice, and useful, to derive these values for few examples, as those displayed in Fig \ref{VanDerWaals}. In the recent years, there is a growing interest in the engineering literature for (formally) characterising the effective medium for such structures, see for instance \cite{Guanya-et-al} and the references therein. Such analysis will be done systematically in a near future.
\bigskip

\item  The analysis in this work can be done for different materials. Here, we discussed dielectric nanoparticles. However a similar study can be carried out for plasmonics or even non-dispersive material. It can also be done for bubbles or droplets in acoustics and elastic inclusions as well. The specificity of the dielectric nanoparticles, as for other dispersive materials, is that they can generate subwavelength scattering resonances. A cluster of such material generate extended low frequency resonances as well, see \cite{cao2023all} for more details. This phenomenon can be seen here as coefficients $\alpha$ or/and $\beta$, above, can be very large for singular values of $\xi$. Remember that $\xi$ is build up from the incident frequency, see (\ref{def-xi}).  Therefore, we can expect $\Omega$ to generate low-frequency all-dielectric (or low-frequency Mie-type) resonances. Similarly, the coefficients $\alpha$ or/and $\beta$ can have negative values to allow generations of (low frequency) plasmonic resonances on the extended domain $\Omega$. Such questions deserve to be studied in more details. The approach developed in \cite{cao2023all} can be useful for this purpose. 
\bigskip

\item In the presentation above we form the Van der Waals structure as superposition of sheets which are all identical. This is not necessary. We could also superpose different structured sheets. In this case, we can expect as an effective medium a stratified medium (i.e. varying in the superposition direction). We can also push further this freedom to design more elaborated structures. A typical structure is related to the Moir\'e materials, see \cite{Moire-material}, where the sheets are superposed after being rotated. We do believe that with our approach, we can characterize the corresponding equivalent medium. 
\bigskip

\item In addition to characterizing the effective mediums for such structures, several questions are naturally related. Of a particular importance is the study of the spectral properties of the  operators related to such structures. In particular when we have periodicity of the stratification. 
\end{enumerate}
\section{Proof of Theorem \ref{MainThm}}
\subsection{Rewriting the algebraic system, and formulating an integral equation, according to the distribution of the nano-particles}\label{GIEGAS}
In this subsection, we derive a general algebraic system and set its corresponding general integral equation. Afterwards, we estimate the disparity between the general algebraic system and the discretized integral equation. In addition, we justify that both the general algebraic system and the general integral equation can be reduced, in terms of complexity, to a simpler algebraic system and a simpler integral equation, that we call in the sequel Lippmann-Schwinger Equation (L.S.E). Based on 
\textbf{Proposition \ref{prop-discrete}}, for the discrete form of the Foldy-Lax approximation \eqref{approximation-E}, and the corresponding linear algebraic system \eqref{linear-discrete}, by denoting 
\begin{equation}\label{vari-substi}
	U_n=\pm c_0 \, a^{h-5}{\bf P_0^{-1}} \cdot Q_n,
\end{equation} 
we obtain 
\begin{equation*}
    	U_n \, - \, \frac{\eta\, k^2}{\pm c_0}\, a^{5-h} \sum_{p=1 \atop p \neq n}^M \Upsilon_k(z_n, z_p) \cdot {\bf P_0} \cdot U_p \, = \, i \, k\, H^{Inc}(z_n).
\end{equation*}
Then by combining with the fact that $\eta \, = \, \eta_{0}\, a^{-2}$, $\aleph = [d^{-3}]$, using $(\ref{def-xi})$ and $(\ref{distribute})$, we derive the following expression,
\begin{equation}\label{alg-1}
 U_n -  \pm \, \xi \, d^{3} \, \sum_{p=1 \atop p \neq n}^{[d^{-3}] \times (K+1)} \Upsilon_k(z_n, z_p) \cdot {\bf P_0} \cdot U_p \,=  \, i \, k\, H^{Inc}(z_n).
\end{equation}
Following the distribution assumption presented in {\bf Assumption}, we can rewrite \eqref{alg-1} in terms of the split of the sub-subdomains $\Omega_{m_\ell}$, with $m=1, 2, \cdots, [d^{-3}]$ and $\ell= 1, 2, \cdots, K+1$, as
\begin{equation*}
	U_{m_\ell} \, - \,  \pm \, \xi \, d^{3} \, \sum_{j=1}^{[d^{-3}]}\sum_{i=1\atop j_i\neq m_\ell}^{K+1}\Upsilon_{k}(z_{m_\ell}, z_{j_i}) \cdot {\bf P_0} \cdot U_{j_i} \, = \, i \, k \, H^{Inc}(z_{m_\ell}),
\end{equation*}
which indicates precisely that
\begin{equation}\label{al-discrete-pas1}
	U_{m_\ell} \, - \,  \pm \, \xi \, d^{3} \, \sum_{i=1\atop i\neq \ell}^{K+1}\Upsilon_{k}(z_{m_\ell}, z_{m_i}) \cdot {\bf P_0} \cdot U_{m_i} \, - \,  \pm \, \xi \, d^{3}  \, \sum_{j=1\atop j\neq m}^{[d^{-3}]}\sum_{i=1}^{K+1}\Upsilon_{k}(z_{m_\ell}, z_{j_i}) \cdot {\bf P_0} \cdot U_{j_i} \, = \, i \, k \, H^{Inc}(z_{m_\ell}).
\end{equation}
Considering the second term on the L.H.S of \eqref{al-discrete-pas1}, since we have\footnote{This step is merely technical. It allows us to avoid the dependency of the tensor $
     \D_m\left({\bf P_0}\mathbb{I}\right)$, given by $(\ref{TensorDP})$, on the frequency $k$.} 
\begin{equation*}
	\Upsilon_{k}(z_{m_\ell}, z_{m_i}) \, = \, \Upsilon_{0}(z_{m_\ell}, z_{m_i}) \, + \, \left(\Upsilon_{k}-\Upsilon_{0}\right)(z_{m_\ell}, z_{m_i}),
\end{equation*}
the equation $(\ref{al-discrete-pas1})$ becomes, 
\begin{eqnarray}\label{al-discrete-original}
\nonumber
	U_{m_\ell} \, &-& \, \pm \, \xi \, d^{3} \, \sum_{i=1\atop i\neq \ell}^{K+1}\Upsilon_{0}(z_{m_\ell}, z_{m_i}) \cdot {\bf P_0} \cdot U_{m_i} \\ &-& \, \pm \, \xi \, d^{3} \,  \sum_{j=1\atop j\neq m}^{[d^{-3}]}\sum_{i=1}^{K+1}\Upsilon_{k}(z_{m_\ell}, z_{j_i}) \cdot {\bf P_0} \cdot U_{j_i} \, = \, i \, k \, H^{Inc}(z_{m_\ell}) \, + \, Term_{m_{\ell}},
\end{eqnarray}
where 
\begin{equation}\label{Termml}
    Term_{m_{\ell}} \,= \, \pm \, \xi \, d^{3} \, \sum_{i=1\atop i\neq \ell}^{K+1} \left( \Upsilon_{k} \, - \, \Upsilon_{0}\right)(z_{m_\ell}, z_{m_i}) \cdot {\bf P_0} \cdot U_{m_i}.
\end{equation}
 Then, by using the Taylor expansion for the incident field, we derive from $(\ref{al-discrete-original})$ the following equation
 \begin{equation}\label{al-dis-matrix}
 	\left(\mathbb{I} \, - \, \pm \, \xi \, d^{3} \, \D_m\left({\bf P_0}\mathbb{I}\right)\right) \cdot \U_m \, - \, \pm \, \xi \, d^{3} \,  \sum_{j=1 \atop j \neq m}^{[d^{-3}]} \, \C_{mj}\left({\bf P_0}\mathbb{I}\right)_{k} \cdot \U_j \, = \, i \, k \, \hh^{Inc}_{m} \, + \, Error_m^{(1)},
 \end{equation}
 where $\mathbb{I} \, := \, \mathbb{I}_{3(K+1)\times 3(K+1)}$ is the identity matrix, $\D_m\left({\bf P_0}\mathbb{I}\right)$ is the tensor given by 
 \begin{equation}\label{TensorDP}
     \D_m\left({\bf P_0}\mathbb{I}\right) \, := \, \left(\begin{array}{cccc}
 	0& \Upsilon_{0}(z_{m_1}, z_{m_2})\cdot{\bf P_0}  & \cdots  & \Upsilon_{0}(z_{m_1}, z_{m_{K+1}})\cdot{\bf P_0} \\ 
 	\Upsilon_{0}(z_{m_2}, z_{m_1})\cdot{\bf P_0}&0  &\cdots  & \Upsilon_{0}(z_{m_2}, z_{m_{K+1}})\cdot{\bf P_0} \\ 
 	\vdots& \vdots & \ddots &\vdots  \\ 
 	\Upsilon_{0}(z_{m_{K+1}}, z_{m_1})\cdot{\bf P_0}&\Upsilon_{0}(z_{m_{K+1}}, z_{m_2})\cdot{\bf P_0} &\cdots  &0 
 	\end{array} 
 	\right),
 \end{equation}
 the tensor $\C_{mj}\left({\bf P_0}\mathbb{I}\right)_{k}$ is given by 
 \begin{equation}\label{TensorCmj}
     \C_{mj}\left({\bf P_0}\mathbb{I}\right)_{k} \, := \, \left(
 	\begin{array}{cccc}
 	\Upsilon_{k}(z_{m_1}, z_{j_1}) \cdot {\bf P_0}& \Upsilon_{k}(z_{m_1}, z_{j_2}) \cdot {\bf P_0} &\cdots  & \Upsilon_{k}(z_{m_1}, z_{j_{K+1}}) \cdot {\bf P_0} \\ 
 	\Upsilon_{k}(z_{m_2}, z_{j_1}) \cdot {\bf P_0}& \Upsilon_{k}(z_{m_2}, z_{j_2}) \cdot {\bf P_0} &\cdots  & \Upsilon_{k}(z_{m_{2}}, z_{j_{K+1}}) \cdot {\bf P_0} \\ 
 	\vdots&\vdots  &\ddots  &\vdots  \\ 
 	\Upsilon_{k}(z_{m_{K+1}}, z_{j_1}) \cdot {\bf P_0}&\Upsilon_{k}(z_{m_{K+1}}, z_{j_2}) \cdot {\bf P_0}  &\cdots  &\Upsilon_{k}(z_{m_{K+1}}, z_{j_{K+1}}) \cdot {\bf P_0} 
 	\end{array} 
 	\right),
 \end{equation}
 showing the $(K+1)^{2}$-blocks interactions between the nano-particles located in $\Omega_{j}$ and $\Omega_{m}$, with $j \neq m$. Moreover, the vectors $\mathbb{U}_m$, $\hh^{Inc}_m$ and $Error_m^{(1)}$, are given by 
 \begin{equation}\label{notation-tensor}
 \mathbb{U}_m \, := \, \left(\begin{array}{c}
	U_{m_1}\\ 
	\vdots\\ 
	U_{m_{K+1}}
	\end{array} \right) \quad \text{,} \quad  \hh^{Inc}_m \, := \, \left(\begin{array}{c}
	H^{Inc}(z_{m_1})\\ 
	\vdots\\ 
	H^{Inc}(z_{m_{1}})
	\end{array} \right) \quad \text{and} \quad
    Error_m^{(1)} \, := \,  \begin{pmatrix}
 	    Term_{m_{1}}^{\star} \\
            \vdots \\
            Term_{m_{K+1}}^{\star}
 	\end{pmatrix},
 \end{equation}
 where 
 \begin{equation}\label{Termmlstar}
     Term_{m_{\ell}}^{\star} \, := \, Term_{m_{\ell}} \, + \, i \, k \, \left( H^{Inc}(z_{m_{\ell}}) \, - \,  H^{Inc}(z_{m_{1}}) \right), \quad 1 \leq \ell \leq K+1. 
 \end{equation}
  In addition, since the subdomains $\Omega_{m}$, $m=1, \cdots, [d^{-3}]$, are distributed periodically in terms of the reference set $\Omega_{0}$, we know that $\D_m$ are equivalent, and we denote the unified tensor as $\D_0:=\D_m$, $m=1, \cdots, [d^{-3}]$. Then in \eqref{al-dis-matrix}, by denoting 
\begin{equation}\label{def-T0-matrix}
	\mathbb{T}_0 \, := \, \mathbb{I} \, - \, \pm \, \xi \, d^{3} \,  \D_0({\bf P_0}{\mathbb{I}}),
\end{equation}
we get the following general algebraic system (G.A.S)
\begin{equation}\label{al-dis1}
	\mathbb{T}_0 \cdot \U_m \, - \, \pm \, \xi \, d^{3} \, \sum_{j=1 \atop j \neq m}^{[d^{-3}]}\C_{mj}({\bf P_0}\mathbb{I})_{k} \cdot \U_j \, = \, \, i \, k \, \hh^{Inc}_m \, + \, Error_m^{(1)}.
\end{equation}
Next, our aim is to associate $(\ref{al-dis1})$ with a corresponding general integral equation. To do this, inspired by $(\ref{al-dis1})$ and \cite[Section 3]{Cao-Sini}, we set up the following general integral equation (G.I.E)
\begin{equation}\label{G.I.E}
    \mathbb{T}_1 \cdot \mathcal{F}(x)  \, - \, \pm  \, \xi \, \int_{\Omega} \, \C\left({\bf P_0}\mathbb{I}\right)_{k}(x,y) \cdot \mathcal{F}(y) \, dy 
       \,  = \, i \, k  \, \mathcal{H}^{Inc}(x), \quad  \text{in} \quad \Omega,
\end{equation}
where the matrix-kernel $\C\left({\bf P_0}\mathbb{I}\right)_{k}(\cdot,\cdot)$ is given by 
\begin{equation}\label{Re344}
    \C\left({\bf P_0}\mathbb{I}\right)_{k}(x,y) \, := \, \begin{pmatrix}
        \Upsilon_{k}(x,y) \cdot {\bf P_0} & \cdots & \Upsilon_{k}(x,y) \cdot {\bf P_0} \\
        \vdots & \ddots & \vdots \\
        \Upsilon_{k}(x,y) \cdot {\bf P_0} & \cdots & \Upsilon_{k}(x,y) \cdot {\bf P_0}
    \end{pmatrix}, \quad x \neq y,
\end{equation}
and the tensor $\mathbb{T}_1$ is given by
\begin{equation}\label{DefT1}
    \mathbb{T}_1 \, := \, \mathbb{T}_0 \, - \, \pm  \, \xi \,  \frac{1}{\left\vert \Omega_{0} \right\vert} \, \int_{\Omega_{0}}  \, \begin{pmatrix}
        \nabla {\bf M}_{\Omega_{0}}\left( I_{3} \right)(x) \cdot {\bf P_0} & \cdots &  \nabla {\bf M}_{\Omega_{0}}\left( I_{3} \right)(x) \cdot {\bf P_0} \\
        \vdots & \ddots & \vdots \\
        \nabla {\bf M}_{\Omega_{0}}\left( I_{3} \right)(x) \cdot {\bf P_0} & \cdots &  \nabla {\bf M}_{\Omega_{0}}\left( I_{3} \right)(x) \cdot {\bf P_0}
    \end{pmatrix} \, dx,
\end{equation}
which is a shift of $\mathbb{T}_0$, defined by $(\ref{def-T0-matrix})$, with $\nabla {\bf M}_{\Omega_{0}}$ is the Magnetization operator defined by $(\ref{DefMagnetization})$. The shift's origin will be explained later, refer to $(\ref{Delta2m})-(\ref{ShiftCoeff})$. Besides, 
\begin{equation}\label{Re34}
    \mathcal{H}^{Inc} \, := \, \begin{pmatrix}
        H^{Inc} \\
        \vdots \\
        H^{Inc} 
    \end{pmatrix}  \quad \text{and} \quad \mathcal{F} \, := \, \begin{pmatrix}
        F_{1} \\
        \vdots \\
        F_{K+1} 
    \end{pmatrix}, \quad \text{with} \quad F_{j} \in \mathbb{C}^{1 \times 3}.  
\end{equation}
Now, using the definition of the matrix-kernel $\C\left({\bf P_0}\mathbb{I}\right)_{k}(\cdot,\cdot)$, see $(\ref{Re344})$, and the definition of the vector $\mathcal{F}(\cdot)$, see  $(\ref{Re34})$, we rewrite $(\ref{G.I.E})$ as, 
\begin{equation}\label{Re10}
 \mathbb{T}_{1} \cdot  \mathcal{F}(x) \, - \, \pm  \, \xi  \, \int_{\Omega} \, \begin{pmatrix}
     \Upsilon_{k}(x,y) \cdot {\bf P_0} \cdot \overset{K+1}{\underset{\ell=1}{\sum}} F_{\ell}(y)  \\ 
     \vdots \\
     \Upsilon_{k}(x,y) \cdot {\bf P_0} \cdot \overset{K+1}{\underset{\ell=1}{\sum}} F_{\ell}(y)
 \end{pmatrix}  \, dy \, = \, i \, k  \,  \mathcal{H}^{Inc}(x), \qquad x \in \Omega.
\end{equation}
It is clear from the previous equation that  the vector $\mathcal{F}(\cdot)$ can be reconstructed from the sum of its components, i.e. $\overset{K+1}{\underset{\ell=1}{\sum}} F_{\ell}(\cdot)$. In the sequel, we will derive and solve an integral equation that is satisfied by $\overset{K+1}{\underset{\ell=1}{\sum}} F_{\ell}(\cdot)$. Successively, by taking the inverse of $\mathbb{T}_{1}$ on both sides of $(\ref{Re10})$ and multiplying the obtained equation by the matrix $\left( I_{3}, \cdots, I_{3} \right) \in \mathbb{R}^{3(K+1) \times 3}$, we get 
\begin{equation}\label{T1MAS}
   \sum_{\ell=1}^{K+1} F_{\ell}(x) \, - \, \pm  \, \xi  \, \int_{\Omega} \, \left( I_{3}, \cdots, I_{3} \right) \cdot \mathbb{T}_{1}^{-1} \cdot \begin{pmatrix}
     \Upsilon_{k}(x,y) \cdot {\bf P_0} \cdot \overset{K+1}{\underset{\ell=1}{\sum}} F_{\ell}(y)  \\ 
     \vdots \\
     \Upsilon_{k}(x,y) \cdot {\bf P_0} \cdot \overset{K+1}{\underset{\ell=1}{\sum}} F_{\ell}(y)
 \end{pmatrix}  \, dy \, = \, i \, k  \,  \left( I_{3}, \cdots, I_{3} \right) \cdot \mathbb{T}_{1}^{-1} \cdot \mathcal{H}^{Inc}(x).
\end{equation}
For the purpose of writing compact formulas, we introduce the following notations,  
\begin{equation}\label{T1M}
    \left( I_{3}, \cdots, I_{3} \right) \cdot \mathbb{T}_{1}^{-1} := \left(\widetilde{\mathcal{A}}_{1}, \cdots, \widetilde{\mathcal{A}}_{K+1} \right) \quad \text{and} \quad \mathcal{A} := \sum_{\ell=1}^{K+1} \widetilde{\mathcal{A}}_{\ell}.
\end{equation}
Hence, using $(\ref{T1M})$ and $(\ref{Re34})$, the integral equation given by $(\ref{T1MAS})$ takes the following form,
\begin{equation}\label{ReducedIE}
   \sum_{\ell=1}^{K+1} F_{\ell}(x) \, - \, \pm  \, \xi  \,  \int_{\Omega} \mathcal{A} \cdot
     \Upsilon_{k}(x,y) \cdot {\bf P_0} \cdot \sum_{\ell=1}^{K+1} F_{\ell}(y) \, dy \, = \, i \, k  \, \mathcal{A} \cdot H^{Inc}(x),
\end{equation}
which is the integral equation satisfied by  $\overset{K+1}{\underset{\ell=1}{\sum}} F_{\ell}(\cdot)$. Equivalently, by taking the inverse of $\mathcal{A}^{-1}$ on the both sides of $(\ref{ReducedIE})$, we get 
\begin{equation}\label{Re14}
   \mathcal{A}^{-1} \cdot \sum_{\ell=1}^{K+1} F_{\ell}(x) \, - \, \pm  \, \xi  \, \int_{\Omega} \, 
     \Upsilon_{k}(x,y) \cdot {\bf P_0} \cdot \mathcal{A} \cdot \mathcal{A}^{-1} \cdot 
 \sum_{\ell=1}^{K+1} F_{\ell}(y) \, dy \, = \, i \, k  \,  H^{Inc}(x).
\end{equation}
To write short formulas, we set 
\begin{equation}\label{Re15}
    \textbf{F}(\cdot) \, := \,    \mathcal{A}^{-1} \cdot \sum_{\ell=1}^{K+1} F_{\ell}(\cdot). 
\end{equation}
Combining $(\ref{Re14})$ with $(\ref{Re15})$, we end up with the following Lippmann-Schwinger equation (L.S.E)
\begin{equation}\label{Re16}
    \textbf{F} - \, \pm  \, \xi  \, \left[ - \, \nabla {\bf M}^{k}_{\Omega}\left( {\bf P_0} \cdot \mathcal{A} \cdot \textbf{F} \right) \, + \, k^{2} \, {\bf N}^{k}_{\Omega}\left( {\bf P_0} \cdot \mathcal{A} \cdot \textbf{F} \right) \right] \, = \, i \, k \, H^{Inc},
\end{equation}
where $\nabla {\bf M}^{k}$ is the Magnetization operator with non-vanishing frequency defined by
\begin{equation*}
    \nabla {\bf M}^{k}\left( E \right)(x) \, := \, \nabla \int_{\Omega} \underset{y}{\nabla}\Phi_{k}(x,y) \cdot E(y) \, dy, \quad x \in \Omega, 
\end{equation*}
and ${\bf N}^{k}$ is the Newtonian operator with non-vanishing frequency given by 
\begin{equation*}
   {\bf N}^{k}\left( E \right)(x) \, := \, \int_{\Omega} \Phi_{k}(x,y)  E(y) \, dy, \quad x \in \Omega,
\end{equation*}
with $\Phi_{k}(\cdot,\cdot)$ being the kernel defined by $(\ref{fundsolheleq})$. 
\medskip  
\newline
Clearly, by discretizing the L.S.E given by the equation $(\ref{Re16})$, as done in \cite[Section 4]{cao2023all}, we get the following algebraic system 
\begin{eqnarray}\label{Eq1033}
\nonumber
    \begin{pmatrix}
        \mathcal{Q}_{1} \\
        \mathcal{Q}_{2} \\
        \vdots \\
        \mathcal{Q}_{[d^{-3}]}
    \end{pmatrix} &-& \pm \xi d^{3}  \begin{pmatrix}
        0 & \Upsilon_{k}(z_{1},z_{2}) \cdot {\bf T}^{\mathring\mu_{r}}_{2}  & \cdots &   \Upsilon_{k}(z_{1},z_{[d^{-3}]}) \cdot {\bf T}^{\mathring\mu_{r}}_{[d^{-3}]} \\
        \Upsilon_{k}(z_{2},z_{1}) \cdot {\bf T}^{\mathring\mu_{r}}_{1} & 0 & \cdots & \Upsilon_{k}(z_{2},z_{[d^{-3}]}) \cdot {\bf T}^{\mathring\mu_{r}}_{[d^{-3}]} \\
        \vdots & \vdots & \ddots & \vdots \\
        \Upsilon_{k}(z_{[d^{-3}]},z_{1}) \cdot {\bf T}^{\mathring\mu_{r}}_{1} & \Upsilon_{k}(z_{[d^{-3}]},z_{2}) \cdot {\bf T}^{\mathring\mu_{r}}_{2} & \cdots & 0 
    \end{pmatrix} \cdot \begin{pmatrix}
        \mathcal{Q}_{1} \\
        \mathcal{Q}_{2} \\
        \vdots \\
        \mathcal{Q}_{[d^{-3}]}
    \end{pmatrix} \\ \nonumber
    && \\
    &=& i \, k \,  \begin{pmatrix}
        \frac{1}{\left\vert \Omega_{1} \right\vert} \, \int_{\Omega_{1}} H^{Inc}(x) \, dx \\
        \frac{1}{\left\vert \Omega_{2} \right\vert} \, \int_{\Omega_{2}} H^{Inc}(x) \, dx \\
        \vdots \\
        \frac{1}{\left\vert \Omega_{[d^{-3}]} \right\vert} \, \int_{\Omega_{[d^{-3}]}} H^{Inc}(x) \, dx
    \end{pmatrix} \, + \begin{pmatrix}
        err_{1} \\
        err_{2} \\
        \vdots \\
        err_{[d^{-3}]}
    \end{pmatrix},
\end{eqnarray}
where the constant tensor ${\bf T}^{\mathring\mu_{r}}_{j}$ is defined by 
\begin{equation}\label{TensorTmu}
    {\bf T}^{\mathring\mu_{r}}_{j} \, := \, {\bf P_0} \cdot \mathcal{A} \cdot \left( I_{3} \, \pm  \, \xi \, \frac{1}{\left\vert \Omega_{j} \right\vert} \, \int_{\Omega_{j}} \nabla {\bf M}_{\Omega_{j}}\left( I_{3} \right)(x) \, dx \cdot {\bf P_0} \cdot \mathcal{A} \right)^{-1}, \quad 1 \leq j \leq [d^{-3}], 
\end{equation}
the unknown vector is such that 
\begin{equation*}\label{DefQm}
    \mathcal{Q}_{m} \, := \, \left( I_{3} \, \pm  \, \xi \, \dfrac{1}{\left\vert \Omega_{m} \right\vert} \int_{\Omega_{m}} \nabla {\bf M}_{\Omega_{m}}\left( I_{3} \right)(x) \, dx \cdot {\bf P_0} \cdot \mathcal{A} \right) \cdot \mathcal{A}^{-1} \cdot \frac{1}{\left\vert \Omega_{m} \right\vert} \int_{\Omega_{m}} \sum_{\ell=1}^{K+1} F_{\ell}(x) \, dx, \quad 1 \leq m \leq [d^{-3}],
\end{equation*}
and the error term is given by
\begin{eqnarray*}
    err_{m} \, &:=& \, \pm \, \xi \, \sum_{j=1 \atop j \neq m}^{[d^{-3}]} \frac{1}{\left\vert \Omega_{m} \right\vert} \int_{\Omega_{m}} \int_{\Omega_{j}} \left( \Upsilon_{k}(x,y) \, - \, \Upsilon_{k}(z_{m},z_{j}) \right) \cdot {\bf P_0} \cdot  \sum_{\ell=1}^{K+1} F_{\ell}(y) \, dy \, dx \\
    &\pm&  \xi \,  \frac{1}{\left\vert \Omega_{m} \right\vert} \int_{\Omega_{m}} \int_{\Omega_{m}} \Upsilon_{k}(x,y) \cdot {\bf P_0} \cdot \left( \sum_{\ell=1}^{K+1} F_{\ell}(y) - \frac{1}{\left\vert \Omega_{m} \right\vert} \int_{\Omega_{m}} \sum_{\ell=1}^{K+1} F_{\ell}(z) \, dz \right) \, dy \, dx \\
    &\pm&  \xi \,  \frac{1}{\left\vert \Omega_{m} \right\vert} \int_{\Omega_{m}} \int_{\Omega_{m}} \left( \Upsilon_{k} - \Upsilon_{0} \right)(x,y)  \, dy \, dx \cdot {\bf P_0} \cdot \dfrac{1}{\left\vert \Omega_{m} \right\vert} \int_{\Omega_{m}} \sum_{\ell=1}^{K+1} F_{\ell}(y) dy \\
    &\pm&  \xi \,  \frac{1}{\left\vert \Omega_{m} \right\vert} \int_{\Omega_{m}} \int_{\Omega \setminus \overset{[d^{-3}]}{\underset{j=1}\cup} \Omega_{j}}  \Upsilon_{k}(x,y) \cdot {\bf P_0} \cdot  \sum_{\ell=1}^{K+1} F_{\ell}(y)  \, dy \, dx.
\end{eqnarray*} 
Observe that the algebraic system given by $(\ref{Eq1033})$, has a similar expression as the one given in \cite[Section 6, Equation (6.1)]{CGS}. Hence, by using similar arguments as those used in this reference, we derive that under the condition 
\begin{equation*}
    \left\vert \xi \right\vert \, \underset{1 \leq j \leq [d^{-3}]}{\max} \left\vert {\bf T}^{\mathring\mu_{r}}_{j} \right\vert \, < \, 1,
\end{equation*}
the associated algebraic system, to the L.S.E given by $(\ref{Re16})$, is invertible. Additionally, in \textbf{Section \ref{PositiveOperator}}, we prove the invertibility of the L.S.E given by $(\ref{Re16})$. Hence, by its construction, the equation derived from the general integral equation (G.I.E), given by $(\ref{Re10})$ is also invertible. Consequently, the general integral equation (G.I.E) given by  $(\ref{G.I.E})$ is also invertible. 
\smallskip 
\begin{remark}
The Tensor ${\bf T}^{\mathring\mu_{r}}$ defined in $(\ref{TensorTmu})$ for the case of a Dimer can be expressed when the tensor ${\bf P_0}$, given by $(\ref{defP0})$, is defined over the unit ball, and this is due to $(\ref{ExpressionA})$, as
\begin{equation*}
{\bf T}^{\mathring\mu_{r}} \, = \, \alpha^{\star} \, I_{3} \, - \, \beta^{\star} \, \left( z_{0_{1}} \, - \, z_{0_{2}} \right) \otimes \left( z_{0_{1}} \, - \, z_{0_{2}} \right), 
\end{equation*}
where 
\begin{equation*}
    \alpha^{\star} = \dfrac{24 \, \left( \pi^{3} \mp 4 \xi \right)^{2}}{\pi^{5} \, \left[\pi \, \left(\pi^{3} \mp 8 \xi \right) \pm 3 \xi \right] \pm 8 \xi \left( \pi^{3} \mp 4 \xi \right)^{2}} \quad \text{and} \quad 
    \beta^{\star} = \dfrac{\pm \, 216 \, \xi \, \left( \pi^{3} \mp 4 \xi \right)^{2} \, \left[ \pi \, \left(\pi^{3} \mp 8 \, \xi \right) \pm 3 \xi \right]}{\left[ \left[\pi \, \left(\pi^{3} \mp 8 \xi \right) \pm 3 \xi \right] \pm 8 \xi \left( \pi^{3} \mp 4 \xi \right)^{2} \right] \, d^{2} \, \varkappa},
\end{equation*}
with 
\begin{equation*}
    \varkappa \, = \, \pi^{13} \, \mp \, 8 \, \xi \, \pi^{10} \, \mp \, 3 \, \xi \, \pi^{9} \, - \, 64 \, \xi^{2} \, \pi^{7} \, + \, 48 \, \xi^{2} \, \pi^{6} \, - \, 128 \, \xi^{2} \, \pi^{5} \, \pm \, 748 \, \xi^{3} \, \pi^{3} \, - \, 1024 \, \xi^{4} \, \pi \, + \, 384 \, \xi^{4}.
\end{equation*}
\end{remark}
\smallskip 
\phantom{}
\newline
The following step involves discretizing the general integral equation (G.I.E) given by $(\ref{G.I.E})$, and estimating the difference between the (G.I.E) and the general algebraic system (G.A.S) given by $(\ref{al-dis1})$. Therefore, recalling $(\ref{G.I.E})$, we take the average with respect to the variable $x$ over $\Omega_{m}$ to obtain,
\begin{equation*}
    \mathbb{T}_1 \cdot \frac{1}{\left\vert \Omega_{m} \right\vert} \int_{\Omega_{m}} \mathcal{F}(x) dx  \, - \, \pm  \, \xi \, \frac{1}{\left\vert \Omega_{m} \right\vert} \int_{\Omega_{m}} \int_{\Omega} \, \C\left({\bf P_0}\mathbb{I}\right)_{k}(x,y) \cdot \mathcal{F}(y) \, dy \, dx  
       \,  = \, i \, k  \, \frac{1}{\left\vert \Omega_{m} \right\vert} \int_{\Omega_{m}} \mathcal{H}^{Inc}(x) \, dx.
\end{equation*}
Using the fact that 
\begin{equation*}
    \Omega \, = \, \Omega_{m} \, \cup \, \left( \overset{[d^{-3}]}{\underset{j=1 \atop j \neq m}\cup} \Omega_{j} \right) \, \cup \, \left( \Omega \setminus \overset{[d^{-3}]}{\underset{j=1}\cup} \Omega_{j} \right),  
\end{equation*}
we obtain 
\begin{eqnarray*}
    \mathbb{T}_1 \cdot \frac{1}{\left\vert \Omega_{m} \right\vert} \int_{\Omega_{m}} \mathcal{F}(x) dx  \, 
       &-& \, \pm  \, \xi \, \sum_{j=1 \atop j \neq m}^{[d^{-3}]} \, \frac{1}{\left\vert \Omega_{m} \right\vert} \, \int_{\Omega_{m}} \int_{\Omega_{j}} \, \C\left({\bf P_0}\mathbb{I}\right)_{k}(x,y) \cdot \mathcal{F}(y) \, dy \, dx  
       \\
       &=& \, i \, k  \, \frac{1}{\left\vert \Omega_{m} \right\vert} \int_{\Omega_{m}} \mathcal{H}^{Inc}(x) \, dx 
        \, \pm  \, \xi \, \frac{1}{\left\vert \Omega_{m} \right\vert} \int_{\Omega_{m}} \int_{\Omega_{m}} \, \C\left({\bf P_0}\mathbb{I}\right)_{k}(x,y) \cdot \mathcal{F}(y) \, dy \, dx  
       \\
      &\pm&    \xi \, \frac{1}{\left\vert \Omega_{m} \right\vert} \int_{\Omega_{m}} \int_{\Omega \setminus \overset{[d^{-3}]}{\underset{j=1}\cup} \Omega_{j}} \, \C\left({\bf P_0}\mathbb{I}\right)_{k}(x,y) \cdot \mathcal{F}(y) \, dy \, dx.
\end{eqnarray*}
Knowing that $\left\{ z_{m_{p}} \right\}_{p=1}^{K+1} \subset \Omega_{m}$ and $\left\{ z_{j_{q}} \right\}_{q=1}^{K+1} \subset \Omega_{j}$, and by using the Taylor expansion near $\left\{ z_{m_{p}}, z_{j_{q}} \right\}_{q, p=1}^{K+1}$, we end up with the constant tensor $\C_{mj}\left({\bf P_0}\mathbb{I}\right)_{k}$, see $(\ref{TensorCmj})$ for its definition, as a dominant kernel for the second term on the L.H.S. More precisely, by adding and subtracting  $\C_{mj}\left({\bf P_0}\mathbb{I}\right)_{k}$ and using the fact that $\left\vert \Omega_{j} \right\vert=d^{3} $, we obtain
\begin{eqnarray*}
    \mathbb{T}_1 \cdot \frac{1}{\left\vert \Omega_{m} \right\vert} \int_{\Omega_{m}} \mathcal{F}(x) dx  \, 
       &-& \, \pm  \, \xi \, d^{3} \, \sum_{j=1 \atop j \neq m}^{[d^{-3}]}  \,   \C_{mj}\left({\bf P_0}\mathbb{I}\right)_{k} \cdot \, \frac{1}{\left\vert \Omega_{j} \right\vert} \int_{\Omega_{j}} \, \mathcal{F}(y) \, dy 
       \\
       &=& \, i \, k  \, \frac{1}{\left\vert \Omega_{m} \right\vert} \int_{\Omega_{m}} \mathcal{H}^{Inc}(x) \, dx 
       \, \pm  \, \xi \, \frac{1}{\left\vert \Omega_{m} \right\vert} \int_{\Omega_{m}} \int_{\Omega_{m}} \, \C\left({\bf P_0}\mathbb{I}\right)_{k}(x,y) \cdot \mathcal{F}(y) \, dy \, dx  
       \\
      & \pm & \xi \, \frac{1}{\left\vert \Omega_{m} \right\vert} \int_{\Omega_{m}} \int_{\Omega \setminus \overset{[d^{-3}]}{\underset{j=1}\cup} \Omega_{j}} \, \C\left({\bf P_0}\mathbb{I}\right)_{k}(x,y) \cdot \mathcal{F}(y) \, dy \, dx \\
      & \pm & \xi \, \sum_{j=1 \atop j \neq m}^{[d^{-3}]} \, \frac{1}{\left\vert \Omega_{m} \right\vert} \, \int_{\Omega_{m}} \int_{\Omega_{j}} \, \left( \C\left({\bf P_0}\mathbb{I}\right)_{k}(x,y) - \C_{mj}\left({\bf P_0}\mathbb{I}\right)_{k} \right) \cdot \mathcal{F}(y) \, dy \, dx.  
\end{eqnarray*}
Denoting  
\begin{equation}\label{FM2004}
    \widetilde{\mathbb{U}}_m \, := \, \frac{1}{\left\vert \Omega_{m} \right\vert} \int_{\Omega_{m}} \mathcal{F}(x) \, dx, \quad 1 \leq m \leq [d^{-3}], 
\end{equation}
we reduce the previous algebraic system to
\begin{eqnarray}\label{Re17}
\nonumber
    \mathbb{T}_1 \cdot \widetilde{\mathbb{U}}_m  \, 
       &-& \, \pm  \, \xi \, d^{3} \, \sum_{j=1 \atop j \neq m}^{[d^{-3}]}  \,   \C_{mj}\left({\bf P_0}\mathbb{I}\right)_{k} \cdot \widetilde{\mathbb{U}}_j 
       = \, i \, k  \, \frac{1}{\left\vert \Omega_{m} \right\vert} \int_{\Omega_{m}} \mathcal{H}^{Inc}(x) \, dx \\ \nonumber
        & \pm &  \, \xi \, \frac{1}{\left\vert \Omega_{m} \right\vert} \int_{\Omega_{m}} \int_{\Omega_{m}} \, \C\left({\bf P_0}\mathbb{I}\right)_{k}(x,y) \cdot \mathcal{F}(y) \, dy \, dx  
       \\ \nonumber
      & \pm &  \xi \, \frac{1}{\left\vert \Omega_{m} \right\vert} \int_{\Omega_{m}} \int_{\Omega \setminus \overset{[d^{-3}]}{\underset{j=1}\cup} \Omega_{j}} \, \C\left({\bf P_0}\mathbb{I}\right)_{k}(x,y) \cdot \mathcal{F}(y) \, dy \, dx \\ 
      & \pm &  \xi \, \sum_{j=1 \atop j \neq m}^{[d^{-3}]} \, \frac{1}{\left\vert \Omega_{m} \right\vert} \, \int_{\Omega_{m}} \int_{\Omega_{j}} \, \left( \C\left({\bf P_0}\mathbb{I}\right)_{k}(x,y) - \C_{mj}\left({\bf P_0}\mathbb{I}\right) \right)  \cdot  \mathcal{F}(y) \, dy \, dx .
\end{eqnarray}
Regarding the first term on the R.H.S, with help of Taylor expansion, we have 
\begin{equation}\label{Delta1m}
    \Delta_{1,m} \, := \, \frac{1}{\left\vert \Omega_{m} \right\vert} \int_{\Omega_{m}} \mathcal{H}^{Inc}(x) \, dx  \overset{(\ref{Re34})}{=}  \,     \frac{1}{\left\vert \Omega_{m} \right\vert} \int_{\Omega_{m}} \begin{pmatrix}
        H^{Inc}(x) \\
        \vdots \\
        H^{Inc}(x)
    \end{pmatrix}  \, dx 
     \overset{(\ref{notation-tensor})}{=}  \, \hh^{Inc}_m \, + \, \Gamma_{1,m},
\end{equation}
where $\Gamma_{1,m}$ is given by 
\begin{equation}\label{DefGamma1m}
       \Gamma_{1,m} :=   \frac{1}{\left\vert \Omega_{m} \right\vert} \int_{\Omega_{m}} \begin{pmatrix}
        \int_{0}^{1} \nabla H^{Inc}(z_{m_{1}}+t(x-z_{m_{1}})) \cdot (x - z_{m_{1}}) \, dt \\
        \vdots \\
        \int_{0}^{1} \nabla H^{Inc}(z_{m_{1}}+t(x-z_{m_{1}})) \cdot (x - z_{m_{1}}) \, dt
    \end{pmatrix}  \, dx. 
\end{equation}
Next, we set the second term on the R.H.S of $(\ref{Re17})$ to be 
\begin{eqnarray}\label{Delta2m}
    \Delta_{2,m} \, &:=& \,  \pm  \, \xi \, \frac{1}{\left\vert \Omega_{m} \right\vert} \int_{\Omega_{m}} \int_{\Omega_{m}} \, \C\left({\bf P_0}\mathbb{I}\right)_{k}(x,y) \cdot \mathcal{F}(y) \, dy \, dx \\ \nonumber
     &=& \,  \pm  \, \xi \, \frac{1}{\left\vert \Omega_{m} \right\vert} \int_{\Omega_{m}} \int_{\Omega_{m}} \, \C\left({\bf P_0}\mathbb{I}\right)_{0}(x,y) \, dy \, dx \cdot \frac{1}{\left\vert \Omega_{m} \right\vert} \int_{\Omega_{m}} \mathcal{F}(z) \, dz \, + \, \Gamma_{2,m},
\end{eqnarray}
where 
\begin{eqnarray}\label{DefGamma2m}
\nonumber
      \Gamma_{2,m} &:=&
    \pm \xi \, \frac{1}{\left\vert \Omega_{m} \right\vert} \int_{\Omega_{m}} \int_{\Omega_{m}} \, \C\left({\bf P_0}\mathbb{I}\right)_{k}(x,y) \cdot \left[ \mathcal{F}(y) - \frac{1}{\left\vert \Omega_{m} \right\vert} \int_{\Omega_{m}} \mathcal{F}(z) \, dz \right] \, dy \, dx \\
    &\pm& \xi \, \frac{1}{\left\vert \Omega_{m} \right\vert} \int_{\Omega_{m}} \int_{\Omega_{m}} \left[  \C\left({\bf P_0}\mathbb{I}\right)_{k}(x,y) - \C\left({\bf P_0}\mathbb{I}\right)_{0}(x,y) \right] \, dy \, dx \cdot \frac{1}{\left\vert \Omega_{m} \right\vert} \int_{\Omega_{m}} \mathcal{F}(z) \, dz.
\end{eqnarray}
Thanks to $(\ref{FM2004})$ and $(\ref{Re344})$, the expression of $\Delta_{2,m}$ becomes\footnote{We have used the fact that 
\begin{equation*}
    \int_{\Omega_{m}} \Upsilon_{0}(x,y) \cdot {\bf P_0} \, dy \, := \, \left( - \underset{x}{\nabla} \int_{\Omega_{m}} \underset{y}{\nabla} \Phi_{0}(x,y)  \right) \cdot {\bf P_0} = - \, \nabla {\bf M}_{\Omega_{m}}\left( I_{3} \right)(x)\cdot {\bf P_0}.
\end{equation*}}, 
\begin{eqnarray}\label{Eq1010}
\nonumber
     \Delta_{2,m}  &=&  - \,  \pm  \, \xi \, \frac{1}{\left\vert \Omega_{m} \right\vert} \int_{\Omega_{m}} 
     \begin{pmatrix}
         \nabla {\bf M}_{\Omega_{m}}\left( I_{3} \right)(x) \cdot {\bf P_0} & \cdots &          \nabla {\bf M}_{\Omega_{m}}\left( I_{3} \right)(x) \cdot {\bf P_0} \\
         \vdots & \ddots & \vdots \\
        \nabla {\bf M}_{\Omega_{m}}\left( I_{3} \right)(x) \cdot {\bf P_0} & \cdots &          \nabla {\bf M}_{\Omega_{m}}\left( I_{3} \right)(x) \cdot {\bf P_0}
     \end{pmatrix}  \, dx \cdot \mathbb{U}_{m} \, + \, \Gamma_{2,m} \\ \nonumber
     && \\
     &=&  - \,  \pm  \, \xi \, \frac{1}{\left\vert \Omega_{0} \right\vert} \int_{\Omega_{0}} 
     \begin{pmatrix}
         \nabla {\bf M}_{\Omega_{0}}\left( I_{3} \right)(x) \cdot {\bf P_0} & \cdots &          \nabla {\bf M}_{\Omega_{0}}\left( I_{3} \right)(x) \cdot {\bf P_0} \\
         \vdots & \ddots & \vdots \\
        \nabla {\bf M}_{\Omega_{0}}\left( I_{3} \right)(x) \cdot {\bf P_0} & \cdots &          \nabla {\bf M}_{\Omega_{0}}\left( I_{3} \right)(x) \cdot {\bf P_0}
     \end{pmatrix}  \, dx \cdot \mathbb{U}_{m} \, + \, \Gamma_{2,m},
\end{eqnarray}
where the translation argument from $\Omega_{m}$ to the reference set $\Omega_{0}$ provides with the justification for the final equality, see $(\ref{al-dis-matrix})$ and $(\ref{def-T0-matrix})$. Remark, from the formula above, that the dominant coefficient 
\begin{equation}\label{ShiftCoeff}
  - \,  \pm  \, \xi \, \frac{1}{\left\vert \Omega_{0} \right\vert} \int_{\Omega_{0}} 
     \begin{pmatrix}
         \nabla {\bf M}_{\Omega_{0}}\left( I_{3} \right)(x) \cdot {\bf P_0} & \cdots &          \nabla {\bf M}_{\Omega_{0}}\left( I_{3} \right)(x) \cdot {\bf P_0} \\
         \vdots & \ddots & \vdots \\
        \nabla {\bf M}_{\Omega_{0}}\left( I_{3} \right)(x) \cdot {\bf P_0} & \cdots &          \nabla {\bf M}_{\Omega_{0}}\left( I_{3} \right)(x) \cdot {\bf P_0}
     \end{pmatrix}  \, dx,   
\end{equation}
is the same coefficient used to shift $\mathbb{T}_0$ into $\mathbb{T}_1$, see $(\ref{DefT1})$. While we were inspired by the general algebraic system given by $(\ref{al-dis1})$, which was introduced using the tensor $\mathbb{T}_0$, we decided to use a different tensor  $\mathbb{T}_1$ to write its corresponding general integral equation, see $(\ref{G.I.E})$, which is justified by the above discussion.
\medskip
\newline
Now, by gathering $(\ref{Delta1m}), (\ref{Delta2m})$ and $(\ref{Eq1010})$ with the algebraic system $(\ref{Re17})$ and using the expression of the tensor $\mathbb{T}_1$, given by $(\ref{DefT1})$, we end up with the following algebraic system
\begin{equation}\label{Eq1037}
    \mathbb{T}_{0} \cdot \widetilde{\mathbb{U}}_m  \, 
       - \, \pm  \, \xi \, d^{3} \, \sum_{j=1 \atop j \neq m}^{[d^{-3}]}  \,   \C_{mj}\left({\bf P_0}\mathbb{I}\right)_{k} \cdot \widetilde{\mathbb{U}}_j \, = \, i \, k \, \hh^{Inc}_m \, + \, \Gamma_{m}, 
\end{equation}
where the error term $\Gamma_{m}$ is defined by 
\begin{eqnarray}\label{ExpressionGamma}
\nonumber
     \Gamma_{m} &:=&  i \, k \, \Gamma_{1,m} \, + \, \Gamma_{2,m} \, \pm \, \xi \, \frac{1}{\left\vert \Omega_{m} \right\vert} \int_{\Omega_{m}} \int_{\Omega \setminus \overset{[d^{-3}]}{\underset{j=1}\cup} \Omega_{j}} \, \C\left({\bf P_0}\mathbb{I}\right)_{k}(x,y) \cdot \mathcal{F}(y) \, dy \, dx \\ 
      &\pm&  \xi \, \sum_{j=1 \atop j \neq m}^{[d^{-3}]} \, \frac{1}{\left\vert \Omega_{m} \right\vert} \, \int_{\Omega_{m}} \int_{\Omega_{j}} \, \left( \C\left({\bf P_0}\mathbb{I}\right)_{k}(x,y) - \C_{mj}\left({\bf P_0}\mathbb{I}\right)_{k} \right) \cdot \mathcal{F}(y) \, dy \, dx,
\end{eqnarray}
with $\Gamma_{1,m}$ is defined by $(\ref{DefGamma1m})$ and $\Gamma_{2,m}$ is defined by $(\ref{DefGamma2m})$. In the sequel, we need to estimate the $\left\Vert \cdot \right\Vert_{\ell^{2}}$-norm of the error term given by $\Gamma_{m}$. To achieve this, we set and estimate the $\left\Vert \cdot \right\Vert_{\ell^{2}}$-norm of each term appearing on the R.H.S of $(\ref{ExpressionGamma})$, respectively. 
\begin{enumerate}
    \item[]
    \item Estimation of 
    \begin{eqnarray}\label{Eq0506}
    \nonumber
        \Gamma_{1,m} &\overset{(\ref{DefGamma1m})}{:=}& i \, k \, \,      \frac{1}{\left\vert \Omega_{m} \right\vert} \int_{\Omega_{m}} \begin{pmatrix}
        \int_{0}^{1} \nabla H^{Inc}(z_{m_{1}}+t(x-z_{m_{1}})) \cdot (x - z_{m_{1}}) \, dt \\
        \vdots \\
        \int_{0}^{1} \nabla H^{Inc}(z_{m_{1}}+t(x-z_{m_{1}})) \cdot (x - z_{m_{1}}) \, dt
    \end{pmatrix}  \, dx \\ \nonumber
    \sum_{m=1}^{[d^{-3}]} \left\vert \Gamma_{1,m} \right\vert^{2} & = & k^{2} \, (K+1) \,\sum_{m=1}^{[d^{-3}]} \frac{1}{\left\vert \Omega_{m} \right\vert^{2}} \,  
    \left\vert \int_{\Omega_{m}} \int_{0}^{1} \nabla H^{Inc}(z_{m_{1}} + t (x - z_{m_{1}})) \cdot (x - z_{m_{1}}) \, dt \, dx \right\vert^{2}  \\ \nonumber
    & \overset{(\ref{EincHinc})}{=} & k^{2} \, (K+1) \,\sum_{m=1}^{[d^{-3}]} \frac{1}{\left\vert \Omega_{m} \right\vert^{2}} \,  
    \left\vert \int_{\Omega_{m}} \left( \theta \times \mathrm{p} \right) \, e^{i \, k \, \theta \cdot z_{m_{1}}} \, \left(  e^{i \, k \, \theta \cdot (x - z_{m_{1}})}   \, - \, 1 \right) \, dx \right\vert^{2}  \\ \nonumber
    & \lesssim & k^{2} \,  \,\sum_{m=1}^{[d^{-3}]} \frac{1}{\left\vert \Omega_{m} \right\vert^{2}} \,  
    \left\vert \int_{\Omega_{m}} \left(  e^{i \, k \, \theta \cdot (x - z_{m_{1}})}   \, - \, 1 \right) \, dx \right\vert^{2}  \\ 
    & \lesssim & k^{k} \,  \,\sum_{m=1}^{[d^{-3}]} \frac{1}{\left\vert \Omega_{m} \right\vert^{2}} \,  
    \left\vert \int_{\Omega_{m}}  \theta \cdot (x - z_{m_{1}})    \, dx \right\vert^{2} 
    =  \mathcal{O}\left(k^{4} \, d^{-1} \right).
    \end{eqnarray}
    \item[]
    \item Estimation of
    \item[] 
\begin{eqnarray*}\label{EstG2m}
     \Gamma_{2,m} &\overset{(\ref{DefGamma2m})}{:=}& \pm \, \xi \, \frac{1}{\left\vert \Omega_{m} \right\vert} \int_{\Omega_{m}} \int_{\Omega_{m}} \, \left[  \C\left({\bf P_0}\mathbb{I}\right)_{k}(x,y) \, - \, \C\left({\bf P_0}\mathbb{I}\right)_{0}(x,y) \right] \, dy \, dx \cdot \frac{1}{\left\vert \Omega_{m} \right\vert} \int_{\Omega_{m}} \mathcal{F}(z) \, dz  \\
    &+& \pm \, \xi \, \frac{1}{\left\vert \Omega_{m} \right\vert} \int_{\Omega_{m}} \int_{\Omega_{m}} \, \C\left({\bf P_0}\mathbb{I}\right)_{k}(x,y) \cdot \left[  \mathcal{F}(y) \, - \, \frac{1}{\left\vert \Omega_{m} \right\vert} \int_{\Omega_{m}} \mathcal{F}(z) \, dz \right] \, dy \, dx.
\end{eqnarray*}
We set and estimate the two terms appearing on the R.H.S, respectively. 
\begin{enumerate}
    \item[] 
    \item Estimation of 
    \begin{eqnarray*}
    \Gamma_{2,m,1} &:=& \pm \, \xi \, \frac{1}{\left\vert \Omega_{m} \right\vert} \int_{\Omega_{m}} \int_{\Omega_{m}} \, \left[  \C\left({\bf P_0}\mathbb{I}\right)_{k}(x,y) \, - \, \C\left({\bf P_0}\mathbb{I}\right)_{0}(x,y) \right] \, dy \, dx \cdot \frac{1}{\left\vert \Omega_{m} \right\vert} \int_{\Omega_{m}} \mathcal{F}(z) \, dz  \\
    & \overset{(\ref{Re344})}{=} & \pm \, \xi \, \frac{1}{\left\vert \Omega_{m} \right\vert} \int_{\Omega_{m}} \int_{\Omega_{m}} \begin{pmatrix}
        \left( \Upsilon_{k} - \Upsilon_{0} \right) \cdot {\bf P_0} & \cdots & \left( \Upsilon_{k} - \Upsilon_{0} \right) \cdot {\bf P_0} \\
        \vdots & \ddots & \vdots \\
        \left( \Upsilon_{k} - \Upsilon_{0} \right) \cdot {\bf P_0} & \cdots & \left( \Upsilon_{k} - \Upsilon_{0} \right) \cdot {\bf P_0}
    \end{pmatrix}(x,y) \, dy \, dx \cdot \frac{1}{\left\vert \Omega_{m} \right\vert} \int_{\Omega_{m}} \mathcal{F}(z) \, dz  \\
    \left\vert \Gamma_{2,m,1} \right\vert & \lesssim & \left\Vert \mathcal{F} \right\Vert_{\mathbb{L}^{\infty}(\Omega)} 
    \, \left\vert {\bf P_0} \right\vert \, \frac{1}{\left\vert \Omega_{m} \right\vert} \int_{\Omega_{m}} \int_{\Omega_{m}} \, \left\vert \left( \Upsilon_{k} - \Upsilon_{0} \right)(x,y) \right\vert \, dy \, dx  \\
    & \lesssim & k^{2} \, \left\Vert \mathcal{F} \right\Vert_{\mathbb{L}^{\infty}(\Omega)} 
    \, \left\vert {\bf P_0} \right\vert \, \frac{1}{\left\vert \Omega_{m} \right\vert} \int_{\Omega_{m}} \int_{\Omega_{m}} \frac{1}{\left\vert x - y \right\vert} \, dy \, dx,
    \end{eqnarray*}
    where the last estimate is justifiable by referring to \cite[Inequality (2.4.22)]{AM-1} and \eqref{LR}.
     Hence, using Taylor expansion with respect to the variable $x$ and keeping the dominant term of the resulting expression, we get the following approximation
    \begin{equation*}
    \left\vert \Gamma_{2,m,1} \right\vert         \lesssim  k^{2} \, \left\Vert \mathcal{F} \right\Vert_{\mathbb{L}^{\infty}(\Omega)} 
    \, \left\vert {\bf P_0} \right\vert \,  \int_{\Omega_{m}} \frac{1}{\left\vert z_{m} - y \right\vert} \, dy 
     =  \mathcal{O}\left( k^{2} \, \left\Vert \mathcal{F} \right\Vert_{\mathbb{L}^{\infty}(\Omega)} 
    \, \left\vert {\bf P_0} \right\vert \,  d^{2} \right) 
    \end{equation*}
    and
    \begin{equation}\label{Gamma2m1}
        \sum_{m=1}^{[d^{-3}]} \left\vert \Gamma_{2,m,1} \right\vert^{2} = \mathcal{O}\left( k^{4} \, \left\vert {\bf P_0} \right\vert^{2} \left\Vert \mathcal{F} \right\Vert^{2}_{\mathbb{L}^{\infty}(\Omega)} \, d \right).
    \end{equation}
    \item[]
    \item Estimation of
    \begin{eqnarray*}
        \Gamma_{2,m,2} &:=& \pm \, \xi \, \frac{1}{\left\vert \Omega_{m} \right\vert} \int_{\Omega_{m}} \int_{\Omega_{m}} \, \C\left({\bf P_0}\mathbb{I}\right)_{k}(x,y) \cdot \left[  \mathcal{F}(y) \, - \, \frac{1}{\left\vert \Omega_{m} \right\vert} \int_{\Omega_{m}} \mathcal{F}(z) \, dz \right] \, dy \, dx \\
        & \overset{(\ref{Re344})}{ \underset{(\ref{Re34})}{=}} & \pm \, \xi \, \frac{1}{\left\vert \Omega_{m} \right\vert} \int_{\Omega_{m}} \int_{\Omega_{m}} \begin{pmatrix}
           \Upsilon_{k}(x,y) \cdot {\bf P_0} \cdot \overset{K+1}{\underset{\ell=1}\sum} \left( F_{\ell}(y) - \frac{1}{\left\vert \Omega_{m} \right\vert} \int_{\Omega_{m}} F_{\ell}(z) \, dz \right) \\ 
           \vdots \\
          \Upsilon_{k}(x,y) \cdot {\bf P_0} \cdot \overset{K+1}{\underset{\ell=1}\sum} \left( F_{\ell}(y) - \frac{1}{\left\vert \Omega_{m} \right\vert} \int_{\Omega_{m}} F_{\ell}(z) \, dz \right)
        \end{pmatrix} \, dy \, dx \\
        & & \\
        & = & \pm \, \xi \, \frac{1}{\left\vert \Omega_{m} \right\vert} \int_{\Omega_{m}}  \begin{pmatrix}
            {{ \left( - \, \nabla {\bf M}^{k}_{\Omega_{m}}+k^2 \bf{N}_{\Omega_m}^k \right)}}\left( {\bf P_0} \cdot \overset{K+1}{\underset{\ell=1}\sum} \left( F_{\ell}(\cdot) - \frac{1}{\left\vert \Omega_{m} \right\vert} \int_{\Omega_{m}} F_{\ell}(z) \, dz \right) \right)(x) \\ 
           \vdots \\
        {{\left(- \, \nabla {\bf M}^{k}_{\Omega_{m}}+k^2 \bf{N}_{\Omega_m}^k \right)}} \left( {\bf P_0} \cdot \overset{K+1}{\underset{\ell=1}\sum} \left( F_{\ell}(\cdot) - \frac{1}{\left\vert \Omega_{m} \right\vert} \int_{\Omega_{m}} F_{\ell}(z) \, dz \right) \right)(x)
        \end{pmatrix} \, dx \\
        \left\vert \Gamma_{2,m,2} \right\vert & \lesssim & \frac{1}{\left\vert \Omega_{m} \right\vert} \int_{\Omega_{m}} \left\vert
          {{\left(- \, \nabla {\bf M}^{k}_{\Omega_{m}}+k^2 \bf{N}_{\Omega_m}^k \right)}} \left( {\bf P_0} \cdot \sum_{\ell = 1}^{K+1} \left( F_{\ell}(\cdot) - \frac{1}{\left\vert \Omega_{m} \right\vert} \int_{\Omega_{m}} F_{\ell}(z) \, dz \right) \right)(x)
        \right\vert \, dx \\
        & \lesssim & \frac{1}{\sqrt{\left\vert \Omega_{m} \right\vert}} \,  \left\Vert
         {{\left(- \, \nabla {\bf M}^{k}_{\Omega_{m}}+k^2 \bf{N}_{\Omega_m}^k \right)}} \left( {\bf P_0} \cdot \sum_{\ell = 1}^{K+1} \left( F_{\ell}(\cdot) - \frac{1}{\left\vert \Omega_{m} \right\vert} \int_{\Omega_{m}} F_{\ell}(z) \, dz \right) \right)
        \right\Vert_{\mathbb{L}^{2}(\Omega_{m})}  \\
        & \lesssim & \frac{1}{\sqrt{\left\vert \Omega_{m} \right\vert}} \,  \left\Vert
          {{\left(- \, \nabla {\bf M}^{k}_{\Omega_{m}}+k^2 \bf{N}_{\Omega_m}^k \right)}} \right\Vert_{\mathcal{L}\left(\mathbb{L}^{2}(\Omega_{m});\mathbb{L}^{2}(\Omega_{m})\right)} \left\vert {\bf P_0} \right\vert    \sum_{\ell = 1}^{K+1} \left\Vert  F_{\ell}(\cdot) - \frac{1}{\left\vert \Omega_{m} \right\vert} \int_{\Omega_{m}} F_{\ell}(z) \, dz  
        \right\Vert_{\mathbb{L}^{2}(\Omega_{m})}.  \end{eqnarray*}
        Using the fact that $\left\Vert
          {{\left(- \, \nabla {\bf M}^{k}_{\Omega_{m}}+k^2 \bf{N}_{\Omega_m}^k \right)}} \right\Vert_{\mathcal{L}\left(\mathbb{L}^{2}(\Omega_{m});\mathbb{L}^{2}(\Omega_{m})\right)} \, = \, \mathcal{O}\left( 1 \right)$, we reduce the previous inequality to 
        \begin{eqnarray*}
        \left\vert \Gamma_{2,m,2} \right\vert  & \lesssim & \frac{1}{\sqrt{\left\vert \Omega_{m} \right\vert}} \, \left\vert {\bf P_0} \right\vert \, \sum_{\ell=1}^{K+1} \, \left\Vert  F_{\ell}(\cdot) - \frac{1}{\left\vert \Omega_{m} \right\vert} \int_{\Omega_{m}} F_{\ell}(z) \, dz 
        \right\Vert_{\mathbb{L}^{2}(\Omega_{m})}
        \\
        & \lesssim & \frac{1}{\sqrt{\left\vert \Omega_{m} \right\vert}} \, \left\vert {\bf P_0} \right\vert \, \sqrt{K+1} \, \left( \sum_{\ell=1}^{K+1} \, \left\Vert  F_{\ell}(\cdot) - \frac{1}{\left\vert \Omega_{m} \right\vert} \int_{\Omega_{m}} F_{\ell}(z) \, dz 
        \right\Vert^{2}_{\mathbb{L}^{2}(\Omega_{m})} \right)^{\frac{1}{2}} \\
        & \overset{(\ref{Re34})}{=} & \frac{1}{\sqrt{\left\vert \Omega_{m} \right\vert}} \, \left\vert {\bf P_0} \right\vert \, \sqrt{K+1}  \, \left\Vert  \mathcal{F}(\cdot) - \frac{1}{\left\vert \Omega_{m} \right\vert} \int_{\Omega_{m}} \mathcal{F}(z) \, dz 
        \right\Vert_{\mathbb{L}^{2}(\Omega_{m})},
        \end{eqnarray*}
        and, by knowing that $\mathcal{F}(\cdot)$ satisfies a Hölderian estimate, see $(\ref{LR})$ for more details, we get  
        \begin{eqnarray}\label{Gamma2m2}
        \nonumber
        \left\vert \Gamma_{2,m,2} \right\vert  & \lesssim & \left[  \mathcal{F} \right]_{C^{0,\alpha}(\overline{\Omega})} \, \frac{1}{\sqrt{\left\vert \Omega_{m} \right\vert}} \left\vert {\bf P_0} \right\vert \, \sqrt{\int_{\Omega_{m}} \left\vert y - z_{m} \right\vert^{2 \alpha} \, dy} \\
       \nonumber
        \left\vert \Gamma_{2,m,2} \right\vert 
        & = & \mathcal{O}\left( \left[  \mathcal{F} \right]_{C^{0,\alpha}(\overline{\Omega})} \,  \left\vert {\bf P_0} \right\vert \, d^{\alpha} \right) \\
        \sum_{m=1}^{[d^{-3}]} \left\vert \Gamma_{2,m,2} \right\vert^{2} 
        & = & \mathcal{O}\left( \left[  \mathcal{F} \right]^{2}_{C^{0,\alpha}(\overline{\Omega})} \,  \left\vert {\bf P_0} \right\vert^{2} \, d^{2\alpha-3} \right).
    \end{eqnarray}
\end{enumerate}
Thanks to $(\ref{Gamma2m1})$ and $(\ref{Gamma2m2})$, the following estimation is justified,
\begin{equation}\label{Eq0427}
    \sum_{m=1}^{[d^{-3}]} \left\vert \Gamma_{2,m} \right\vert^{2} \, = \, \mathcal{O}\left( \left[ \mathcal{F} \right]^{2}_{C^{0,\alpha}(\overline{\Omega})} \,  \left\vert {\bf P_0} \right\vert^{2} \, d^{2 \alpha - 3} \right).
\end{equation}
\item[]
\item Estimation of
\begin{eqnarray*}
    \Gamma_{3,m} &:=& \, \pm  \, \xi \, \frac{1}{\left\vert \Omega_{m} \right\vert} \int_{\Omega_{m}} \int_{\Omega \setminus \overset{[d^{-3}]}{\underset{j=1}\cup} \Omega_{j}} \, \C\left({\bf P_0}\mathbb{I}\right)_{k}(x,y) \cdot \mathcal{F}(y) \, dy \, dx \\
    \left\vert \Gamma_{3,m} \right\vert & \lesssim & \, \left\Vert \mathcal{F} \right\Vert_{\mathbb{L}^{\infty}(\Omega)}  \, \frac{1}{\left\vert \Omega_{m} \right\vert} \int_{\Omega_{m}} \int_{\Omega \setminus \overset{[d^{-3}]}{\underset{j=1}\cup} \Omega_{j}} \left\vert \C\left({\bf P_0}\mathbb{I}\right)_{k}(x,y) \right\vert \, dy \, dx \\
    & \lesssim & \, \left\Vert \mathcal{F} \right\Vert_{\mathbb{L}^{\infty}(\Omega)}  \, \left\vert {\bf P_0} \right\vert \, \frac{1}{\left\vert \Omega_{m} \right\vert} \int_{\Omega_{m}} \int_{\Omega \setminus \overset{[d^{-3}]}{\underset{j=1}\cup} \Omega_{j}} \left\vert \Upsilon_{k}(x,y) \right\vert \, dy \, dx \\
    & \lesssim & \, \left\Vert \mathcal{F} \right\Vert_{\mathbb{L}^{\infty}(\Omega)}  \, \left\vert {\bf P_0} \right\vert \,  \int_{\Omega \setminus \overset{[d^{-3}]}{\underset{j=1}\cup} \Omega_{j}} \frac{1}{\left\vert z_{m} - y  \right\vert^{3}} \, dy  \\
    \sum_{m=1}^{[d^{-3}]} \left\vert \Gamma_{3,m} \right\vert^{2}  & \lesssim & \, \left\Vert \mathcal{F} \right\Vert^{2}_{\mathbb{L}^{\infty}(\Omega)}  \, \left\vert {\bf P_0} \right\vert^{2} \, \sum_{m=1}^{[d^{-3}]} \left\vert  \int_{\Omega \setminus \overset{[d^{-3}]}{\underset{j=1}\cup} \Omega_{j}} \frac{1}{\left\vert z_{m} - y  \right\vert^{3}} \, dy  \right\vert^{2}.
\end{eqnarray*}
As a result of \cite[Lemma 6.2]{cao2023all}, we generate the following estimation 
\begin{equation}\label{Eq0434}
         \sum_{m=1}^{[d^{-3}]} \left\vert \Gamma_{3,m} \right\vert^{2} \, = \,  \mathcal{O}\left( \left\Vert \mathcal{F} \right\Vert^{2}_{\mathbb{L}^{\infty}(\Omega)}  \, \left\vert {\bf P_0} \right\vert^{2} \, d^{-1} \right).
\end{equation}
\item[]
\item Estimation of 
\begin{eqnarray*}
         \Gamma_{4,m} &:=&  \pm  \, \xi \, \sum_{j=1 \atop j \neq m}^{[d^{-3}]} \, \frac{1}{\left\vert \Omega_{m} \right\vert} \, \int_{\Omega_{m}} \int_{\Omega_{j}} \, \left( \C\left({\bf P_0}\mathbb{I}\right)_{k}(x,y) - \C_{mj}\left({\bf P_0}\mathbb{I}\right)_{k} \right) \cdot \mathcal{F}(y) \, dy \, dx \\
         \Gamma_{4,m} & \overset{(\ref{Re344})}{ \underset{(\ref{TensorCmj})}{=}} & \pm  \, \xi \, \sum_{j=1 \atop j \neq m}^{[d^{-3}]} \, \frac{1}{\left\vert \Omega_{m} \right\vert} \, \int_{\Omega_{m}} \int_{\Omega_{j}} \, \begin{pmatrix}
            \boldsymbol{\varkappa}_{m_{1},j_{1}} & \cdots & \boldsymbol{\varkappa}_{m_{1},j_{K+1}} \\
            \vdots & \ddots & \vdots \\
            \boldsymbol{\varkappa}_{m_{K+1},j_{K+1}} & \cdots & \boldsymbol{\varkappa}_{m_{K+1},j_{K+1}}
         \end{pmatrix} 
         \cdot \mathcal{F}(y) \, dy \, dx, 
\end{eqnarray*}
where 
\begin{equation*}
    \boldsymbol{\varkappa}_{m_{p},j_{q}} \, := \, \left( \Upsilon_{k}(x,y) - \Upsilon_{k}(z_{m_{p}},z_{j_{q}}) \right) \cdot {\bf P_0}, \quad 1 \leq p , q \leq K+1.
\end{equation*}
Then, 
\begin{equation*}
  \left\vert \Gamma_{4,m} \right\vert  \lesssim  \left\Vert \mathcal{F} \right\Vert_{\mathbb{L}^{\infty}(\Omega)} \, \left\vert {\bf P_0} \right\vert \, \sum_{j=1 \atop j \neq m}^{[d^{-3}]} \, \frac{1}{\left\vert \Omega_{m} \right\vert} \, \int_{\Omega_{m}} \int_{\Omega_{j}} \, \left[ \sum_{p=1}^{K+1} \sum_{q=1}^{K+1}   
             \left\vert \Upsilon_{k}(x,y) - \Upsilon_{k}(z_{m_{p}},z_{j_{q}}) \right\vert^{2}   \right]^{\frac{1}{2}} \, dy \, dx. 
\end{equation*}
Using Taylor expansion, we derive 
\begin{eqnarray*}
    \left\vert \Gamma_{4,m} \right\vert & \lesssim & \left\Vert \mathcal{F} \right\Vert_{\mathbb{L}^{\infty}(\Omega)} \, \left\vert {\bf P_0} \right\vert \, d^{3} \, \sum_{j=1 \atop j \neq m}^{[d^{-3}]} \, \frac{1}{\left\vert \Omega_{m} \right\vert} \, \int_{\Omega_{m}}  \, \left[ \sum_{p=1}^{K+1} \sum_{q=1}^{K+1}   
             \left\vert \int_{0}^{1} \nabla \Upsilon_{k}(z_{m_{p}}+t(x-z_{m_{p}}),z_{j_{q}}) \cdot (x - z_{m_{p}}) dt \right\vert^{2}   \right]^{\frac{1}{2}}  \, dx \\
    & \lesssim & \left\Vert \mathcal{F} \right\Vert_{\mathbb{L}^{\infty}(\Omega)} \, \left\vert {\bf P_0} \right\vert \, d^{4} \, \sum_{j=1 \atop j \neq m}^{[d^{-3}]} \, \frac{1}{\left\vert z_{m} - z_{j} \right\vert^{4}} \\
    & \lesssim & \left\Vert \mathcal{F} \right\Vert_{\mathbb{L}^{\infty}(\Omega)} \, \left\vert {\bf P_0} \right\vert \, d^{4} \, \left( \sum_{j=1 \atop j \neq m}^{\aleph_{1}} \, \frac{1}{\left\vert z_{m} - z_{j} \right\vert^{4}} \, + \, \sum_{j=1 \atop j \neq m}^{\aleph_{2}} \, \frac{1}{\left\vert z_{m} - z_{j} \right\vert^{4}} \right),
\end{eqnarray*}
where $\aleph_{1}$ denotes the number of subdomains $\Omega_{j}$ such that $\Omega_{j} \cap \Omega_{0}^{m} \neq \{ \emptyset \}$, with $\Omega_{0}^{m}$ being a small neighborhood of $\Omega_{m}$, such that there exists $\beta \in ] 0, 1[$ satisfying 
\begin{equation*}
   \left\vert \Omega_{0}^{m} \right\vert \, = \, \mathcal{O}\left( d^{3 \beta} \right), \quad \diam\left( \Omega_{0}^{m} \right) \, = \, \mathcal{O}\left( d^{\beta} \right), 
\end{equation*}
and $\aleph_{2}$ denotes the number of subdomains $\Omega_{j}$ such that $\Omega_{j} \cap \Omega_{0}^{m} = \{ \emptyset \}$, see  \cite[Section 4]{cao2023all} for more details. Then,
\begin{eqnarray*}
        \left\vert \Gamma_{4,m} \right\vert & \lesssim & \left\Vert \mathcal{F} \right\Vert_{\mathbb{L}^{\infty}(\Omega)} \, \left\vert {\bf P_0} \right\vert \, d^{4} \, \left( \sum_{j=1 \atop j \neq m}^{\aleph_{1}} \, \frac{1}{\left\vert z_{m} - z_{j} \right\vert^{4}} \, + \, d^{-4 \beta} \right) \\
        \sum_{m=1}^{[d^{-3}]} \left\vert \Gamma_{4,m} \right\vert^{2} & \lesssim & \left\Vert \mathcal{F} \right\Vert^{2}_{\mathbb{L}^{\infty}(\Omega)} \, \left\vert {\bf P_0} \right\vert^{2} \, d^{8} \, \left( \aleph_{1} \, \sum_{j=1 \atop j \neq m}^{\aleph_{1}} \, \sum_{m=1}^{[d^{-3}]} \frac{1}{\left\vert z_{m} - z_{j} \right\vert^{8}} \, + \, d^{- 8 \beta - 3} \right) \\
        & = & \mathcal{O}\left( \left\Vert \mathcal{F} \right\Vert^{2}_{\mathbb{L}^{\infty}(\Omega)} \, \left\vert {\bf P_0} \right\vert^{2} \, d^{8} \, \left( \aleph_{1}^{2} \, d^{-8} \, + \, d^{- 8 \beta - 3} \right) \right).
\end{eqnarray*}
Knowing that $\aleph_{1} \, = \, \mathcal{O}\left( d^{3 \beta - 3} \right)$, see \cite[Section 4]{cao2023all}, we obtain
\begin{equation*}
       \sum_{m=1}^{[d^{-3}]} \left\vert \Gamma_{4,m} \right\vert^{2} 
         =  \mathcal{O}\left( \left\Vert \mathcal{F} \right\Vert^{2}_{\mathbb{L}^{\infty}(\Omega)} \, \left\vert {\bf P_0} \right\vert^{2}  \, \left( d^{6(\beta - 1)}  \, + \, d^{5 - 8 \beta} \right) \right).
\end{equation*}
To match the exponents of the parameter $d$ in the previous estimate, we select $\beta = \frac{11}{14}$. Thus,
\begin{equation}\label{Eq0450}
       \sum_{m=1}^{[d^{-3}]} \left\vert \Gamma_{4,m} \right\vert^{2} 
         =  \mathcal{O}\left( \left\Vert \mathcal{F} \right\Vert^{2}_{\mathbb{L}^{\infty}(\Omega)} \, \left\vert {\bf P_0} \right\vert^{2}  \, d^{-\frac{9}{7}} \right).
\end{equation}
\item[]
\end{enumerate}
Finally, by using $(\ref{Eq0506}), (\ref{Eq0427}), (\ref{Eq0434})$ and $ (\ref{Eq0450})$, the error term $\Gamma_{m}$, defined by $(\ref{ExpressionGamma})$, satisfies 
\begin{equation}\label{ZM}
        \sum_{m=1}^{[d^{-3}]} \left\vert \Gamma_{m} \right\vert^{2} \, = \, \mathcal{O}\left(   d^{-\frac{9}{7}} \, \left\Vert \mathcal{F} \right\Vert^{2}_{\mathbb{L}^{\infty}(\overline{\Omega})} \, \left\vert {\bf P_0} \right\vert^{2} + d^{2\alpha-3}  \,\left[\mathcal{F} \right]^{2}_{C^{0,\alpha}(\overline{\Omega})} \, \left\vert {\bf P_0} \right\vert^{2}  \right).
\end{equation}
Consequently, the difference between the discretized version associated to the general integral equation, see for instance $(\ref{G.I.E})$, given by the algebraic system $(\ref{Eq1037})$  and the general algebraic system given by $(\ref{al-dis1})$ is accorded by 
\begin{equation}\label{Eq1201}
    \mathbb{T}_{0} \cdot \left( \widetilde{\mathbb{U}}_m \, - \, \mathbb{U}_m \right) \, 
       - \, \pm  \, \xi \, d^{3} \, \sum_{j=1 \atop j \neq m}^{[d^{-3}]}  \,   \C_{mj}\left({\bf P_0}\mathbb{I}\right)_{k} \cdot \left( \widetilde{\mathbb{U}}_j - \mathbb{U}_j \right) \, =  \, \Gamma_{m}^{\star}, 
\end{equation}
with 
\begin{equation}\label{DefGammastar}
    \Gamma_{m}^{\star} \, := \, \Gamma_{m} \, - \, Error_{m}^{(1)},
\end{equation}
where $\left\{ \Gamma_{m} \right\}_{m=1}^{[d^{-3}]}$ is defined by $(\ref{ExpressionGamma})$ and $\left\{ Error_{m}^{(1)} \right\}_{m=1}^{[d^{-3}]}$ is defined by $(\ref{notation-tensor})$. Therefore, we have 
\begin{eqnarray*}
   \sum_{m=1}^{[d^{-3}]} \left\vert \Gamma_{m}^{\star} \right\vert^{2} \, & \lesssim & \,\sum_{m=1}^{[d^{-3}]} \left\vert \Gamma_{m} \right\vert^{2} \, + \,    \sum_{m=1}^{[d^{-3}]} \left\vert Error_{m}^{(1)} \right\vert^{2} \\
   & \overset{(\ref{notation-tensor})}{\leq} & \,\sum_{m=1}^{[d^{-3}]} \left\vert \Gamma_{m} \right\vert^{2} \, + \,    \sum_{m=1}^{[d^{-3}]} \, \sum_{\ell=1}^{K+1} \, \left\vert Term_{m_{\ell}}^{\star} \right\vert^{2} \\
   & \overset{(\ref{Termmlstar})}{\lesssim} & \,\sum_{m=1}^{[d^{-3}]} \left\vert \Gamma_{m} \right\vert^{2} \, + \,    \sum_{m=1}^{[d^{-3}]} \, \sum_{\ell=1}^{K+1} \, \left\vert Term_{m_{\ell}} \right\vert^{2} \, + \, k^{2} \, \sum_{m=1}^{[d^{-3}]} \, \sum_{\ell=1}^{K+1} \, \left\vert H^{Inc}(z_{m_{\ell}}) \, - \, H^{Inc}(z_{m_{1}}) \right\vert^{2},
\end{eqnarray*}
which, by using Taylor expansion and the definition of the $H^{Inc}(\cdot)$, given by $(\ref{EincHinc})$, gives us
\begin{equation*}
   \sum_{m=1}^{[d^{-3}]} \left\vert \Gamma_{m}^{\star} \right\vert^{2} \,  \lesssim  \, \sum_{m=1}^{[d^{-3}]} \left\vert \Gamma_{m} \right\vert^{2} \, + \,    \sum_{m=1}^{[d^{-3}]} \, \sum_{\ell=1}^{K+1} \, \left\vert Term_{m_{\ell}} \right\vert^{2} \, + \, k^{4} \, \sum_{m=1}^{[d^{-3}]} \, \sum_{\ell=1}^{K+1} \, \left\vert  z_{m_{\ell}} -  z_{m_{1}}  \right\vert^{2}.
\end{equation*}
We estimate the last term on the R.H.S as $\mathcal{O}\left( k^{4} \, d^{-1} \right)$. Then,
\begin{equation*}
   \sum_{m=1}^{[d^{-3}]} \left\vert \Gamma_{m}^{\star} \right\vert^{2} \,  
    \lesssim  \, \sum_{m=1}^{[d^{-3}]} \left\vert \Gamma_{m} \right\vert^{2} \, + \,    \sum_{m=1}^{[d^{-3}]} \, \sum_{\ell=1}^{K+1} \, \left\vert Term_{m_{\ell}} \right\vert^{2} \, + \, k^{4} \, d^{-1}.
\end{equation*}
Now, by using $(\ref{Termml})$, we obtain 
\begin{equation*}
   \sum_{m=1}^{[d^{-3}]} \left\vert \Gamma_{m}^{\star} \right\vert^{2} \,  
     \lesssim  \, \sum_{m=1}^{[d^{-3}]} \left\vert \Gamma_{m} \right\vert^{2} \, + \,   d^{6} \, \left\vert {\bf P_0} \right\vert^{2} \, \sum_{m=1}^{[d^{-3}]} \, \sum_{\ell=1}^{K+1} \,   \sum_{i=1\atop i\neq \ell}^{K+1} \left\vert \left( \Upsilon_{k} \, - \, \Upsilon_{0}\right)(z_{m_\ell}, z_{m_i})  \right\vert^{2} \, \sum_{i=1}^{K+1} \left\vert U_{m_i}  \right\vert^{2} \, + \, k^{4} \, d^{-1}. 
\end{equation*}
Thanks to \cite[Inequality (2.4.22)]{BouzekriThesis}, there holds 
\begin{eqnarray*}
   \sum_{m=1}^{[d^{-3}]} \left\vert \Gamma_{m}^{\star} \right\vert^{2} \,  
    & \lesssim & \, \sum_{m=1}^{[d^{-3}]} \left\vert \Gamma_{m} \right\vert^{2} \, + \,   d^{6} \, \left\vert {\bf P_0} \right\vert^{2} \, k^{4} \, \sum_{m=1}^{[d^{-3}]} \, \sum_{\ell=1}^{K+1} \,   \sum_{i=1\atop i\neq \ell}^{K+1} \frac{1}{\left\vert  z_{m_\ell} - z_{m_i}  \right\vert^{2}}  \, \sum_{i=1}^{K+1} \left\vert U_{m_i}  \right\vert^{2} \, + \, k^{4} \, d^{-1} \\
    & \overset{(\ref{dmin})}{\lesssim} & \, \sum_{m=1}^{[d^{-3}]} \left\vert \Gamma_{m} \right\vert^{2} \, + \,   d^{4} \, \left\vert {\bf P_0} \right\vert^{2} \, k^{4} \, \sum_{m=1}^{[d^{-3}]} \,  \sum_{i=1}^{K+1} \left\vert U_{m_i}  \right\vert^{2} \, + \, k^{4} \, d^{-1}. 
\end{eqnarray*}
For the second term on the R.H.S, by referring to $(\ref{EqLemma26})$ and $(\ref{LR})$, we derive that
\begin{eqnarray*}
   \sum_{m=1}^{[d^{-3}]} \left\vert \Gamma_{m}^{\star} \right\vert^{2} \, 
    & \lesssim & \, \sum_{m=1}^{[d^{-3}]} \left\vert \Gamma_{m} \right\vert^{2} \, + \,   d \, \left\vert {\bf P_0} \right\vert^{2} \, k^{4} \, \left\vert \mathbb{T}_{1}^{-1} \right\vert^{2} \, \left\Vert H^{\mathring{\mu}_r}  \right\Vert^{2}_{\mathbb{L}^{\infty}(\Omega)} \, + \, k^{4} \, d^{-1} \\
    & \lesssim & \, \sum_{m=1}^{[d^{-3}]} \left\vert \Gamma_{m} \right\vert^{2} \, +  \, k^{4} \, d^{-1} 
     \overset{(\ref{ZM})}{=}  \, \mathcal{O}\left( d^{-\frac{9}{7}} \, \left\Vert \mathcal{F} \right\Vert^{2}_{\mathbb{L}^{\infty}(\overline{\Omega})} \, \left\vert {\bf P_0} \right\vert^{2} + d^{2\alpha-3}  \,\left[ \mathcal{F} \right]^{2}_{C^{0,\alpha}(\overline{\Omega})} \, \left\vert {\bf P_0} \right\vert^{2} \right). 
\end{eqnarray*}
Lastly, by the use of $(\ref{LR})$, we come up with the following estimation  
\begin{equation}\label{Eq0630}
       \sum_{m=1}^{[d^{-3}]} \left\vert \Gamma_{m}^{\star} \right\vert^{2} \, = \, \mathcal{O}\left( d^{-\frac{9}{7}} \, \left\vert \mathbb{T}_{1}^{-1} \right\vert^{2} \, \left\Vert H^{\mathring{\mu}_r}  \right\Vert^{2}_{\mathbb{L}^{\infty}(\Omega)}  \, \left\vert {\bf P_0} \right\vert^{2} + d^{2\alpha-3}  \, \left\vert \mathbb{T}_{1}^{-1} \right\vert^{2} \,  \left[ H^{\mathring{\mu}_r} \right]^{2}_{C^{0,\alpha}(\overline{\Omega})} \, \left\vert {\bf P_0} \right\vert^{2} \right).
\end{equation}
In conclusion, the difference between the general integral equation given by $(\ref{G.I.E})$ and the general algebraic system given by $(\ref{al-dis1})$ is estimated by $(\ref{Eq0630})$. This implies, by using $(\ref{Eq1201})$ and $(\ref{Eq0630})$, 
\begin{equation}\label{Eq1158}
     \sum_{m=1}^{[d^{-3}]} \left\Vert \widetilde{\mathbb{U}}_{m} \, - \, \mathbb{U}_{m} \right\Vert^{2}_{\ell^2} \, = \, \mathcal{O}\left( d^{-\frac{9}{7}} \, \left\vert \mathbb{T}_{1}^{-1} \right\vert^{2} \, \left\Vert H^{\mathring{\mu}_r}  \right\Vert^{2}_{\mathbb{L}^{\infty}(\Omega)}  \, \left\vert {\bf P_0} \right\vert^{2} + d^{2\alpha-3}  \, \left\vert \mathbb{T}_{1}^{-1} \right\vert^{2} \,  \left[ H^{\mathring{\mu}_r} \right]^{2}_{C^{0,\alpha}(\overline{\Omega})} \, \left\vert {\bf P_0} \right\vert^{2} \right).
\end{equation}
\medskip
\begin{remark}\label{RemarkCut}
Two points are worth making. 
\begin{enumerate}
    \item[]
    \item If the choice is made to cut $\Omega$ into $\aleph$ cubes according to \textbf{Assumptions $(\ref{assII})$}, i.e. each $\Omega_{m}$ is a cube, $\mathbb{T}_{1}$'s expression given by $(\ref{DefT1})$ will alter,    
    \begin{equation*}
    \mathbb{T}_1 \, = \, \mathbb{T}_0 \, - \, \pm  \, \frac{1}{3} \, \xi \, 
     \begin{pmatrix}
        {\bf P_0} & \cdots &  {\bf P_0} \\
         \vdots & \ddots & \vdots \\
         {\bf P_0} & \cdots &   {\bf P_0}
     \end{pmatrix},
\end{equation*}
where the tensor $\mathbb{T}_0$ is defined by $(\ref{def-T0-matrix})$ and the tensor ${\bf P_0}$ is given by $(\ref{defP0})$. This is due to the fact that when $\Omega_{0}$ is a cube, we have 
\begin{equation*}
    \nabla {\bf M}_{\Omega_{0}}\left( I_{3} \right)(x) \, = \, \frac{1}{3} \, I_{3}, \quad \text{for} \quad x \in \Omega_{0}, 
\end{equation*}
see \textbf{Subsection \ref{AddSS}}.
    \item[]
    \item Since the tensor $\mathcal{A}$ is defined through the tensor $\mathbb{T}_{1}$, see $(\ref{SKM})$, the above remark applies to $\mathcal{A}$. 
\end{enumerate}
\end{remark}
\subsection{The generated effective medium}
Our starting equation is $(\ref{approximation-E})$, which by using the fact that $\aleph = [d^{-3}]$ and $\eta \, = \, \eta_{0} \, a^{-2}$, the formulas $(\ref{def-xi})$ and $(\ref{distribute})$, and the relation given by $(\ref{vari-substi})$, gives us the following relation
\begin{equation}\label{far1}
	E^\infty(\hat x) \, =  \, - \frac{i\, k}{\pm\, 4\, \pi} \, \xi \, d^{3} \, \sum_{n=1}^{M} e^{-ik\hat{x}\cdot z_n}\hat{x}\times {\bf P_0} \cdot U_n \, + \, \Oh(a^{\frac{h}{3}}).
\end{equation}
Additionally, under the distribution assumption, see \textbf{Assumptions} (2), we obtain 
\begin{equation}\label{KTMB}
    E^\infty(\hat{x}) \,  =  \, - \frac{i\, k}{\pm\, 4\, \pi} \, \xi \, d^{3} \,  \sum_{m=1}^{[d^{-3}]}\sum_{\ell=1}^{K+1}e^{- i k \hat{x}\cdot z_{m_\ell}}\hat{x}\times {\bf P}_0 \cdot U_{m_\ell} \,+ \, \Oh(a^{\frac{h}{3}}). 
\end{equation}
Now, by using the fact that $\left\vert z_{m_\ell} \, - \, z_{m_1} \right\vert$ is small, we get 
    \begin{eqnarray}\label{far-dis1}
    \nonumber
	    E^\infty(\hat{x}) &=& - \frac{i\, k}{\pm\, 4\, \pi} \, \xi \, d^{3} \, \sum_{m=1}^{[d^{-3}]}e^{- i k \hat{x} \cdot z_{m_1}}\hat{x}\times {\bf P}_0 \cdot \left(\sum_{\ell=1}^{K+1}U_{m_\ell}\right) \\
		&+&  - \frac{i\, k}{\pm\, 4\, \pi} \, \xi \, d^{3} \, \sum_{m=1}^{[d^{-3}]}\sum_{\ell=1}^{K+1} e^{- \, i \, k \hat{x}\cdot z_{m_1}} \left( \sum_{n = 1}^{\infty} \frac{\left(- i k \hat{x}\cdot(z_{m_\ell}- z_{m_1})\right)^n}{n!} \right) \, \hat{x} \times {\bf P}_0 \cdot U_{m_\ell} \, + \, \Oh(a^{\frac{h}{3}}).
	\end{eqnarray}
 Next, we estimate the second term on the R.H.S. By denoting 
 \begin{equation*}
      Residuum \, := \,  - \frac{i\, k}{\pm\, 4\, \pi} \, \xi \, d^{3} \, \sum_{m=1}^{[d^{-3}]}\sum_{\ell=1}^{K+1} e^{- \, i \, k \hat{x}\cdot z_{m_1}} \left( \sum_{n = 1}^{\infty} \frac{\left(- i k \hat{x}\cdot(z_{m_\ell}- z_{m_1})\right)^n}{n!} \right) \, \hat{x} \times {\bf P}_0 \cdot U_{m_\ell},
\end{equation*}
we have,
\begin{eqnarray*}
     \left\vert Residuum \right\vert & = & \left\vert - \frac{i\, k}{\pm\, 4\, \pi} \, \xi \, d^{3} \, \sum_{m=1}^{[d^{-3}]}\sum_{\ell=1}^{K+1} e^{- \, i \, k \hat{x}\cdot z_{m_1}} \left( \sum_{n = 1}^{\infty} \frac{\left(- i k \hat{x}\cdot(z_{m_\ell}- z_{m_1})\right)^n}{n!} \right) \, \hat{x} \times {\bf P}_0 \cdot U_{m_\ell} \right\vert \\
     & \lesssim &  k  \, d^{3} \, \left\vert {\bf P}_0 \right\vert \sum_{m=1}^{[d^{-3}]}\sum_{\ell=1}^{K+1}  \left\vert \sum_{n = 1}^{\infty} \frac{\left(- i k \hat{x}\cdot(z_{m_\ell}- z_{m_1})\right)^n}{n!} \right\vert \, \left\vert U_{m_\ell} \right\vert \\
    & \lesssim &  \, k^{2}  \, d^{3} \, \left\vert {\bf P}_0 \right\vert \sum_{m=1}^{[d^{-3}]}\sum_{\ell=1}^{K+1} \left\vert z_{m_\ell} - z_{m_1} \right\vert  \, \left\vert U_{m_\ell} \right\vert \\
    & \lesssim &  k^{2}  \, d^{3} \, \left\vert {\bf P}_0 \right\vert \,  \left( \sum_{m=1}^{[d^{-3}]} \, \sum_{\ell=1}^{K+1} \left\vert z_{m_\ell} - z_{m_1} \right\vert^{2} \right)^{\frac{1}{2}} \, \left( \sum_{m=1}^{[d^{-3}]} \, \sum_{\ell=1}^{K+1} \left\vert U_{m_\ell} \right\vert^{2} \right)^{\frac{1}{2}} \\
    & \lesssim &  k^{2}  \, d^{\frac{5}{2}} \, \left\vert {\bf P}_0 \right\vert  \, \left( \sum_{m=1}^{[d^{-3}]} \, \sum_{\ell=1}^{K+1} \left\vert U_{m_\ell} \right\vert^{2} \right)^{\frac{1}{2}}.
 \end{eqnarray*}
 As ${\bf P}_0$ is of order one, with respect to the parameter $d$, we end up with the following estimation 
 \begin{equation*}
         \left\vert Residuum \right\vert \; = \; \mathcal{O}\left( k^{2} \, d^{\frac{5}{2}}   \, \left( \sum_{m=1}^{[d^{-3}]} \, \sum_{\ell=1}^{K+1} \left\vert U_{m_\ell} \right\vert^{2} \right)^{\frac{1}{2}} \right).
 \end{equation*}
 Defining $\Psi$ as 
\begin{equation}\label{DefPsi}
 \Psi \; := \; \sum_{m=1}^{[d^{-3}]} \, \sum_{\ell=1}^{K+1} \left\vert U_{m_\ell} \right\vert^{2},
\end{equation}
reduces the length of our formulas. Hence, $(\ref{far-dis1})$ becomes, 
 \begin{equation*}
	    E^\infty(\hat{x}) \, = \, - \frac{i\, k}{\pm\, 4\, \pi} \, \xi \, d^{3} \, \sum_{m=1}^{[d^{-3}]}e^{- i k \hat{x} \cdot z_{m_1}}\hat{x}\times {\bf P}_0 \cdot \left( \sum_{\ell=1}^{K+1} U_{m_\ell} \right) 
		+  \mathcal{O}\left( k^{2} \, d^{\frac{5}{2}}   \, \Psi^{\frac{1}{2}} \right) \, + \, \Oh(a^{\frac{h}{3}}),
\end{equation*}  
which is an alternative way to rewrite $(\ref{far1})$ using the unknown $\overset{K+1}{\underset{\ell=1}\sum} U_{m_\ell}$ instead of the unknown $\mathbb{U}_{m}$. The next step will give a better understanding of this transformation. Based on the previous equation,
\begin{equation}\label{Eq235}
  E^\infty(\hat{x}) \, \overset{(\ref{notation-tensor})}{=} - \frac{i\, k}{\pm\, 4\, \pi}  \xi  d^{3} \, \sum_{m=1}^{[d^{-3}]}e^{- i k \hat{x} \cdot z_{m_1}} \hat{x} \times \left( {\bf P}_0, \cdots , {\bf P}_0  \right) \cdot \mathbb{U}_m 
		+  \mathcal{O}\left( k^{2} \, d^{\frac{5}{2}}   \, \Psi^{\frac{1}{2}} \right) \, + \, \mathcal{O}\left( a^{\frac{h}{3}} \right),
\end{equation}
where $\mathbb{U}_m$ is the solution of the general algebraic system (G.A.S) given by $(\ref{al-dis1})$. In addition, by referring to \cite[Section 2.3]{cao2023all}, we can derive that
   \begin{eqnarray}\label{ASEq2}
  \nonumber
        E^{\infty}_{eff, +}(\hat{x}) \, &=&  - \, \pm \, \frac{i \, k}{4 \, \pi} \, \xi  \, \sum_{m = 1}^{[d^{-3}]} \, \int_{ \Omega_{m}}  \, e^{- \, i \, k \, \hat{x} \cdot z}  \, \hat{x} \, \times  {\bf P_0} \cdot \left( \sum_{\ell=1}^{K+1} F_{\ell}(z) \, - \frac{1}{\left\vert \Omega_{m} \right\vert} \, \int_{\Omega_{m}} \sum_{\ell=1}^{K+1} F_{\ell}(y) \, dy \right) \, dz \, \\ \nonumber &-& \, \pm \, \frac{i \, k}{4 \, \pi} \, \xi  \, \sum_{m = 1}^{[d^{-3}]} \, \int_{ \Omega_{m}}  \, e^{- \, i \, k \, \hat{x} \cdot z}  \, \hat{x} \, \times  \, {\bf P_0} \cdot \left(  \frac{1}{\left\vert \Omega_{m} \right\vert} \, \int_{\Omega_{m}} \sum_{\ell=1}^{K+1} F_{\ell}(y) \, dy \right) \, dz \\ &+& \, \mathcal{O}\left( k \, d  \, \left\vert  {\bf P_0}  \right\vert \, \left\Vert  \sum_{\ell=1}^{K+1} F_{\ell}(\cdot) \right\Vert_{\mathbb{L}^{\infty}(\Omega)} \right). 
\end{eqnarray}
In this manner, we express the discrete form of the far-field formula by the solution to $(\ref{al-dis1})$, thanks to the transformation $(\ref{far1})$-$(\ref{Eq235})$, and present the integral form of the far-field, corresponding to the solution to the effective medium $(\ref{G.I.E})$. It is worth mentioning that we have already calculated the difference between the general algebraic system (G.A.S), given by $(\ref{al-dis1})$, and the general integral equation (G.I.E), given by $(\ref{G.I.E})$. Furthermore, by subtracting $(\ref{ASEq2})$ from $(\ref{Eq235})$, and using $(\ref{distribute})$, we get 
   \begin{eqnarray}\label{1644}
  \nonumber
       \gimel(\hat{x}) &:=& \left( E^{\infty}_{eff, +}-E^{\infty} \right)(\hat{x}) \\ \nonumber &=& -\pm \, \frac{i \, k}{4 \, \pi} \, \xi  \, \sum_{m = 1}^{[d^{-3}]} \, \int_{ \Omega_{m}}  \, e^{- \, i \, k \, \hat{x} \cdot z}  \, \hat{x} \, \times  \, {\bf P_0} \cdot \left( \sum_{\ell=1}^{K+1} F_{\ell}(z) \, - \frac{1}{\left\vert \Omega_{m} \right\vert} \, \int_{\Omega_{m}} \sum_{\ell=1}^{K+1} F_{\ell}(y) \, dy \right) \, dz \, \\ \nonumber &- &  \pm \, \frac{i \, k}{4 \, \pi} \, \xi  \, \sum_{m = 1}^{[d^{-3}]} \, \int_{ \Omega_{m}} \hat{x} \, \times \Bigg[  e^{- \, i \, k \, \hat{x} \cdot z}  \,   \, {\bf P_0} \cdot \left(  \frac{1}{\left\vert \Omega_{m} \right\vert} \, \int_{\Omega_{m}}  \sum_{\ell=1}^{K+1} F_{\ell}(y) \, dy \right) - e^{- \, i \, k \, \hat{x} \cdot z_{m_{1}}}  \left( {\bf P_0}, \cdots,  {\bf P_0} \right) \cdot \mathbb{U}_{m} \Bigg] \, dz \\ &+& \mathcal{O}\left(k^{2} \, d^{\frac{5}{2}} \, \Psi^{\frac{1}{2}} \right) + \mathcal{O}\left( a^{\frac{h}{3}}\right)  + \mathcal{O}\left( k \, d \, \left\vert {\bf P_0} \right\vert \, \left\Vert \sum_{\ell=1}^{K+1} F_{\ell}(\cdot) \right\Vert_{\mathbb{L}^{\infty}(\Omega)} \right). 
\end{eqnarray}
Then, we set and estimate the first term on the R.H.S of $(\ref{1644})$.
\begin{eqnarray*}
    J_{1}(\hat{x}) \, &:=& - \, \pm \, \frac{i \, k}{4 \, \pi} \, \xi  \, \sum_{m = 1}^{[d^{-3}]} \, \int_{ \Omega_{m}}  \, e^{- \, i \, k \, \hat{x} \cdot z}  \, \hat{x} \, \times  \, {\bf P_0} \cdot \left( \sum_{\ell=1}^{K+1} F_{\ell}(z) \, - \frac{1}{\left\vert \Omega_{m} \right\vert} \, \int_{\Omega_{m}} \sum_{\ell=1}^{K+1} F_{\ell}(y) \, dy \right) \, dz  \\
    \left\vert J_{1}(\hat{x}) \right\vert \, & \lesssim & k  \, \left[ \sum_{\ell=1}^{K+1} F_{\ell}(\cdot) \right]_{C^{0,\alpha}(\overline{\Omega})} \, \left\vert {\bf P_0} \right\vert \, \sum_{m = 1}^{[d^{-3}]} \, \int_{ \Omega_{m}}  \, \left\vert z - z_{m_{1}} \right\vert^{\alpha} \, dz \\
    \left\Vert J_{1}(\cdot) \right\Vert_{\mathbb{L}^{\infty}(\Omega)} &=& \mathcal{O}\left(k \, d^{\alpha} \, \left\vert {\bf P_0} \right\vert \, \left[ \sum_{\ell=1}^{K+1} F_{\ell}(\cdot) \right]_{C^{0,\alpha}(\overline{\Omega})} 
  \right).
\end{eqnarray*}
Hence, the formula $(\ref{1644})$ becomes, 
  \begin{eqnarray*}
       \gimel(\hat{x}) &=& - \pm \frac{i \, k}{4 \, \pi} \xi  \sum_{m = 1}^{[d^{-3}]}  \int_{ \Omega_{m}} \hat{x}  \times \left[  e^{- \, i \, k \, \hat{x} \cdot z}   {\bf P_0} \cdot \left(  \frac{1}{\left\vert \Omega_{m} \right\vert} \int_{\Omega_{m}}  \sum_{\ell=1}^{K+1} F_{\ell}(y) \, dy \right) - e^{- \, i \, k \, \hat{x} \cdot z_{m_{1}}}   \left( {\bf P_0}, \cdots,  {\bf P_0} \right) \cdot \mathbb{U}_{m} \right]  dz \\ &+& \mathcal{O}\left(k^{2} \, d^{\frac{5}{2}} \, \Psi^{\frac{1}{2}} \right) + \mathcal{O}\left( a^{\frac{h}{3}}\right)   +  \mathcal{O}\left(k \, d^{\alpha} \, \left\vert {\bf P_0} \right\vert \, \left[ \sum_{\ell=1}^{K+1} F_{\ell}(\cdot) \right]_{C^{0,\alpha}(\overline{\Omega})}
  \right) 
\end{eqnarray*}
and by Taylor expansion, we get 
  \begin{eqnarray}\label{Eq1709}
  \nonumber
       \gimel(\hat{x}) \, &=& - \, \pm \, \frac{i \, k}{4 \, \pi} \, \xi  \, \sum_{m = 1}^{[d^{-3}]} \, e^{- \, i \, k \, \hat{x} \cdot z_{m_{1}}} \, \left\vert \Omega_{m} \right\vert \hat{x} \, \times  \left[ {\bf P_0} \cdot \left(  \frac{1}{\left\vert S_{m} \right\vert} \, \int_{S_{m}}  \sum_{\ell=1}^{K+1} F_{\ell}(y) \, dy \right) - \left( {\bf P_0}, \cdots,  {\bf P_0} \right) \cdot \mathbb{U}_{m} \right] \, \\ \nonumber
       &-&  \, \pm \, \frac{k^{2}}{4 \, \pi} \, \xi  \, \sum_{m = 1}^{[d^{-3}]} \, \int_{ \Omega_{m}}  \rho(k,\hat{x},z) \, \hat{x} \, \times {\bf P_0} \cdot \left(  \frac{1}{\left\vert S_{m} \right\vert} \, \int_{S_{m}}  \sum_{\ell=1}^{K+1} F_{\ell}(y) \, dy \right)  \, dz \\ &+& \mathcal{O}\left(k^{2} \, d^{\frac{5}{2}} \, \Psi^{\frac{1}{2}} \right) + \mathcal{O}\left( a^{\frac{h}{3}}\right)   +  \mathcal{O}\left(k \, d^{\alpha} \, \left\vert {\bf P_0} \right\vert \, \left[ \sum_{\ell=1}^{K+1} F_{\ell}(\cdot) \right]_{C^{0,\alpha}(\overline{\Omega})} 
  \right), 
\end{eqnarray}
where 
\begin{equation*}
   \rho(k,\hat{x},z) \, := \, \int_{0}^{1}  e^{- \, i \, k \, \hat{x} \cdot (z_{m_{1}}+t(z-z_{m_{1})})}  \, \hat{x} \cdot (z - z_{m_{1}}) \, dt.
\end{equation*}
We set and we estimate the second term on the R.H.S. 
\begin{eqnarray*}
    J_{2}(\hat{x}) \, &:=& - \, \pm \, \frac{k^{2}}{4 \, \pi} \, \xi  \, \sum_{m = 1}^{[d^{-3}]} \, \int_{ \Omega_{m}} \, \rho(k,\hat{x},z) \, \hat{x} \, \times  \, {\bf P_0} \cdot \left(  \frac{1}{\left\vert S_{m} \right\vert} \, \int_{S_{m}}  \sum_{\ell=1}^{K+1} F_{\ell}(y) \, dy \right)  \, dz \\
    \left\vert J_{2}(\hat{x}) \right\vert \, & \lesssim & \, k^{2}  \, \left\vert {\bf P_0} \right\vert \left\Vert  \sum_{\ell=1}^{K+1} F_{\ell}(\cdot) \right\Vert_{\mathbb{L}^{\infty}(\Omega)}\, \, \sum_{m = 1}^{[d^{-3}]} \, \int_{ \Omega_{m}}  \left\vert z - z_{m_{1}} \right\vert \,  dz \\
    \left\Vert J_{2}(\cdot) \right\Vert_{\mathbb{L}^{\infty}(\Omega)} \, & = & \mathcal{O}\left(  k^{2}  \, \left\vert {\bf P_0} \right\vert \left\Vert  \sum_{\ell=1}^{K+1} F_{\ell}(\cdot) \right\Vert_{\mathbb{L}^{\infty}(\Omega)} \,  d \right).
\end{eqnarray*}
Hence, the formula $(\ref{Eq1709})$ becomes, 
  \begin{eqnarray}\label{Eq0645}
  \nonumber
       \gimel(\hat{x}) \, &=& - \, \pm \, \frac{i \, k}{4 \, \pi} \, \xi  \, \sum_{m = 1}^{[d^{-3}]} \, e^{- \, i \, k \, \hat{x} \cdot z_{m_{1}}} \, \left\vert \Omega_{m} \right\vert \hat{x} \, \times  \left[ {\bf P_0} \cdot \left(  \frac{1}{\left\vert S_{m} \right\vert} \, \int_{S_{m}}  \sum_{\ell=1}^{K+1} F_{\ell}(y) \, dy \right) - \left( {\bf P_0}, \cdots,  {\bf P_0} \right) \cdot \mathbb{U}_{m} \right] \, \\  &+& \mathcal{O}\left(k^{2} \, d^{\frac{5}{2}} \, \Psi^{\frac{1}{2}} \right) + \mathcal{O}\left( a^{\frac{h}{3}}\right)   +  \mathcal{O}\left(k \, d^{\alpha} \, \left\vert {\bf P_0} \right\vert \, \left[ \sum_{\ell=1}^{K+1} F_{\ell}(\cdot) \right]_{C^{0,\alpha}(\overline{\Omega})} 
  \right). 
\end{eqnarray}
Now, we are in a position to analyze the following term appearing on the R.H.S of the above equation
\begin{equation*}
    J_{m} :=  {\bf P_0} \cdot \left(  \frac{1}{\left\vert S_{m} \right\vert} \, \int_{S_{m}}  \sum_{\ell=1}^{K+1} F_{\ell}(y) \, dy \right) - \left( {\bf P_0}, \cdots,  {\bf P_0} \right) \cdot \mathbb{U}_{m}  
     \underset{\eqref{FM2004}}{\overset{(\ref{Re34})}{=}}  \left( {\bf P_0}; \cdots; {\bf P_0} \right) \cdot \left( \widetilde{\mathbb{U}}_m  -  \mathbb{U}_{m} \right).
\end{equation*}
Consequently, the far-fields estimation derived in $(\ref{Eq0645})$ becomes, 
  \begin{eqnarray*}
  \nonumber
       \gimel(\hat{x}) \, &=& \mp \, \frac{i \, k}{4 \, \pi} \, \xi  \, \sum_{m = 1}^{[d^{-3}]} \, e^{- \, i \, k \, \hat{x} \cdot z_{m_{1}}} \, \left\vert \Omega_{m} \right\vert \hat{x} \, \times   \left( {\bf P_0}, \cdots,  {\bf P_0} \right) \cdot \left( \widetilde{\mathbb{U}}_m  -  \mathbb{U}_{m} \right)\\ &+& \mathcal{O}\left(k^{2} \, d^{\frac{5}{2}} \, \Psi^{\frac{1}{2}} \right) + \mathcal{O}\left( a^{\frac{h}{3}}\right)   +  \mathcal{O}\left(k \, d^{\alpha} \, \left\vert {\bf P_0} \right\vert \, \left[ \sum_{\ell=1}^{K+1} F_{\ell}(\cdot) \right]_{C^{0,\alpha}(\overline{\Omega})} 
  \right). 
\end{eqnarray*}
We have justified that the difference between the general integral equation, given by $(\ref{G.I.E})$, and its associated general algebraic system, given by $(\ref{al-dis1})$, is the term noted by $\Gamma_{m}^{\star}$, see $(\ref{DefGammastar})$ for its definition. Now, by taking the $\left\Vert \cdot \right\Vert_{\mathbb{L}^{\infty}(\mathbb{S}^2)}$-norm on the both sides of the previous equation, we obtain 
  \begin{eqnarray*}
       \left\Vert \gimel(\cdot) \right\Vert_{\mathbb{L}^{\infty}(\mathbb{S}^2)}  \, &=& \mathcal{O}\left(  k \, \left\vert {\bf P_0} \right\vert \, d^{\frac{3}{2}} \, \left( \sum_{m = 1}^{[d^{-3}]} \, \left\vert \Gamma_{m}^{\star} \right\vert^{2} \right)^{\frac{1}{2}} \right) + \mathcal{O}\left(k^{2} \, d^{\frac{5}{2}} \, \Psi^{\frac{1}{2}} \right) + \mathcal{O}\left( a^{\frac{h}{3}}\right) \\ &+&  \mathcal{O}\left(k \, d^{\alpha} \, \left\vert {\bf P_0} \right\vert \, \left[ \sum_{\ell=1}^{K+1} F_{\ell}(\cdot) \right]_{C^{0,\alpha}(\overline{\Omega})} 
  \right). 
\end{eqnarray*}
Moreover, thanks to $(\ref{Eq0630})$ and $(\ref{LR})$, we get  
  \begin{eqnarray*}
       \left\Vert \gimel(\cdot) \right\Vert_{\mathbb{L}^{\infty}(\mathbb{S}^2)}  \, &=& \mathcal{O}\left(  k \, \left\vert {\bf P_0} \right\vert^{2}  \, \left( d^{\frac{12}{7}} \, \left\Vert \mathcal{F} \right\Vert^{2}_{\mathbb{L}^{\infty}(\overline{\Omega})} \,  + d^{2\alpha}  \,\left[  \mathcal{F} \right]^{2}_{C^{0,\alpha}(\overline{\Omega})} \,   \right)^{\frac{1}{2}} \right) \\&+& \mathcal{O}\left(k^{2} \, d^{\frac{5}{2}} \, \Psi^{\frac{1}{2}} \right) + \mathcal{O}\left( a^{\frac{h}{3}}\right)   +  \mathcal{O}\left(k \, d^{\alpha} \, \left\vert {\bf P_0} \right\vert \, \left[ \sum_{\ell=1}^{K+1} F_{\ell}(\cdot) \right]_{C^{0,\alpha}(\overline{\Omega})} 
  \right). 
\end{eqnarray*}
Remark that, based on the definition of $\mathcal{F}(\cdot)$, see $(\ref{Re34})$, we have
\begin{equation*}
 \left[ \overset{K+1}{\underset{\ell=1}\sum} F_{\ell}(\cdot) \right]_{C^{0,\alpha}(\overline{\Omega})} \, \simeq \, \left[ \mathcal{F}(\cdot) \right]_{C^{0,\alpha}(\overline{\Omega})},   
\end{equation*}
 hence  
   \begin{eqnarray}\label{Eeff-E}
   \nonumber
       \left\Vert \gimel(\cdot) \right\Vert_{\mathbb{L}^{\infty}(\mathbb{S}^2)}  \, &=& \mathcal{O}\left(  k \, \left\vert {\bf P_0} \right\vert^{2}  \, \left( d^{\frac{12}{7}} \, \left\Vert \mathcal{F} \right\Vert^{2}_{\mathbb{L}^{\infty}(\overline{\Omega})} \,  + d^{2\alpha}  \,\left[  \mathcal{F} \right]^{2}_{C^{0,\alpha}(\overline{\Omega})} \,   \right)^{\frac{1}{2}} \right) \\ &+& \mathcal{O}\left(k^{2} \, d^{\frac{5}{2}} \, \Psi^{\frac{1}{2}} \right) + \mathcal{O}\left( a^{\frac{h}{3}}\right)   +  \mathcal{O}\left(k \, d^{\alpha} \, \left\vert {\bf P_0} \right\vert \, \left[ \mathcal{F} \right]_{C^{0,\alpha}(\overline{\Omega})} 
  \right). 
\end{eqnarray}
In order to finish the R.H.S estimation, it is necessary to estimate $\Psi$, see for instance $(\ref{DefPsi})$. For this, we have,
\begin{eqnarray*}
    \Psi  \overset{(\ref{DefPsi})}{=}   \,      \sum_{m=1}^{[d^{-3}]} \,  \left\Vert       \mathbb{U}_m \right\Vert^{2}_{\ell^{2}} \; & \lesssim & \; \sum_{m=1}^{[d^{-3}]} \,  \left\Vert  \widetilde{\mathbb{U}}_m \right\Vert^{2}_{\ell^{2}} \; + \; \sum_{m=1}^{[d^{-3}]} \,  \left\Vert       \mathbb{U}_m \, - \, \widetilde{\mathbb{U}}_m \right\Vert^{2}_{\ell^{2}}  \\ & \overset{(\ref{FM2004})}{=} & \sum_{m=1}^{[d^{-3}]} \,  \left\Vert       \frac{1}{\left\vert S_{m} \right\vert} \int_{S_{m}} \mathcal{F}(x) \, dx \right\Vert^{2}_{\ell^{2}} \; + \; \sum_{m=1}^{[d^{-3}]} \,  \left\Vert       \mathbb{U}_m \, - \, \widetilde{\mathbb{U}}_m \right\Vert^{2}_{\ell^{2}} \\ &\leq& \,  \sum_{m=1}^{[d^{-3}]} \, \frac{1}{\left\vert S_{m} \right\vert}         \int_{S_{m}} \left\Vert \mathcal{F}(x) \right\Vert^{2}_{\ell^{2}} \, dx \; + \; \sum_{m=1}^{[d^{-3}]} \,  \left\Vert       \mathbb{U}_m \, - \, \widetilde{\mathbb{U}}_m \right\Vert^{2}_{\ell^{2}} \\ &=& \, \sum_{m=1}^{[d^{-3}]} \, \left\Vert \mathcal{F} \right\Vert^{2}_{\mathbb{L}^{\infty}(\Omega)} \;  + \; \sum_{m=1}^{[d^{-3}]} \,  \left\Vert       \mathbb{U}_m \, - \, \widetilde{\mathbb{U}}_m \right\Vert^{2}_{\ell^{2}} \\
    &=&  d^{-3} \, \left\Vert \mathcal{F} \right\Vert^{2}_{\mathbb{L}^{\infty}(\Omega)}  \;  + \; \sum_{m=1}^{[d^{-3}]} \,  \left\Vert       \mathbb{U}_m \, - \, \widetilde{\mathbb{U}}_m \right\Vert^{2}_{\ell^{2}},
\end{eqnarray*}
which, by using $(\ref{Eq1158})$, gives us 
\begin{eqnarray*}
   \Psi & = & d^{-3} \, \left\Vert \mathcal{F} \right\Vert^{2}_{\mathbb{L}^{\infty}(\Omega)}  \; + \, \mathcal{O}\left( d^{-\frac{9}{7}} \, \left\vert \mathbb{T}_{1}^{-1} \right\vert^{2} \, \left\Vert H^{\mathring{\mu}_r}  \right\Vert^{2}_{\mathbb{L}^{\infty}(\Omega)}  \, \left\vert {\bf P_0} \right\vert^{2} + d^{2\alpha-3}  \, \left\vert \mathbb{T}_{1}^{-1} \right\vert^{2} \,  \left[ H^{\mathring{\mu}_r} \right]^{2}_{C^{0,\alpha}(\overline{\Omega})} \, \left\vert {\bf P_0} \right\vert^{2} \right) \\
   & \overset{(\ref{LR})}{=} & d^{-3} \, \left\Vert \mathcal{F} \right\Vert^{2}_{\mathbb{L}^{\infty}(\Omega)}  \; + \, \mathcal{O}\left( d^{-\frac{9}{7}} \,  \left\Vert \mathcal{F}  \right\Vert^{2}_{\mathbb{L}^{\infty}(\Omega)}  \, \left\vert {\bf P_0} \right\vert^{2} + d^{2\alpha-3}   \,  \left[ \mathcal{F} \right]^{2}_{C^{0,\alpha}(\overline{\Omega})} \, \left\vert {\bf P_0} \right\vert^{2} \right) \\
    & = &  \mathcal{O}\left( d^{-3} \, \left\Vert \mathcal{F} \right\Vert^{2}_{\mathbb{L}^{\infty}(\Omega)} + d^{2\alpha-3}   \,  \left[ \mathcal{F} \right]^{2}_{C^{0,\alpha}(\overline{\Omega})} \, \left\vert {\bf P_0} \right\vert^{2} \right).
\end{eqnarray*}
Hence, 
\begin{equation}\label{EqLemma26}
    \Psi \, 
         = \,  \mathcal{O}\left(d^{-3} \, \left\Vert \mathcal{F} \right\Vert^{2}_{\mathbb{L}^{\infty}(\Omega)} \right).
\end{equation}
Now, thanks to $(\ref{EqLemma26})$, the estimation $(\ref{Eeff-E})$ takes the following form 
\begin{equation*}
       \left\Vert \gimel(\cdot) \right\Vert_{\mathbb{L}^{\infty}(\mathbb{S}^2)} 
       = \mathcal{O}\left(  k \, \left\vert {\bf P_0} \right\vert^{2}  \, \left( d^{\frac{12}{7}} \, \left\Vert \mathcal{F} \right\Vert^{2}_{\mathbb{L}^{\infty}(\overline{\Omega})} \,  + d^{2\alpha}  \,\left[ \mathcal{F} \right]^{2}_{C^{0,\alpha}(\overline{\Omega})} \,   \right)^{\frac{1}{2}} \right) +  \mathcal{O}\left( a^{\frac{h}{3}}\right)   +  \mathcal{O}\left(k \, d^{\alpha} \, \left\vert {\bf P_0} \right\vert \, \left[ \mathcal{F} \right]_{C^{0,\alpha}(\overline{\Omega})} 
  \right). 
\end{equation*}
The last step consists in investigating how $\mathcal{F}(\cdot)$ depends on the magnetic field $H^{\mathring{\mu_{r}}}(\cdot)$. By comparing the expression of the L.S.E derived in \cite[Equation (1.37)]{cao2023all} and the L.S.E associated to our reduced integral equation, see $(\ref{Re16})$, we deduce that the magnetic field admits the following representation
\begin{equation}\label{H=F}
    H^{\mathring{\mu_{r}}}(\cdot) \, = \, \mathcal{A}^{-1} \cdot \sum_{\ell=1}^{K+1} F_{\ell}(\cdot),
\end{equation}
which by using the definition of $\mathcal{F}(\cdot)$, given by $(\ref{Re34})$, and the expression of the local distribution tensor $\mathcal{A}$, given by $(\ref{SKM})$, we end up with the following relation 
\begin{equation*}
    \mathcal{F}(\cdot) \, = \, \mathbb{T}_{1}^{-1} \cdot \begin{pmatrix}
        I_{3} \\
        \vdots \\
        I_{3}
    \end{pmatrix} \cdot H^{\mathring{\mu_{r}}}(\cdot),
\end{equation*}
where $\mathbb{T}_{1}$ is the constant tensor defined by $(\ref{DefT1})$. This allows us to deduce that 
\begin{equation}\label{LR}
  \left[ \mathcal{F} \right]_{C^{0,\alpha}(\overline{\Omega})} = \mathcal{O}\left( \left\vert \mathbb{T}_{1}^{-1} \right\vert \, \left[ H^{\mathring{\mu_{r}}} \right]_{C^{0,\alpha}(\overline{\Omega})}  \right)  \quad \text{and} \quad \left\Vert \mathcal{F} \right\Vert_{\mathbb{L}^{\infty}(\Omega)} = \mathcal{O}\left( \left\vert \mathbb{T}_{1}^{-1} \right\vert \, \left\Vert H^{\mathring{\mu_{r}}} \right\Vert_{\mathbb{L}^{\infty}(\Omega)}  \right).
\end{equation}
Finally, by recalling the definition of $\gimel(\cdot)$, 
\begin{eqnarray*}
       \left\Vert E^{\infty}_{eff, +}  -          E^{\infty} \right\Vert_{\mathbb{L}^{\infty}(\mathbb{S}^2)} 
       &=& \mathcal{O}\left(  k \, \left\vert {\bf P_0} \right\vert^{2}  \, \left\vert \mathbb{T}_{1}^{-1} \right\vert \, \left( d^{\frac{12}{7}}  \, \left\Vert H^{\mathring{\mu_{r}}} \right\Vert^{2}_{\mathbb{L}^{\infty}(\Omega)} \,  + d^{2\alpha}  \, \left[ H^{\mathring{\mu_{r}}} \right]^{2}_{C^{0,\alpha}(\overline{\Omega})} \,   \right)^{\frac{1}{2}} \right) \\ &+&  \mathcal{O}\left( a^{\frac{h}{3}}\right)   +  \mathcal{O}\left(k \, d^{\alpha} \, \left\vert {\bf P_0} \right\vert \, \left\vert \mathbb{T}_{1}^{-1} \right\vert \, \left[ H^{\mathring{\mu_{r}}} \right]_{C^{0,\alpha}(\overline{\Omega})} 
  \right). 
\end{eqnarray*}
This proves $(\ref{Eq113})$.
\section{Inversion of (\ref{Re16}) and the  Magnetic field regularity.}\label{PositiveOperator}
We start by recalling, from $(\ref{Re16})$, the following L.S.E, 
\begin{equation*}
    \textbf{F} - \, \pm  \, \xi  \, \left[ - \, \nabla {\bf M}^{k}_{\Omega}\left( {\bf P_0} \cdot \mathcal{A} \cdot \textbf{F} \right) \, + \, k^{2} \, {\bf N}^{k}_{\Omega}\left( {\bf P_0} \cdot \mathcal{A} \cdot \textbf{F} \right) \right] \, = \, i \, k \, H^{Inc},
\end{equation*}
which, by using the relation given by $(\ref{H=F})$, becomes 
\begin{equation}\label{Eq1901}
	H^{\mathring{\mu}_r} \, - \, \pm \, \xi  \, \left[ - \, \nabla {\bf M}^{k}\left( {\bf P_0} \cdot \mathcal{A} \cdot H^{\mathring{\mu}_r}\right) \, + \, k^{2} \, {\bf N^{k}} \left( {\bf P_0} \cdot \mathcal{A} \cdot H^{\mathring{\mu}_r}\right) \right] \, = \, i\, k \, H^{Inc}.
\end{equation}
Now, thanks to the Helmholtz decomposition given by $(\ref{hel-decomp})$, we project $(\ref{Eq1901})$ into the sub-spaces $\mathbb{H}_0(\div=0), \mathbb{H}_0(Curl=0)$ and $\nabla \mathcal{H}armonic$. 
\begin{enumerate}
    \item[]
    \item Projection onto $\mathbb{H}_0(\div=0)$.
    \begin{eqnarray*}
     \langle	H^{\mathring{\mu}_r}, \overline{ \overset{1}{\mathbb{P}}\left( H^{\mathring{\mu}_r} \right)}\rangle_{\mathbb{L}^{2}(\Omega)}  \, &+& \, \pm \, \xi  \, \langle \nabla {\bf M}^{k}\left( {\bf P_0} \cdot \mathcal{A} \cdot H^{\mathring{\mu}_r}\right), \overline{ \overset{1}{\mathbb{P}}\left( H^{\mathring{\mu}_r} \right)}\rangle_{\mathbb{L}^{2}(\Omega)} \\
      &-& \, \pm \, \xi \, k^{2}  \, \langle  {\bf N^{k}} \left( {\bf P_0} \cdot \mathcal{A} \cdot H^{\mathring{\mu}_r}\right), \overline{ \overset{1}{\mathbb{P}}\left( H^{\mathring{\mu}_r} \right)}\rangle_{\mathbb{L}^{2}(\Omega)} \, = \, i\, k \, \langle H^{Inc}, \overline{\overset{1}{\mathbb{P}}\left( H^{\mathring{\mu}_r} \right)}\rangle_{\mathbb{L}^{2}(\Omega)}.
\end{eqnarray*}
Taking the adjoint of the operator $\nabla {\bf M}^{k}$ and using the fact that $\nabla {\bf M}^{k}{|_{\mathbb{H}_0(\div=0)}} \, = \, 0$, we come up with the following equation 
    \begin{equation*}
     \left\Vert \overset{1}{\mathbb{P}}\left( H^{\mathring{\mu}_r} \right) \right\Vert^{2}_{\mathbb{L}^{2}(\Omega)} \, 
       -  \, \pm \, \xi \, k^{2}  \, \langle  {\bf N^{k}} \left( {\bf P_0} \cdot \mathcal{A} \cdot H^{\mathring{\mu}_r}\right), \overline{ \overset{1}{\mathbb{P}}\left( H^{\mathring{\mu}_r} \right)}\rangle_{\mathbb{L}^{2}(\Omega)} \, = \, i\, k \, \langle H^{Inc}, \overline{\overset{1}{\mathbb{P}}\left( H^{\mathring{\mu}_r} \right)}\rangle_{\mathbb{L}^{2}(\Omega)}.
    \end{equation*}
    In addition, we have 
    \begin{equation*}
        \left\vert \pm \, \xi \, k^{2}  \, \langle  {\bf N^{k}} \left( {\bf P_0} \cdot \mathcal{A} \cdot H^{\mathring{\mu}_r}\right), \overline{ \overset{1}{\mathbb{P}}\left( H^{\mathring{\mu}_r} \right)}\rangle_{\mathbb{L}^{2}(\Omega)} \right\vert \, \leq \,  \xi  \, \sqrt{3} \, k^{2}  \, \left\Vert {\bf N^{k}} \right\Vert_{\mathcal{L}\left( \mathbb{L}^{2}(\Omega); \mathbb{L}^{2}(\Omega)\right)} \, \left\Vert {\bf P_0} \cdot \mathcal{A} \right\Vert_{\mathbb{\ell}^{2}} \,  \left\Vert H^{\mathring{\mu}_r} \right\Vert_{\mathbb{L}^{2}(\Omega)} \; \left\Vert \overset{1}{\mathbb{P}}\left( H^{\mathring{\mu}_r} \right)\right\Vert_{\mathbb{L}^{2}(\Omega)}. 
    \end{equation*}
    Based on the expression of ${\bf N^{k}}$ given in \cite[Equation (2.8)]{CGS} and the estimation of $\left\Vert {\bf N} \right\Vert_{\mathcal{L}\left( \mathbb{L}^{2}(\Omega); \mathbb{L}^{2}(\Omega)\right)}$ derived in \cite[page 23]{cao2023all}, we get 
    \begin{equation*}
        \left\Vert {\bf N^{k}} \right\Vert_{\mathcal{L}\left( \mathbb{L}^{2}(\Omega); \mathbb{L}^{2}(\Omega)\right)} \, \lesssim \, \frac{\left\vert \Omega \right\vert}{\pi} \; c_{1}(k),  
    \end{equation*}
    where $c_{1}(\cdot)$ is the frequency function given by 
    \begin{equation*}
        c_{1}(k) \, := \, 1 \, + \, \frac{k}{4} \, + \, \frac{1}{4 \, \diam(\Omega)} \, \sum_{n \geq 2} \frac{\left(k \, \diam(\Omega) \right)^{n}}{n!}. 
    \end{equation*}
    Then, 
    \begin{equation*}
     \left\Vert \overset{1}{\mathbb{P}}\left( H^{\mathring{\mu}_r} \right) \right\Vert_{\mathbb{L}^{2}(\Omega)} \, \left( \left\Vert \overset{1}{\mathbb{P}}\left( H^{\mathring{\mu}_r} \right)\right\Vert_{\mathbb{L}^{2}(\Omega)} \,
       -  \,  \xi  \, k^{2}  \, \sqrt{3} \, \frac{\left\vert \Omega \right\vert}{\pi} \, c_{1}(k) \, \left\Vert {\bf P_0} \cdot \mathcal{A} \right\Vert_{\ell^{2}} \,  \left\Vert H^{\mathring{\mu}_r} \right\Vert_{\mathbb{L}^{2}(\Omega)}  \right) \, \leq \, i\, k \, \langle H^{Inc}, \overline{\overset{1}{\mathbb{P}}\left( H^{\mathring{\mu}_r} \right)}\rangle_{\mathbb{L}^{2}(\Omega)}.
    \end{equation*}
And, under the condition 
\begin{equation}\label{IneqI}
        \xi  \, k^{2}  \, \sqrt{3} \, \frac{\left\vert \Omega \right\vert}{\pi} \, c_{1}(k) \, \left\Vert {\bf P_0} \cdot \mathcal{A} \right\Vert_{\ell^{2}} \,  \left\Vert H^{\mathring{\mu}_r} \right\Vert_{\mathbb{L}^{2}(\Omega)} \, < \, \left\Vert \overset{1}{\mathbb{P}}\left( H^{\mathring{\mu}_r} \right)\right\Vert_{\mathbb{L}^{2}(\Omega)}, 
\end{equation}
we derive the positivity of the operator $I_{3} \, - \, \pm \, \xi  \, \left[ - \, \nabla {\bf M}^{k} \circ {\bf P_0} \cdot \mathcal{A}  \, + \, k^{2} \, {\bf N^{k}} \circ {\bf P_0} \cdot \mathcal{A}  \right]$, appearing on the L.H.S of $(\ref{Eq1901})$, on the subspace $\mathbb{H}_0(\div=0)$. \\
    \item[]
    \item Projection onto $\mathbb{H}_0(Curl=0)$.
    \begin{eqnarray*}
     \langle	H^{\mathring{\mu}_r}, \overline{ \overset{2}{\mathbb{P}}\left( H^{\mathring{\mu}_r} \right)}\rangle_{\mathbb{L}^{2}(\Omega)}  \, &+& \, \pm \, \xi  \, \langle \nabla {\bf M}^{k}\left( {\bf P_0} \cdot \mathcal{A} \cdot H^{\mathring{\mu}_r}\right), \overline{ \overset{2}{\mathbb{P}}\left( H^{\mathring{\mu}_r} \right)}\rangle_{\mathbb{L}^{2}(\Omega)} \\
      &-& \, \pm \, \xi \, k^{2}  \, \langle  {\bf N^{k}} \left( {\bf P_0} \cdot \mathcal{A} \cdot H^{\mathring{\mu}_r}\right), \overline{ \overset{2}{\mathbb{P}}\left( H^{\mathring{\mu}_r} \right)}\rangle_{\mathbb{L}^{2}(\Omega)} \, = \, i\, k \, \langle H^{Inc}, \overline{\overset{2}{\mathbb{P}}\left( H^{\mathring{\mu}_r} \right)}\rangle_{\mathbb{L}^{2}(\Omega)}.
\end{eqnarray*}
Since $H^{Inc}$ is a divergence free vector field, the R.H.S of the previous equation will be vanishing. In addition, on $\mathbb{H}_0(Curl=0)$, we have 
\begin{equation*}
    \nabla {\bf M}^{k} \, = \, k^{2} \, {\bf N}^{k} \, + \, I_{3}. 
\end{equation*}
Thus, 
 \begin{equation}\label{MAB}
     \left\Vert \overset{2}{\mathbb{P}}\left( H^{\mathring{\mu}_r} \right) \right\Vert^{2}_{\mathbb{L}^{2}(\Omega)}   \, \pm \, \xi  \, \langle {\bf P_0} \cdot \mathcal{A} \cdot H^{\mathring{\mu}_r}, \overline{ \overset{2}{\mathbb{P}}\left( H^{\mathring{\mu}_r} \right)}\rangle_{\mathbb{L}^{2}(\Omega)} \, =  \, 0.
\end{equation}
This implies the vanishing character of the operator $I_{3} \, - \, \pm \, \xi  \, \left[ - \, \nabla {\bf M}^{k} \circ {\bf T}^{\mathring\mu_{r}}  \, + \, k^{2} \, {\bf N^{k}} \circ  {\bf T}^{\mathring\mu_{r}}  \right]$, appearing on the L.H.S of $(\ref{Eq1901})$, on the subspace $\mathbb{H}_0(Curl=0)$. 
Furthermore, $(\ref{MAB})$ implies, 
\begin{equation}\label{IneqII}
     \xi \, \sqrt{3} \, \left\Vert {\bf P_0} \cdot \mathcal{A} \right\Vert_{\ell^{2}} \; \left\Vert  H^{\mathring{\mu}_r}  \right\Vert_{\mathbb{L}^{2}(\Omega)} \; < \; \left\Vert \overset{2}{\mathbb{P}}\left( H^{\mathring{\mu}_r} \right) \right\Vert_{\mathbb{L}^{2}(\Omega)}. 
\end{equation}
    \item[]
    \item Projection onto $\nabla \mathcal{H}armonic$. 
     \begin{eqnarray}\label{ZDM}
     \nonumber
     \langle	H^{\mathring{\mu}_r}, \overline{ \overset{3}{\mathbb{P}}\left( H^{\mathring{\mu}_r} \right)}\rangle_{\mathbb{L}^{2}(\Omega)}   \, & \pm & \, \xi  \, \langle \nabla {\bf M}\left( {\bf P_0} \cdot \mathcal{A} \cdot H^{\mathring{\mu}_r}\right), \overline{ \overset{3}{\mathbb{P}}\left( H^{\mathring{\mu}_r} \right)}\rangle_{\mathbb{L}^{2}(\Omega)} \\ & \pm &  \xi  \, \langle {\bf K}\left( {\bf P_0} \cdot \mathcal{A} \cdot H^{\mathring{\mu}_r}\right), \overline{ \overset{3}{\mathbb{P}}\left( H^{\mathring{\mu}_r} \right)}\rangle_{\mathbb{L}^{2}(\Omega)}  \, = \, i\, k \, \langle H^{Inc}, \overline{\overset{3}{\mathbb{P}}\left( H^{\mathring{\mu}_r} \right)}\rangle_{\mathbb{L}^{2}(\Omega)}, 
\end{eqnarray}
where ${\bf K}$ is the operator defined by
\begin{equation*}
    {\bf K} \, := \, \left( \nabla {\bf M}^{k} \, - \, \nabla {\bf M} \right) \, - \, k^{2} \, {\bf N}^{k}.
\end{equation*}
For the second and the third terms on the L.H.S of $(\ref{ZDM})$, we have 
\begin{eqnarray}\label{Eq1154}
\nonumber
   L_{2,3} & := & \, \pm  \, \xi  \, \langle \nabla {\bf M}\left( {\bf P_0} \cdot \mathcal{A} \cdot H^{\mathring{\mu}_r}\right), \overline{ \overset{3}{\mathbb{P}}\left( H^{\mathring{\mu}_r} \right)}\rangle_{\mathbb{L}^{2}(\Omega)} \, \pm \,  \xi  \, \langle {\bf K}\left( {\bf P_0} \cdot \mathcal{A} \cdot H^{\mathring{\mu}_r}\right), \overline{ \overset{3}{\mathbb{P}}\left( H^{\mathring{\mu}_r} \right)}\rangle_{\mathbb{L}^{2}(\Omega)} \\
   \left\vert L_{2,3} \right\vert & \leq & \xi \, \sqrt{3} \, \left\Vert {\bf P_0} \cdot \mathcal{A} \right\Vert_{\ell^{2}} \, \left\Vert H^{\mathring{\mu}_r}\right\Vert_{\mathbb{L}^{2}(\Omega)} \; \left\Vert \overset{3}{\mathbb{P}}\left( H^{\mathring{\mu}_r} \right)  \right\Vert_{\mathbb{L}^{2}(\Omega)} \, \left( \left\Vert \nabla {\bf M} \right\Vert_{\mathcal{L}\left(\mathbb{L}^{2}(\Omega);\mathbb{L}^{2}(\Omega)\right)} \, + \, \left\Vert {\bf K} \right\Vert_{\mathcal{L}\left(\mathbb{L}^{2}(\Omega); \mathbb{L}^{2}(\Omega)\right)}  \right).
\end{eqnarray}
By referring to \cite[Inequality (3.22)]{cao2023all}, we have 
\begin{equation*}
    \left\Vert {\bf K} \right\Vert_{\mathcal{L}\left(\mathbb{L}^{2}(\Omega); \mathbb{L}^{2}(\Omega)\right)} \, \leq \, \frac{\left\vert \Omega \right\vert}{\pi} \, c_{2}(k),  
\end{equation*}
where $c_{2}(\cdot)$ is the frequency function given by
\begin{equation*}
    c_{2}(k) : =  k^2 \,  +  \frac{k^3}{6}  \,   + \frac{k^2}{4  \, \diam(\Omega)} \, \sum_{n \geq 2}\frac{\left( k \, \diam(\Omega) \right)^{n}}{n !} + \frac{k^3}{16} \, \sum_{n \geq 1}\frac{\left( k \, \diam(\Omega) \right)^{n}}{n !},  
\end{equation*}
and $\left\Vert \nabla {\bf M} \right\Vert_{\mathcal{L}\left(\mathbb{L}^{2}(\Omega);\mathbb{L}^{2}(\Omega)\right)} \, = \, 1$, see \cite[Lemma 5.5]{GS}. Hence, the inequality $(\ref{Eq1154})$ becomes
\begin{equation*}
   \left\vert L_{2,3} \right\vert \, \leq  \,   \xi \, \sqrt{3} \, \left\Vert {\bf P_0} \cdot \mathcal{A} \right\Vert_{\ell^{2}} \, \left\Vert H^{\mathring{\mu}_r}\right\Vert_{\mathbb{L}^{2}(\Omega)} \; \left\Vert \overset{3}{\mathbb{P}}\left( H^{\mathring{\mu}_r} \right)  \right\Vert_{\mathbb{L}^{2}(\Omega)} \, \left( 1 \, + \, \frac{\left\vert \Omega \right\vert}{\pi} \, c_{2}(k)  \right).
\end{equation*}
Then, from $(\ref{ZDM})$, we derive the coming estimation 
     \begin{equation*}
     \left\Vert \overset{3}{\mathbb{P}}\left( H^{\mathring{\mu}_r} \right)  \right\Vert_{\mathbb{L}^{2}(\Omega)} \, \left[ \left\Vert \overset{3}{\mathbb{P}}\left( H^{\mathring{\mu}_r} \right)  \right\Vert_{\mathbb{L}^{2}(\Omega)} \, - \,  
 \xi  \, \sqrt{3} \, \left\Vert {\bf P_0} \cdot \mathcal{A} \right\Vert_{\ell^{2}}  \, \left\Vert H^{\mathring{\mu}_r}\right\Vert_{\mathbb{L}^{2}(\Omega)} \;  \, \left( 1 \, + \, \frac{\left\vert \Omega \right\vert}{\pi} \, c_{2}(k)  \right) \right] \, \leq \, i\, k \, \langle H^{Inc}, \overline{\overset{3}{\mathbb{P}}\left( H^{\mathring{\mu}_r} \right)}\rangle_{\mathbb{L}^{2}(\Omega)}. 
\end{equation*}
And, under the condition,
\begin{equation}\label{IneqIII}
   \xi \, \sqrt{3} \, \left\Vert {\bf P_0} \cdot \mathcal{A} \right\Vert_{\ell^{2}} \, \left\Vert H^{\mathring{\mu}_r}\right\Vert_{\mathbb{L}^{2}(\Omega)} \;  \, \left( 1 \, + \, \frac{\left\vert \Omega \right\vert}{\pi} \, c_{2}(k)  \right) \, < \,   \left\Vert \overset{3}{\mathbb{P}}\left( H^{\mathring{\mu}_r} \right)  \right\Vert_{\mathbb{L}^{2}(\Omega)}, 
\end{equation}
we derive the positivity of the operator $I_{3} \, - \, \pm \, \xi  \, \left[ - \, \nabla {\bf M}^{k} \circ {\bf P_0} \cdot \mathcal{A}  \, + \, k^{2} \, {\bf N^{k}} \circ  {\bf P_0} \cdot \mathcal{A}  \right]$, appearing on the L.H.S of $(\ref{Eq1901})$, on the subspace $\nabla \mathcal{H}armonic$.
\end{enumerate}
\medskip
Now, by taking the square on each side of the inequalities $(\ref{IneqI}), (\ref{IneqII})$ and $(\ref{IneqIII})$ and summing up the three obtained inequalities, we come up with the following sufficient condition ensuring the positivity of the operator $I_{3} \, - \, \pm \, \xi  \, \left[ - \, \nabla {\bf M}^{k} \circ {\bf P_0} \cdot \mathcal{A}  \, + \, k^{2} \, {\bf N^{k}} \circ  {\bf P_0} \cdot \mathcal{A}  \right]$, appearing on the L.H.S of $(\ref{Eq1901})$, on $\mathbb{L}^{2}(\Omega)$, 
\begin{equation*}
     \xi \, \sqrt{3} \, \left\Vert {\bf P_0} \cdot \mathcal{A} \right\Vert_{\ell^{2}} \, \sqrt{ \left( k^{2} \, \frac{\left\vert \Omega \right\vert}{\pi} \, c_{1}(k)  \right)^{2} \, + \, 1 \, + \, \left( 1 \, + \, \frac{\left\vert \Omega \right\vert}{\pi} \, c_{2}(k)  \right)^{2}}   \, < \, 1, 
\end{equation*}
which can be reduced, by taking the dominant part of $c_{1}(\cdot)$ and $c_{2}(\cdot)$, to 
\begin{equation*}
       \left\Vert {\bf P_0} \right\Vert_{\ell^{2}} \;  \left\Vert \mathcal{A} \right\Vert_{\ell^{2}} \, < \, \frac{\pi}{3 \, \xi \, \sqrt{6} \, \left( \pi \, + \, k \, \left\vert \Omega \right\vert \right) } . 
\end{equation*}
\section{Appendix}
\subsection{Computation of $\mathcal{A}$ for dimers}\label{Appendix}
We recall, from $(\ref{SKM})$, the expression of  $\mathcal{A}$ and we assume that we have two nano-particles located on $z_{0_{1}}$ and $z_{0_{2}}$. 
Thus, 
\begin{equation*}
        \mathcal{A} \, = \,  \begin{pmatrix}
            I_{3} & I_{3} 
        \end{pmatrix} \cdot \mathbb{T}_1^{-1} \cdot \begin{pmatrix}
            I_{3} \\
            I_{3}
        \end{pmatrix},         
\end{equation*}
where $\mathbb{T}_1$ is the tensor given by $(\ref{DefT1})$. Then, 
\begin{equation*}
         \mathcal{A} \,  =  \,  \begin{pmatrix}
            I_{3} & I_{3} 
        \end{pmatrix} 
         \cdot \left( \mathbb{T}_0 \, - \, \pm \,\xi \,  \frac{1}{\left\vert \Omega_{0} \right\vert} \, \int_{\Omega_{0}}  \, \begin{pmatrix}
        \nabla {\bf M}_{\Omega_{0}}\left( I_{3} \right)(x) \cdot {\bf P_0} &   \nabla {\bf M}_{\Omega_{0}}\left( I_{3} \right)(x) \cdot {\bf P_0} \\
        \nabla {\bf M}_{\Omega_{0}}\left( I_{3} \right)(x) \cdot {\bf P_0} &   \nabla {\bf M}_{\Omega_{0}}\left( I_{3} \right)(x) \cdot {\bf P_0}
    \end{pmatrix} \, dx \right)^{-1} \cdot 
            \begin{pmatrix}
            I_{3} \\
            I_{3} 
            \end{pmatrix}. 
\end{equation*}
For simplicity, we assume that the used nano-particles are of spherical shape. This implies, 
\begin{equation}\label{SSOMA}
    {\bf P_0} \,=\, \frac{12}{\pi^{3}} \, I_{3},  
\end{equation}
see \cite[Equation (3.4)]{cao2023all}. Hence, 
\begin{equation*}
         \mathcal{A} \,  =  \,  \begin{pmatrix}
            I_{3} & I_{3} 
        \end{pmatrix} 
         \cdot \left( \mathbb{T}_0 \, - \, \pm \,\xi \, \frac{12}{\pi^{3}} \, \frac{1}{\left\vert \Omega_{0} \right\vert} \, \int_{\Omega_{0}}  \, \begin{pmatrix}
        \nabla {\bf M}_{\Omega_{0}}\left( I_{3} \right)(x)  &   \nabla {\bf M}_{\Omega_{0}}\left( I_{3} \right)(x)  \\
        \nabla {\bf M}_{\Omega_{0}}\left( I_{3} \right)(x)  &   \nabla {\bf M}_{\Omega_{0}}\left( I_{3} \right)(x) 
    \end{pmatrix} \, dx \right)^{-1} \cdot 
            \begin{pmatrix}
            I_{3} \\
            I_{3} 
            \end{pmatrix}. 
\end{equation*}
Furthermore, the formula below can be obtained by referring to \textbf{Subsection \ref{AddSS}.} 
\begin{equation}\label{SSOMAI}
    \nabla {\bf M}_{\Omega_{0}}\left( I_{3} \right)(x) \, = \, \frac{1}{3} \, I_{3}, \quad \text{for} \quad x \in \Omega_{0}.   
\end{equation}
This implies, 
\begin{eqnarray*}
                 \mathcal{A} \, & \overset{}{=} & \,  \begin{pmatrix}
            I_{3} & I_{3} 
        \end{pmatrix} 
         \cdot \left( \mathbb{T}_0 \, - \, \pm \, \frac{4 \, \xi}{\pi^{3}}  \, \begin{pmatrix}
         I_{3} &  I_{3} \\
         I_{3} &  I_{3}
    \end{pmatrix} \right)^{-1} \cdot 
            \begin{pmatrix}
            I_{3} \\
            I_{3} 
            \end{pmatrix} \\    
        & \overset{(\ref{def-T0-matrix})}{=}  & \ \begin{pmatrix}
            I_{3} & I_{3}
        \end{pmatrix} \cdot \left( \mathbb{I} \, - \, \pm \, \xi \, d^{3} \,  \D_0({\bf P_0}{\mathbb{I}}) - \, \pm \, \frac{4 \, \xi}{\pi^3}  \, \begin{pmatrix}
         I_{3} &   I_{3} \\
         I_{3} &   I_{3}
    \end{pmatrix} \right)^{-1} \cdot            \begin{pmatrix}
            I_{3} \\
            I_{3} 
        \end{pmatrix}. 
\end{eqnarray*}
 Hence, by using $(\ref{SSOMA})$, 
\begin{equation*}
        \mathcal{A} \, = \,    \begin{pmatrix}
            I_{3} & I_{3}
        \end{pmatrix}  \cdot \left( \begin{pmatrix}
            I_{3} & 0 \\
            0 & I_{3}
        \end{pmatrix} \, - \, \pm \, \xi \, d^{3} \, \frac{12}{\pi^{3}} \,  \begin{pmatrix}
            0 & \Upsilon_{0} \\
            \Upsilon_{0} & 0
        \end{pmatrix}  \, - \, \pm \, \xi \, \frac{4}{\pi^{3}} \, \begin{pmatrix}
            I_{3} & I_{3} \\
            I_{3} & I_{3}
        \end{pmatrix} \right)^{-1} \cdot 
             \begin{pmatrix}
            I_{3} \\
            I_{3} 
        \end{pmatrix},
\end{equation*}
where $\Upsilon_{0} := \Upsilon_{0}(z_{0_{1}}, z_{0_{2}})$. Then, 
\begin{equation*}\label{ADP}
               \mathcal{A} \,  =  \dfrac{\left(\pi^{3}  \mp 4 \, \xi \right)}{\pi^{3}} \,   \begin{pmatrix}
            I_{3} & I_{3}
        \end{pmatrix}  \cdot \begin{pmatrix}   \, I_{3} &  \dfrac{\mp \, 4 \, \xi}{\left( \pi^{3} \mp 4 \, \xi \right)}  \, \left[I_{3} + 3 \, d^{3} \, \Upsilon_{0} \right] \\ 
        & \\  \dfrac{\mp \, 4 \, \xi}{\left( \pi^{3} \mp 4 \, \xi \right)}  \, \left[I_{3} + 3 \, d^{3} \, \Upsilon_{0} \right] &
             I_{3}  \end{pmatrix}^{-1} \cdot
             \begin{pmatrix}
            I_{3} \\
            I_{3}
        \end{pmatrix}.
\end{equation*}
And, by using Block matrix inversion\footnote{\url{https://en.wikipedia.org/wiki/Block_matrix}}, we obtain
\begin{equation}\label{Betha}
\mathcal{A} \,  =  \dfrac{\left(\pi^{3}  \mp 4 \, \xi \right)}{\pi^{3}} \,   \begin{pmatrix}
            I_{3} & I_{3}
        \end{pmatrix}  \cdot \begin{pmatrix}   \, \beth_{11} &  \beth_{12} \\ 
        & \\  \beth_{21} &
             \beth_{22}  \end{pmatrix} \cdot
             \begin{pmatrix}
            I_{3} \\
            I_{3}
        \end{pmatrix} 
        = \dfrac{\left(\pi^{3}  \mp 4 \, \xi \right)}{\pi^{3}} \,    \left(   \, \beth_{11} + \beth_{12} +  \beth_{21} +
             \beth_{22}  \right), 
\end{equation}
where 
\begin{equation*}
\begin{matrix}
\qquad \;\;\; \beth_{11}  :=  I_{3} \, + \, \mathcal{D} \cdot \left(I_{3} \, - \, \mathcal{D}^{2} \right)^{-1} \cdot \mathcal{D} & & \qquad \; \beth_{12}  :=  - \, \mathcal{D} \cdot \left(I_{3} \, - \, \mathcal{D}^{2} \right)^{-1} \\
& & \\ 
\beth_{21}  :=  - \, \left(I_{3} \, - \, \mathcal{D}^{2} \right)^{-1} \cdot \mathcal{D} & &     \beth_{22}  :=   \left(I_{3} \, - \, \mathcal{D}^{2} \right)^{-1}
\end{matrix},
\end{equation*}
with 
\begin{equation}\label{LRD}
    \mathcal{D} \, := \, \dfrac{\mp \, 4 \, \xi}{\left( \pi^{3} \mp 4 \, \xi \right)}  \, \left[I_{3} + 3 \, d^{3} \, \Upsilon_{0} \right]. 
\end{equation}
By going back to $(\ref{Betha})$ and using the expression of $\beth_{11}, \beth_{12}, \beth_{21}$ and $\beth_{22}$, we obtain 
\begin{equation}\label{Eq2203}
\mathcal{A} = \dfrac{2 \, \left(\pi^{3}  \mp 4 \, \xi \right)}{\pi^{3}} \,    \left(  I_{3} \, + \,
             \mathcal{D}  \right)^{-1} 
             \overset{(\ref{LRD})}{=} \dfrac{2 \, \left(\pi^{3}  \mp 4 \, \xi \right)^{2}}{\pi^{3} \, \left(\pi^{3}  \mp 8 \, \xi \right) } \,    \left(  I_{3} \, + \, \dfrac{\mp \, 12 \, \xi  \, d^{3}}{\left( \pi^{3} \mp 8 \, \xi \right)}  \, \Upsilon_{0} 
               \right)^{-1}. 
\end{equation}
We recall that, 
\begin{eqnarray}\label{DimerUpsilon}
\nonumber
    \Upsilon_{0} := \Upsilon_{0}\left(z_{0_{1}}, z_{0_{2}}\right) \, &=& \, \frac{-1}{4 \, \pi} \, \frac{1}{\left\vert z_{0_{1}} \, - \, z_{0_{2}} \right\vert^{3}} \, \left(  \, I_{3} \, -  \, \frac{3}{\left\vert z_{0_{1}} \, - \, z_{0_{2}} \right\vert^{2}} \left( z_{0_{1}} \, - \, z_{0_{2}} \right) \otimes \left( z_{0_{1}} \, - \, z_{0_{2}} \right) \right) \\
     &=& \, \frac{-1}{4 \, \pi} \, \frac{1}{d^{3}} \, \left(  \, I_{3} \, -  \, \frac{3}{d^{2}} \left( z_{0_{1}} \, - \, z_{0_{2}} \right) \otimes \left( z_{0_{1}} \, - \, z_{0_{2}} \right) \right).
\end{eqnarray}
Then, by gathering $(\ref{Eq2203})$ with $(\ref{DimerUpsilon})$, we end up with the following equation 
\begin{eqnarray}\label{ExpressionA}
\nonumber
\mathcal{A} \, 
             & = & \, \dfrac{2 \, \left(\pi^{3}  \mp 4 \, \xi \right)^{2}}{\pi^{2} \, \left[ \pi \, \left(\pi^{3}  \mp 8 \, \xi \right) \pm 3 \, \xi \right]} \,    \left(  I_{3} \, - \, \dfrac{\pm \, 9 \, \xi}{d^{2} \, \left[ \pi 
            \left( \pi^{3} \mp 8 \, \xi \right) \pm 3 \, \xi \right]}  \, \left( z_{0_{1}} \, - \, z_{0_{2}} \right) \otimes \left( z_{0_{1}} \, - \, z_{0_{2}} \right) 
               \right)^{-1} \\
           & = & \, \dfrac{2 \, \left(\pi^{3}  \mp 4 \, \xi \right)^{2}}{\pi^{2} \, \left[ \pi \, \left(\pi^{3}  \mp 8 \, \xi \right) \pm 3 \, \xi \right]} \, \left( I_{3} \, - \,      \dfrac{\pm \, 9 \, \xi}{d^{2} \, \left[ \pi 
            \left( \pi^{3} \mp 8 \, \xi \right) - \pm 6 \, \xi \right]}  \, \left( z_{0_{1}} \, - \, z_{0_{2}} \right) \otimes \left( z_{0_{1}} \, - \, z_{0_{2}} \right) 
               \right).
\end{eqnarray}
\subsection{Justification of $(\ref{SSOMAI})$}\label{AddSS}
We recall, from $(\ref{SSOMAI})$, and we set the following expression 
\begin{equation*}
    \mathcal{W}(x) \, := \, \nabla {\bf M}_{\Omega_{0}}(I_{3})(x), \quad x \in \Omega_{0}.
\end{equation*}
Without losing generalities, we assume that $\Omega_{0}$ is the unit cube. To compute the matrix $\mathcal{W}(\cdot)$, we start by computing its diagonal components. To do this, by taking the trace operator on the both sides of the previous equation we obtain  
\begin{eqnarray*}
    Trace(\mathcal{W}(x)) \, = \, Trace \left( \nabla {\bf M}_{\Omega_{0}}(I_{3})(x) \right) 
    := \, Trace \left(  \underset{x}{\nabla} \int_{\Omega_{0}} \underset{y}{\nabla} \Phi_{0}(x,y) \, dy \right) = \, Trace \left(- \,  \underset{x}{\nabla} \underset{x}{\nabla} \int_{\Omega_{0}}  \Phi_{0}(x,y) \, dy \right).
\end{eqnarray*}
As $Trace\left( \underset{x}{\nabla} \, \underset{x}{\nabla} \cdots \right) \, = \, \underset{x}{\Delta} \left( \cdots \right)$, we obtain 
\begin{equation*}
    Trace(\mathcal{W}(x)) \, = \,- \,   \underset{x}{\Delta} \int_{\Omega_{0}}  \Phi_{0}(x,y) \, dy,
\end{equation*}
which, by taking the the Laplace operator over the integral sign, using the fact that $\Delta \Phi_{0} \, = \, - \, \delta_{0}$ and $x \in \Omega_{0}$, gives us
\begin{equation*}
    Trace(\mathcal{W}(x)) \,  = \,1, \quad x \in \Omega_{0}.
\end{equation*}
Moreover, since $\Omega_{0}$ is symmetric we deduce that the diagonal components are equals to $\dfrac{1}{3}$, i.e.
\begin{equation}\label{Eq525}
    \mathcal{W}(x) \,  =  \, \begin{pmatrix}
        \dfrac{1}{3} & \cdot & \cdot \\
        \cdot & \dfrac{1}{3} & \cdot \\
        \cdot & \cdot & \dfrac{1}{3}
    \end{pmatrix}
    , \quad x \in \Omega_{0}.
\end{equation}
Now, for the off-diagonal components of $\mathcal{W}(\cdot)$, by referring to \cite[Formula (13)]{Samokhin}, we obtain 
\begin{eqnarray*}
    \kappa_{\epsilon,i,j} \, &:=& \, \left(\frac{\partial }{ \partial x_{i}} \, \frac{\partial }{ \partial x_{j}} \int_{\Omega_{0}} \frac{1}{4 \pi} \frac{1}{\left\vert x - y \right\vert} \, dy \right) \\
    &=& \, \lim_{\epsilon \rightarrow 0} \, \left[  \left(\frac{\partial }{ \partial x_{i}} \, \frac{\partial }{ \partial x_{j}} \int_{\Omega_{0} \setminus \Omega(x ,  \epsilon) } \frac{1}{4 \pi} \frac{1}{\left\vert x - y \right\vert^{3}} \, \left[ \frac{3 \, (x_{i} - y_{i}) (x_{j} - y_{j})}{\left\vert x - y \right\vert^{2}} \, - \, \delta_{i,j} \right] \, dy \right) -  \frac{\delta_{i,j}}{3} \right],
\end{eqnarray*}
where $\Omega(x,\epsilon)$ is a small cube of center $x$ and radius $\epsilon$, with $\epsilon \ll 1$. The above expression, for $i \neq j$, indicates that 
\begin{equation*}
    \kappa_{\epsilon,i,j} \, = \, \frac{3}{4 \pi} \, \lim_{\epsilon \rightarrow 0} \,  \left(\frac{\partial }{ \partial x_{i}} \, \frac{\partial }{ \partial x_{j}} \int_{\Omega_{0} \setminus \Omega(x ,  \epsilon) }  \frac{(x_{i} - y_{i}) (x_{j} - y_{j})}{\left\vert x - y \right\vert^{5}}  \, dy \right).
\end{equation*}
The domain of integration is explicitly defined, resulting in
\begin{equation*}
    \int_{\Omega_{0} \setminus \Omega(x ,  \epsilon) }  \frac{(x_{i} - y_{i}) (x_{j} - y_{j})}{\left\vert x - y \right\vert^{5}}  \, dy \, = \, \overset{3}{\underset{k=1}\times} \left( \int_{-1}^{-\epsilon + x_{k}} \, + \, \int_{\epsilon + x_{k}}^{1} \right) \frac{(x_{i} - y_{i}) (x_{j} - y_{j})}{\left\vert x - y \right\vert^{5}}  \, dy_{1} \, dy_{2} \, dy_{3}.
\end{equation*}
Hence, 
\begin{equation*}
    \kappa_{\epsilon,i,j} \, = \, \frac{3}{4 \pi} \, \lim_{\epsilon \rightarrow 0}  \,  \left(\frac{\partial }{ \partial x_{i}} \, \frac{\partial }{ \partial x_{j}} \overset{3}{\underset{k=1}\times} \left( \int_{-1}^{-\epsilon + x_{k}} \, + \, \int_{\epsilon + x_{k}}^{1} \right) \frac{(x_{i} - y_{i}) (x_{j} - y_{j})}{\left\vert x - y \right\vert^{5}}  \, dy_{1} \, dy_{2} \, dy_{3} \right).
\end{equation*}
Without loss of generalities, we assume that $i=1$ and $j=2$. Then, let's compute 
\begin{eqnarray*}
    \lim_{\epsilon \rightarrow 0} \mathcal{J}_{1}(x, y_{2}, y_{3}, \epsilon) \, &:=& \, \lim_{\epsilon \rightarrow 0}  \left( \int_{-1}^{-\epsilon + x_{1}} \, + \, \int_{\epsilon + x_{1}}^{1} \right) \frac{(x_{1} - y_{1}) (x_{2} - y_{2})}{\left\vert x - y \right\vert^{5}}  \, dy_{1} \\
    &=& \lim_{\epsilon \rightarrow 0}  \, (x_{2} - y_{2}) \, \left( \int_{-1}^{-\epsilon + x_{1}} \, + \, \int_{\epsilon + x_{1}}^{1} \right) \frac{(x_{1} - y_{1})}{\left\vert x - y \right\vert^{5}}  \, dy_{1} \\
    &=& \lim_{\epsilon \rightarrow 0}  \, \frac{(x_{2} - y_{2})}{3} \, \left( \frac{1}{\left\vert x - y \right\vert^{3}} \Bigg|_{-1 = y_{1}}^{-\epsilon + x_{1} = y_{1}}  \, + \,  \frac{1}{\left\vert x - y \right\vert^{3}} \Bigg|_{\epsilon + x_{1} = y_{1}}^{1 = y_{1}}  \right) = 0.
\end{eqnarray*}
This implies, 
\begin{equation*}
    \kappa_{\epsilon,1,2} \, = \, \frac{3}{4 \pi} \, \lim_{\epsilon \rightarrow 0}  \,  \left(\frac{\partial }{ \partial x_{i}} \, \frac{\partial }{ \partial x_{j}} \overset{3}{\underset{k=2}\times} \left( \int_{-1}^{-\epsilon + x_{k}} \, + \, \int_{\epsilon + x_{k}}^{1} \right) \mathcal{J}_{1}(x, y_{2}, y_{3}, \epsilon)   \, dy_{2} \, dy_{3} \right) = 0.
\end{equation*}
Hence, $\mathcal{W}(\cdot)$ is off-diagonal vanishing. Together with $(\ref{Eq525})$, we derive that 
\begin{equation*}
    \mathcal{W}(x) \, = \, \frac{1}{3} \; I_{3}, \quad x \in \Omega_{0}. 
\end{equation*}
This proves $(\ref{SSOMAI})$. 

\end{document}